\documentclass[12pt]{article}
\usepackage[a4paper, total={6in, 9in}]{geometry}
\usepackage[utf8]{inputenc}
\usepackage{amssymb,amsmath}
\usepackage{mathtools}
\usepackage{bm}
\usepackage{bbm}
\usepackage{tensor}
\usepackage{xfrac}
\usepackage{amsthm}
\usepackage{enumerate}
\usepackage{xcolor}
\usepackage{diffcoeff}
\usepackage{array}
\usepackage{tikz}
\usetikzlibrary{cd}
\usepackage[style=alphabetic, maxalphanames=5, maxbibnames=99, url=false, giveninits=true]{biblatex}
\renewbibmacro{in:}{}
\DeclareFieldFormat{titlecase}{\MakeSentenceCase*{#1}}
\AtEveryBibitem{%
  \clearfield{doi}%
  \clearfield{issn}%
  \clearfield{isbn}%
  \clearlist{language}%
  \clearfield{month}%
  \clearfield{day}%
  \clearfield{pubstate}%
}
\usepackage{hyperref}

\theoremstyle{plain}
\newtheorem{theorem}{Theorem}[section]
\newtheorem{thmdef}[theorem]{Theorem/Definition}
\newtheorem{lemma}[theorem]{Lemma}
\newtheorem{corollary}[theorem]{Corollary}
\newtheorem{proposition}[theorem]{Proposition}
\newtheorem*{theorem*}{Theorem}
\newtheorem*{lemma*}{Lemma}
\newtheorem*{corollary*}{Corollary}
\newtheorem*{proposition*}{Proposition}

\theoremstyle{definition}
\newtheorem{definition}[theorem]{Definition}
\newtheorem*{definition*}{Definition}

\theoremstyle{remark}
\newtheorem{remark}[theorem]{Remark}

\theoremstyle{plain}
\newtheorem{introtheorem}{Theorem}

\title{$G$-typical Witt vectors with coefficients and the norm}
\author{Thomas Read\\\texttt{thomas.read@warwick.ac.uk}}
\date{}

\newcommand{\can}[1]{\underline{#1}}

\begin{document}

\maketitle

\newcommand{\din}
  {\mathrel{%
    \begin{tikzpicture}[baseline=-0.58ex]
    \begin{scope}[line cap=round, thin]
    \draw[yshift=-.25mm] (0,0) -- (-.275,0);
    \draw[yshift= .25mm] (0,0) -- (-.275,0);
    \end{scope}
    \filldraw[ rounded corners=.25pt
             , thin
             , rotate around={45:(0,0)}
             , fill=backgroundcolor
             ] (-.45mm,-.45mm) rectangle (.45mm,.45mm);
    \end{tikzpicture}%
  }}

\newcommand\Osq{\mathbin{\text{\scalebox{.84}{$\square$}}}}

\DeclarePairedDelimiter\abs{\lvert}{\rvert}%
\DeclarePairedDelimiter\norm{\lVert}{\rVert}%

\newcommand{\xrightarrowdbl}[2][]{%
  \xrightarrow[#1]{#2}\mathrel{\mkern-14mu}\rightarrow
}

\makeatletter
\let\oldabs\abs
\renewcommand\abs{\@ifstar{\oldabs}{\oldabs*}}
\let\oldnorm\norm
\renewcommand\norm{\@ifstar{\oldnorm}{\oldnorm*}}
\makeatother

\newcommand{\slashedrightarrow}{\relbar\joinrel\mapstochar\joinrel\rightarrow}

\section*{Abstract}

For a profinite group $G$ we describe an abelian group $W_G(R; M)$ of $G$-typical Witt vectors with coefficients in an $R$-module $M$ (where $R$ is a commutative ring). This simultaneously generalises the ring $W_G(R)$ of Dress and Siebeneicher and the Witt vectors with coefficients $W(R; M)$ of Dotto, Krause, Nikolaus and Patchkoria, both of which extend the usual Witt vectors of a ring. We use this new variant of Witt vectors to give a purely algebraic description of the zeroth equivariant stable homotopy groups of the Hill-Hopkins-Ravenel norm $N_{\{e\}}^G(X)$ of a connective spectrum $X$, for any finite group $G$. Our construction is reasonably analogous to the constructions of previous variants of Witt vectors, and as such is amenable to fairly explicit concrete computations.

\setcounter{tocdepth}{2}
\tableofcontents

\section{Introduction}

Witt vectors were first described in
\cite{witt_zyklische_1937} for the purpose of classifying field extensions, but
have since found wider significance in a variety of areas of mathematics. Given a commutative ring $R$, the ring of $n$-truncated, $p$-typical Witt vectors of $R$ is a commutative ring $W_{n, p}(R)$. It has underlying set $\prod_{0 \le i < n} R$
but a more complicated addition
and multiplication, defined such that for all $0 \le j < n$ the ``ghost component'' map $w_j : \prod_{0
  \le i < n} R \to R$ with formula
\begin{equation}\label{eq:intro_ghost}(a_i) \mapsto \sum_{0 \le i \le j} p^i a_i^{p^{j-i}}
\end{equation}
is a ring homomorphism.

The norm is an important construction in equivariant stable homotopy theory,
studied by Hill, Hopkins and Ravenel in their work \cite{hill_nonexistence_2016}
on the Kervaire invariant one problem. In the case we are interested in, and
using the point-set model of orthogonal $G$-spectra, the
definition is straightforward. For a finite
group $G$ and a spectrum $X$, the norm $N_{\{e\}}^G X$ is the
$G$-spectrum obtained by taking the smash product $X^{\wedge
  \abs{G}}$ of $\abs{G}$-many copies of (a cofibrant replacement of) $X$, and letting $G$ act by permuting
the factors.

In order to describe the zeroth homotopy groups of the norm, we introduce a
generalisation of Witt vectors. Given a profinite group $G$, a commutative ring
$R$ and an $R$-module $M$, we define a topological abelian group $W_{G}(R;
M)$ of $G$-typical Witt vectors with coefficients in $M$.
This group is specified by a list of defining properties
(Theorem~\ref{thm:witt_properties}) somewhat analogous to those for the usual Witt
vectors. Let $S$ be the set of open subgroups of $G$, and let $\can{S} \subseteq S$
consist of a choice of representative for each conjugacy class of open
subgroups. Then the underlying space of $W_{G}(R; M)$ is a quotient of
\[\prod_{V \in \can{S}} M^{\otimes_R G/V} \text{.}\]
The quotient and the addition operation are governed by ``ghost components''
\[w_U : \prod_{V \in \can{S}} M^{\otimes_R G/V} \to M^{\otimes_R G/U}\]
for $U \in S$, defined by a formula analogous to (\ref{eq:intro_ghost}) but
replacing powers by tensor powers and scalars by transfers.

When $T$ is a free commutative ring and $Q$ is a free $T$-module, we prove in
Proposition~\ref{prop:witt_free} that
there is a (not natural) isomorphism of topological abelian groups
\begin{equation}\label{eq:intro_witt_free}W_{G}(T; Q) \cong \prod_{V \in \can{S}} (Q^{\otimes_T G/V})_{N_G(V)} \text{,}\end{equation}
where $N_G(V)$ denotes the normaliser of $V$ in $G$.

The main aim of the paper is to prove the following (see Theorem~\ref{thm:main_theorem}).
\begin{introtheorem}\label{introthm:basic_iso}
  For a finite group $G$ and a connective spectrum $X$, we
  have an isomorphism of abelian groups
  \[\pi^G_0(N_{\{e\}}^GX) \cong W_G(\mathbb{Z}; \pi_0 X) \text{,}\]
  natural in $X$. Here $\pi_0^G$ denotes the zeroth $G$-equivariant stable homotopy group.
\end{introtheorem}

We believe that our Witt vector construction is interesting both from a topological point
of view and from a purely algebraic
point of view. Topologically, it completes a pattern of using Witt vector-like
methods to calculate the zeroth
homotopy groups of certain spectra related to topological
Hochschild homology. We describe this in Section~\ref{sec:prior_work} below. The
construction also gives us a concrete understanding of
the rich structure of the zeroth equivariant homotopy groups of the norm: for example the isomorphism
(\ref{eq:intro_witt_free}) above corresponds to the tom Dieck splitting of the norm of
a suspension spectrum. Algebraically, the
$G$-typical Witt vectors with coefficients provide a common generalisation of
several previous variants of Witt vectors, and so give insight into how these
structures are related and how they can be interpreted.
We note that if a reader just wishes to study
$W_G(R; M)$ as an algebraic construction,
they can safely read Sections~\ref{sec:preliminaries}, \ref{sec:witt_definition} and
most of \ref{sec:witt_mackey} as well as the worked example in Appendix~\ref{app:computation}
with no knowledge of homotopy theory.

\subsection{Motivation and prior work} \label{sec:prior_work}

In \cite{hesselholt_k-theory_1997}, Hesselholt and Madsen prove that the zeroth
equivariant stable homotopy
group of the topological Hochschild homology (THH) spectrum can be computed using the
$p$-typical Witt vectors of a ring. Specifically, they show that given a commutative ring
$R$ there is a ring isomorphism
\[\pi_0^{C_{p^n}} (\text{THH}(R)) \cong W_{n+1, p}(R) \text{.}\]

In \cite{dotto_witt_2022} and \cite{dotto_witt_2025}, the authors generalise
this to compute the zeroth homotopy group of topological restriction homology
$\text{TR}$ with coefficients. Given a (not necessarily commutative) ring $R$
and an $(R, R)$-bimodule $M$, they define an abelian group
$W_{n+1, p}(R; M)$
analogous to the Witt vectors, which they call the group of $n+1$-truncated
$p$-typical Witt vectors with coefficients in $M$.
For a connective ring
spectrum $E$ and a connective $(E, E)$-bimodule spectrum $X$, they give an isomorphism
of abelian groups
\begin{equation}\pi_0 (\text{TR}^{n+1}(E; X)) \cong W_{n+1, p}(\pi_0 E;
  \pi_0 X) \text{,}\label{eq:tr_witt}\end{equation}
where $\text{TR}^{n+1}(E; X)$ is the truncated topological restriction homology with
coefficients of \cite{lindenstrauss_taylor_2012} and
\cite{krause_polygonic_2023}.
This extends \cite{hesselholt_k-theory_1997} since for $R$ a commutative ring
and $HR$ the corresponding Eilenberg-MacLane spectrum, $\text{TR}^{n+1}(HR; HR) \simeq \text{THH}(R)^{C_{p^n}}$
and $W_{n+1, p}(R; R) \cong W_{n+1, p}(R)$.

In the special case where $E = \mathbb{S}$ is the sphere spectrum (so $X$
can be any connective spectrum, considered as an $(\mathbb{S},
\mathbb{S})$-bimodule) then by construction there is an equivalence
$\text{TR}^{n+1}(\mathbb{S}; X) \simeq (N_{\{e\}}^{C_{p^n}} X)^{C_{p^n}}$.
Hence the isomorphism (\ref{eq:tr_witt}) specialises to an isomorphism
\begin{equation}\label{eq:dknp_norm}\pi^{C_{p^n}}_0(N_{\{e\}}^{C_{p^n}} X) \cong W_{n+1, p}(\mathbb{Z}; \pi_0 X) \text{.}\end{equation}

The main application of the present paper is to generalise (\ref{eq:dknp_norm})
to a computation of
$\pi^G_0(N_{\{e\}}^GX)$
for $G$ any finite group, in terms of a Witt vector-like construction on $\pi_0 X$.
We take inspiration from the work of Dress and
Siebeneicher in \cite{dress_burnside_1988}, where they describe a variant of the
usual $p$-typical Witt vectors of a ring with the prime $p$ replaced by a group.
Given a commutative ring $R$ and a profinite group $G$, they build a commutative
 ring $W_G(R)$, which we will refer to as the ring of $G$-typical Witt
vectors of $R$.\footnote{There appears to be no consensus on a name
  for this construction. Other names include the Witt-Burnside ring and the ring
  of $G$-Witt vectors.} This generalises the $p$-typical Witt vectors in the sense that
$W_{C_{p^n}}(R) \cong W_{n+1, p}(R)$
(and taking $G$ to be the profinite completion of the integers recovers the
big Witt vectors $W(R)$ of \cite{atiyah_group_1969}). They define $W_G(\mathbb{Z})$ to be the Burnside
ring of $G$ when $G$ is finite, and a completed version of the Burnside ring for
$G$ infinite; they then extend this to rings other than $\mathbb{Z}$.

Our construction $W_{G}(R; M)$ gives a common generalisation of the Witt vectors with coefficients
and the $G$-typical Witt vectors.
We have $W_{C_{p^n}}(R; M) \cong W_{n+1, p}(R; M)$ and $W_G(R; R) \cong
W_G(R)$. For a connective spectrum $X$ and a finite group
$G$, Theorem~\ref{introthm:basic_iso} gives an isomorphism
\[\pi_0^G(N_{\{e\}}^G X) \cong W_G(\mathbb{Z}; \pi_0 X)\text{,}\]
completing
the pattern. More generally we can compute the $H$-equivariant zeroth homotopy
group for any subgroup $H$, and the Mackey functor structure of
the homotopy groups corresponds to our versions of the usual operators
between Witt vectors---we explain this in more detail in Section~\ref{sec:overview}.

As a sanity check, observe that we have $\pi^G_0(N_{\{e\}}^G\mathbb{S}) \cong W_G(\mathbb{Z};
\mathbb{Z}) \cong W_G(\mathbb{Z})$. Indeed this makes sense: $N_{\{e\}}^G\mathbb{S}
\simeq \mathbb{S}$ so $\pi^G_0(N_{\{e\}}^G \mathbb{S}) \cong \pi^G_0(\mathbb{S})$, and $\pi^G_0(\mathbb{S})$ and $W_G(\mathbb{Z})$ are both the
Burnside ring of $G$.

We note that we are not the first to give an algebraic description of
the Mackey functor of the equivariant homotopy groups
$\underline{\pi}_0(N_{\{e\}}^GX)$. In
\cite{ullman_symmetric_2013}, Ullman introduces a norm construction on Mackey functors. Given
a finite group $G$, a subgroup $H$ and an $H$-Mackey functor $\underline{M}$,
he defines $N^G_H(\underline{M})$ as a certain sub-Mackey
functor of the free Tambara functor on $\text{Ind}_H^G \underline{M}$; this can also be converted into a somewhat complicated
presentation via generators and relations. For any connective $H$-spectrum $X$
he shows
\[\underline{\pi}_0(N_H^G X) \cong N^G_H(\underline{\pi}_0X) \text{.}\]
Hoyer (in \cite{hoyer_two_2014})
gives a more abstract description of $N^G_H$ as a left Kan extension along a functor between Burnside categories.
Our approach provides new insight into the structure of these Mackey functors in
the $H = \{e\}$ case, and
is more amenable to explicit calculation. It also links these ideas with the
work on Witt vectors described above, which may suggest areas for future enquiry
on either side. We expect our construction for infinite profinite groups $G$ to
give a $G$-Mackey profunctor in the sense of Kaledin \cite{kaledin_mackey_2022}, and hence a connection to the quasifinitely
genuine $G$-spectra of \cite{krause_polygonic_2023}, but leave
investigation of this to future work (see Remark~\ref{rem:g_infinite}).

Other relevant work includes \cite{hill_equivariant_2019}, which gives an explicit
description of the norm of Mackey functors in the special case of cyclic $p$-groups, and
\cite{lindenstrauss_loday_2024}, which develops a version of the Loday construction
for $G$-Tambara functors.

\subsection{Overview of the results} \label{sec:overview}

In Section~\ref{sec:preliminaries} we give preliminary definitions and conventions.

In Section~\ref{sec:witt_definition} we briefly recall the construction of the usual
Witt vectors, then introduce our Witt vector construction $W_{G}(R; M)$. In fact we give a slightly more
general construction. Firstly, in order to compute the full Mackey functor of
zeroth equivariant homotopy groups of the norm, we want a group $W_{H \le G}(R; M)$ for each open
subgroup $H \le G$, where we define $W_G(R; M) \coloneqq W_{G \le G}(R; M)$.
In fact it turns out (Lemma~\ref{lem:HGHHiso}) that $W_{H \le G}(R; M) \cong W_{H \le H}(R; M^{\otimes_R G/H})$,
but this isomorphism is not canonical, so we need to parametrise by both $H$ and
$G$ to avoid a lot of unnecessary bookkeeping. Secondly, we will need a
truncated version of the Witt vectors, generalising the truncated big Witt vectors of a ring
(e.g.\ as described in \cite{hesselholt_big_2015}). Let $S$ be
a set of open subgroups of $H$ that is upwards closed
and closed under conjugation; we call this a truncation set for $H$. Then in
Theorem~\ref{thm:witt_properties} we will give defining properties for a group of truncated Witt
vectors $W_{H \le G}^S(R; M)$. Note $W_{H \le G}(R; M) \coloneqq W^{S_0}_{H \le
  G}(R; M)$ where $S_0$ is the set of all open subgroups of $H$.

The construction of the $G$-typical Witt vectors with coefficients is largely analogous to that in \cite{dotto_witt_2025}, with
some inspiration from \cite{dress_burnside_1988}. The key technical ingredient is the Dwork lemma
(Lemma~\ref{lem:dwork}), which takes a similar form to Theorem~2.7.3 of
\cite{dress_burnside_1988} but has an almost entirely new proof. We complete the
construction with Definition~\ref{def:witt}.

We define operators between these groups in Section~\ref{sec:operators}.
As with the usual Witt vectors, there are Frobenius and Verschiebung operators.
Given $K$ an open subgroup of $H$ we have natural
additive maps
\[F^H_K : W^S_{H \le G}(R; M) \to W^{S \mid_K}_{K \le G}(R; M) \text{,}\]
\[V^H_K : W^{S \mid_K}_{K \le G}(R; M) \to W^S_{H \le G}(R; M) \text{,}\]
where $S\!\!\mid_K$ is the set of subgroups in $S$ that are contained in $K$.
We also define a conjugation operator
\[c_g : W^S_{H \le G}(R; M) \to W^{gSg^{-1}}_{gHg^{-1} \le G}(R; M)\]
for $g \in G$, where $gSg^{-1}$ is the truncation set for $gHg^{-1}$
obtained by conjugating the subgroups in $S$.
Given another truncation set $S' \subseteq S$ for $H$ we define a truncation operator
\[R_{S'} : W^S_{H \le G}(R; M) \to W^{S'}_{H \le G}(R; M) \text{.}\]
We define a Teichm\"uller map
\[\tau_{G/H} : M^{\otimes_R G/H} \to W^S_{H \le G}(R; M) \text{.}\]
This depends on a choice of coset representatives for $G/H$, and is
not necessarily additive.
Finally we construct an external product
\[\star : W^S_{H \le G}(R; M) \otimes_{\mathbb{Z}} W^S_{H \le G}(R'; M') \to W^S_{H \le G}(R
  \otimes_{\mathbb{Z}} R'; M \otimes_{\mathbb{Z}} M') \text{,}\]
making $W^S_{H \le G}$ a lax monoidal functor from a suitable category of modules to
the category of abelian groups.

In Section~\ref{sec:generalise} we show that the $G$-typical Witt vectors
with coefficients generalise the previous variants of Witt vectors described in
Section~\ref{sec:prior_work}. The external product lets us recover the
multiplication on the $p$- and $G$-typical Witt vectors of a ring.

To prove that $W_{H \le G}(R; M)$ computes the zeroth homotopy of the norm we will use induction and the isotropy separation sequence
of spectra.
We review isotropy separation (and the corresponding algebraic
constructions on Mackey functors) in Section~\ref{sec:isotropy_and_mackey}. When
$G$ is finite the Frobenius, Verschiebung and conjugation operators
give $H/K \mapsto W^{S\mid_K}_{K \le G}(R; M)$ the structure of an $H$-Mackey functor, which we
denote $\underline{W}^S_G(R; M)$ (Definition~\ref{def:witt_mackey}). The analogue of the isotropy separation sequence is an exact
sequence relating different truncations (Lemma~\ref{lem:witt_exact}); %
note the complement of a truncation set is precisely the usual notion of a family of subgroups in
equivariant stable homotopy theory.
We also prove that the external product makes the untruncated Witt vector construction
$\underline{W}_G$ into a strong symmetric
monoidal functor from modules to $G$-Mackey functors, sending tensor products of
modules to box products of Mackey functors.

In Section~\ref{sec:norm} we recall the norm construction and prove the main
theorem (a more refined version of Theorem~\ref{introthm:basic_iso} above):
\begin{introtheorem}
  For $G$ a finite group and $X$ a connective spectrum, we
  have an isomorphism of Mackey functors
  \[\underline{\pi}_0(N_{\{e\}}^G X) \cong \underline{W}_G(\mathbb{Z}; \pi_0 X) \text{.}\]

  More generally, suppose $S$ is a truncation set for a subgroup $H \le G$ and let
  $\mathcal{F}(S) = \{U \le H \mid U \not\in S\}$ be the family of
  subgroups of $H$ that are not in $S$. Then
  \[\underline{\pi}_0(N_{\{e\}}^G X \wedge {\tilde{E}\mathcal{F}(S)}) \cong \underline{W}_G^S(\mathbb{Z}; \pi_0 X) \text{,}\]
  where ${\tilde{E} \mathcal{F}(S)}$ is the based $H$-space with
  ${\tilde{E} \mathcal{F}(S)}^U$ homotopy equivalent to $S^0$ if $U \in S$, and
  contractible otherwise.
\end{introtheorem}
This will be Theorem~\ref{thm:main_theorem} in the main text. Our proof takes an
approach
broadly analogous to that in \cite{dotto_witt_2025}. The main new work is
the construction of an analogue of the Teichm\"uller map for the norm, in Section~\ref{sec:norm_teichmuller}.

In Appendix~\ref{app:computation} we outline some strategies for doing explicit
computations with the $G$-typical Witt vectors with coefficients,
illustrated by the calculation
\[W_{D_6}(\mathbb{Z}; \mathbb{Z}/3\mathbb{Z}) \cong (\mathbb{Z}/3\mathbb{Z})^2
  \oplus \mathbb{Z}/9\mathbb{Z} \text{.}\]

\subsection{Acknowledgements}

We would like to thank Emanuele Dotto and Irakli Patchkoria for many helpful
conversations, and the anonymous referee for their helpful comments. %

The author is supported by the Warwick Mathematics Institute Centre for Doctoral Training, and gratefully acknowledges funding from the University of Warwick.

\section{Preliminaries}\label{sec:preliminaries}

Before we start the construction of the $G$-typical Witt vectors with
coefficients, we will introduce some basic concepts and notation.

\subsection{Conventions and notation}

Unless otherwise specified, all rings are commutative.

We will write $\text{Top}_\text{Haus}$ for the category of Hausdorff topological spaces
and $\text{Ab}_\text{Haus}$ for the category of Hausdorff topological abelian groups.

Let $R$ be a ring and $M$ an $R$-module. We use $M^{\otimes_R n}$ to denote the $n$-fold
tensor product of $M$ over $R$. For $X$ a set, we use $M^{\otimes_R X}$ to
denote the $\abs{X}$-fold tensor product of $M$ over $R$, which we consider to
be indexed by elements of $X$. It is important that $R$ is
commutative, so ${\otimes_R}$ gives a symmetric monoidal product on the
category of $R$-modules; this means that $M^{\otimes_R X}$ is well-defined
without needing a distinguished ordering on the set $X$.\footnote{The
  Witt vectors with coefficients of \cite{dotto_witt_2022} use an $(R, R)$-bimodule $M$ where $R$
  is not necessarily commutative. Instead of usual tensor powers, they use the
  construction $M^{\circledcirc_R n} \coloneqq
  M^{\otimes_R n}/[R, M^{\otimes_R n}]$,
  which can be thought of as $n$ copies of $M$ ``tensored together around
  a circle'' and as such has a natural action of the cyclic group $C_n$. An
  arbitrary group does not come with any particular cyclic or linear ordering, so we
  do not expect to be able to usefully extend our construction to
  non-commutative rings.}
We will often treat $M^{\otimes_R X}$ as a discrete topological abelian group.

Let $G$ be a profinite group. We write $V \le_o G$ to denote an open subgroup. Given $g \in G$, we write
$\prescript{g}{}{V}$ to denote $gVg^{-1}$ and $V^g$ to denote $g^{-1}Vg$.
Let $U$ be another open subgroup of $G$. We say that $U$ is subconjugate to $V$ if
there exists some conjugate of $U$ that is contained in $V$.

\subsection{Category of modules}\label{sec:cat_of_modules}

Our Witt vector construction can be applied to any module over any (commutative)
ring, and will be functorial in the choice of module.
To formalise this, recall the category of all modules.

\begin{definition}Let $\text{Mod}$ denote the category of all modules. The objects are pairs
  \[(R; M)\]
  where $R$ is a ring and $M$ is an $R$-module. Morphisms are pairs
  \[(\alpha; f) : (R; M) \to (R'; M')\]
  consisting of a ring homomorphism $\alpha : R \to R'$ together with an
  $R$-module homomorphism $f: M \to \alpha^\ast M'$ (where $\alpha^\ast$ denotes restriction of scalars).
\end{definition}

Let $\mathbb{Z}[X]$ denote the free ring on a set $X$, i.e.\ the polynomial ring
with variables $X$. Let $R(Y)$
denote the free $R$-module on a set $Y$. We have a free-forgetful adjunction, where the forgetful functor $U : \text{Mod} \to \text{Set} \times \text{Set}$ given by
$(R; M) \mapsto (R, M)$ has left adjoint $F$ given by $(X, Y) \mapsto
(\mathbb{Z}[X], \mathbb{Z}[X](Y))$. When we say ``a free object of
$\text{Mod}$'' we mean an object in the essential image of $F$, i.e.\ a free module over a
free ring.

\subsection{Reflexive coequalisers}

A reflexive coequaliser in a category $\mathcal{C}$ is a coequaliser
\[A \, \substack{\xrightarrow{f}\\[-0.2em] \xrightarrow[g]{}} \, B \xrightarrowdbl{} C\]
where the parallel morphisms $f, g : A \to B$ of $\mathcal{C}$ have a common section $s :
B \to A$. A reflexive
coequaliser of abelian groups can be computed by
taking the coequaliser of the underlying sets and giving it
the quotient group structure. But this coequaliser is particularly nice.
Given $b, b' \in B$, say
that $b \sim b'$ iff there exists $a \in A$ such that $f(a) = b$ and $g(a) =
b'$. Then $\sim$ is an equivalence relation, and the coequaliser is the quotient
$B / \!\!\sim$. For general reflexive coequalisers of sets, the analogous relation is not transitive, so we
would need to identify elements connected by chains of such relations. Reflexive coequalisers in categories of rings or modules can be computed
by taking the coequaliser of the underlying abelian groups (or sets).
Using
this, it is easy to check that $\text{Mod}$ has reflexive coequalisers
and these are preserved by the forgetful functor $U : \text{Mod} \to \text{Set} \times \text{Set}$.

We will need to understand reflexive coequalisers in $\text{Top}_\text{Haus}$ and
$\text{Ab}_\text{Haus}$. To compute a reflexive coequaliser in $\text{Top}_\text{Haus}$, take the coequaliser
in $\text{Top}$, then quotient by the smallest
equivalence relation such that the result is a Hausdorff space. %
To compute a
reflexive coequaliser in $\text{Ab}_\text{Haus}$, take the coequaliser in $\text{Top}_\text{Haus}$ and give
it the quotient group structure. %
Equivalently, the coequaliser of $A \,
\substack{\xrightarrow{f}\\[-0.2em] \xrightarrow[g]{}} \, B$ in
$\text{Ab}_{\text{Haus}}$ is the quotient of $B$ by the closure of the subgroup
$\{f(x)-g(x) \mid x \in A\}$.

We gather together some lemmas about reflexive coequalisers.

\begin{lemma} \label{lem:space_coeq_quotient}
  Reflexive coequalisers in $\text{Top}_\text{Haus}$ preserve topological quotient maps.
\end{lemma}
\begin{proof}
  Consider a commutative diagram
  \[\begin{tikzcd}
      A_1 \ar[r, shift left=0.5ex] \ar[r, shift right=0.5ex] \ar[d, "\theta_A"] & B_1 \ar[r,
      "p_1"] \ar[d, "\theta_B"] & C_1 \ar[d, "\theta_C"] \\
      A_2 \ar[r, shift left=0.5ex] \ar[r, shift right=0.5ex] & B_2 \ar[r, "p_2"] & C_2
    \end{tikzcd}\]
  where the rows are reflexive coequalisers in $\text{Top}_\text{Haus}$ and $\theta_B$ is a quotient map.
  Then since $p_1$ and $p_2$ are quotient maps, considering the commutative
  square $\theta_C p_1 = p_2 \theta_B$ shows that $\theta_C$ is also a quotient map.
\end{proof}

\begin{corollary}
  \label{cor:abgroup_coeq_quotient}
  Reflexive coequalisers in $\text{Ab}_\text{Haus}$ preserve topological
  quotient maps.
\end{corollary}
\begin{proof}
  As noted above the underlying space of a reflexive coequaliser of Hausdorff topological abelian
  groups is the reflexive coequaliser of the underlying Hausdorff spaces, so
  this follows from the above lemma.
\end{proof}

\begin{lemma} \label{lem:ab_product_preserves_reflexive_coequalisers}
  Let $I$ be a (possibly infinite) indexing set. Then the product functor $\text{Ab}^I \to \text{Ab}_\text{Haus}$ preserves reflexive coequalisers.
\end{lemma}
\begin{proof}
  Suppose
  \[A_i \, \substack{\xrightarrow{f}\\[-0.2em]
      \xleftarrow{}\\[-0.2em] \xrightarrow[g]{}} \, B_i \to C_i\]
  is a reflexive coequaliser of abelian groups, for all $i \in I$. We want to show that
  \[\prod_{i \in I} A_i \, \substack{\xrightarrow{}\\[-0.2em]
      \xleftarrow{}\\[-0.2em] \xrightarrow[]{}} \, \prod_{i \in I} B_i \to
    \prod_{i \in I} C_i\]
  is a reflexive coequaliser in $\text{Ab}_\text{Haus}$, where the terms have
  the usual product topology.

  We first observe that this diagram is a reflexive coequaliser in $\text{Ab}$.
  This is fairly straightforward to check, crucially using the fact that $f(x)
  \sim g(x)$ for $x \in A_i$ is already an equivalence relation on $B_i$ (note
  infinite products do not in general preserve reflexive coequalisers of sets since this
  step fails; the abelian group structure is crucial here). Since
  $\prod_{i \in I} B_i \to \prod_{i \in I} C_i$ is a quotient map (it is
  surjective and open), the diagram is also a reflexive coequaliser of
  topological abelian groups. Since all terms are Hausdorff, we conclude that the diagram is a reflexive coequaliser in $\text{Ab}_\text{Haus}$.
\end{proof}

\begin{lemma}\label{lem:tensorcoequaliser}
  Tensor powers preserve reflexive coequalisers. That is, let
  \[(R_1; M_1) \,\substack{\xrightarrow{} \\[-0.2em]
      \xleftarrow{}\\[-0.2em] \xrightarrow{}} \, (R_0; M_0) \xrightarrowdbl{} (R; M)\]
  be a reflexive coequaliser diagram in $\text{Mod}$. Then
  \[(R_1; M_1^{\otimes_{R_1} n}) \,\substack{\xrightarrow{} \\[-0.2em]
      \xleftarrow{}\\[-0.2em] \xrightarrow{}} \, (R_0; M_0^{\otimes_{R_0} n})
    \xrightarrowdbl{} (R; M^{\otimes_R n})\]
  is a reflexive coequaliser in $\text{Mod}$ (and hence also gives a reflexive
  coequaliser of underlying abelian groups).
\end{lemma}
\begin{proof}
  This is a somewhat tedious but not too hard exercise. Alternatively, see the proof of
  Proposition~1.14 in \cite{dotto_witt_2022}.
\end{proof}

\subsection{Moving between tensor powers}\label{sec:movingtensor}

To define the Witt vectors we will need to use various operations on tensor
powers indexed by cosets.

First we look at the operations induced by right multiplication of cosets. Let
$V$ be an open subgroup of a profinite group $G$. Given $g \in G$, there is an
isomorphism $G/V \cong G/\prescript{g}{}{V}$ given by $aV \mapsto aVg^{-1}$. This induces
an isomorphism
\[g \cdot ({-}) : M^{\otimes_R G/V} \xrightarrow{\cong} M^{\otimes_R G/\prescript{g}{}{V}} \text{,}\]
defined on elementary tensors by sending
\[\bigotimes_{aV \in G/V} m_{aV}\]
in $M^{\otimes_R G/V}$ to the tensor
\[\bigotimes_{a \prescript{g}{}{V} \in G/\prescript{g}{}{V}} m_{ag V}\]
in $M^{\otimes_R G/\prescript{g}{}{V}}$.
Note that this map only depends on the coset $gV \in G/V$. The notation
is chosen to be reminiscent of a group action. Indeed when $g$ is in the
normaliser $N_G(V)$ of $V$ in $G$, we get
\[g \cdot ({-}) : M^{\otimes_R G/V} \xrightarrow{\cong} M^{\otimes_R G/V} \text{,}\]
giving an action of the Weyl group $N_G(V)/V$ on $M^{\otimes_R G/V}$. And for any $g \in
G$, we get a map on the product over all subgroups
\[g \cdot ({-}) : \prod_{U \le_o G} M^{\otimes_R G/U} \xrightarrow{\cong} \prod_{U \le_o G}
  M^{\otimes_R G/\prescript{g}{}{U}} \xrightarrow{\cong} \prod_{U \le_o G} M^{\otimes_R G/U}\]
(where the second isomorphism comes from reindexing the product), giving an action of
$G$ on $\prod_{U \le_o G} M^{\otimes_R G/U}$. These actions give us transfer
maps between groups of fixed points. We have a map
\begin{align*}
  \text{tr}_W^{W'} : (M^{\otimes_R G/V})^W &\to (M^{\otimes_R G/V})^{W'}\\
  x &\mapsto \sum_{aW \in W'/W} a \cdot x
\end{align*}
for $V \le_o W \le_o W' \le_o N_G(V)$, and
\begin{align*}
\text{tr}_K^H : \left( \prod_{U \in S} M^{\otimes_R G/U} \right)^K &\to
  \left(\prod_{U \in S} M^{\otimes_R G/U}\right)^H\\
  x &\mapsto \sum_{aK \in H/K} a \cdot x
\end{align*}
for $K \le_o H \le_o G$. The transfer $\text{tr}_W^{W'}$ factors
through the $W'$-coinvariants, giving a map $\left((M^{\otimes_R
    G/V})^W\right)_{W'} \to (M^{\otimes_R G/V})^{W'}$, and
similarly for the transfer on the product. We will frequently use the following result:
\begin{lemma}
  Let $T$ be a torsion-free ring, and let $Q = T(Y)$ be a
  free $T$-module. Then for $V \le_o W \le_o N_G(V)$ the transfer
  \[\text{tr}_V^{W} : (Q^{\otimes_T G/V})_{W} \to (Q^{\otimes_T G/V})^{W}\]
  is injective.

  Here $({-})_{W}$ denotes the abelian group of orbits, i.e.\ the quotient
  by the subgroup generated by elements of the form $g \cdot x - x$ for $g \in W$.
\end{lemma}
\begin{proof}
  Note $Q^{\otimes_T G/V} \cong T(Y^{\times G/V})$. The $T$-module $(Q^{\otimes_T G/V})_{W}$ is free with basis elements
  corresponding to orbits of $W$ acting on the set $Y^{\times G/V}$. And
  $(Q^{\otimes_T G/V})^{W}$ is free with a basis element corresponding to
  the formal sum of the elements of an orbit in $Y^{\times G/V}$. The transfer
  sends an orbit to a natural number multiple of the formal sum of the elements in the
  orbit. So in this basis it is given by a diagonal matrix, hence injective.
\end{proof}

Let $U \le_o V \le_o G$. To define our version of the ghost map, we will need
some sort of ``tensor power map'' $M^{\otimes_R G/V} \to
M^{\otimes_R G/U}$. However the canonical tensor power map $({-})^{\otimes_R V/U} : M^{\otimes_R
  G/V} \to M^{\otimes_R G/V \times V/U}$ has the wrong codomain. To fix this, we
will postcompose with a (non-canonical) isomorphism $M^{\otimes_R G/V \times V/U} \cong
M^{\otimes_R G/U}$. Suppose we have chosen a set of coset representatives
$\{g_i V\}_{i \in I}$ for $G/V$. Then we can define an isomorphism $G/V \times V/U \cong G/U$ via $(g_i V, sU) \mapsto g_i s U$.

\begin{definition} \label{def:f}
  Given a set of coset representatives for $G/V$, we denote the induced isomorphism on tensor powers by
  \[f_{G/V} : M^{\otimes_R G/V \times V/U} \to M^{\otimes_R G/U} \text{.}\]
\end{definition}

\begin{remark} \label{rem:fidentity}
We will later need the following identity: given $a \in V$ and $n \in M^{\otimes_R G/V}$ we have
\[a \cdot f_{G/V}(n^{\otimes_R V/U}) = f_{G/V}(n^{\otimes_R V/\prescript{a}{}{U}}) \text{.}\]
\end{remark}

\subsection{Frobenius lifts}

The tensor power map $M^{\otimes_R G/V} \to M^{\otimes_R G/U}$ discussed above
is not additive. However when $(R; M)$ is free, there is a related additive map $\phi_U^V$
which we call the Frobenius lift.\footnote{This map $\phi_U^V$ will behave similarly to an
  ``external Frobenius'' as defined in
  \cite{dotto_witt_2025} Definition~A.4. %
  However while the proofs in \cite{dotto_witt_2025} work with a bimodule
  equipped with any choice of external Frobenius, we will simply use the
  specific map $\phi_U^V$.}

Let $T = \mathbb{Z}[X]$ be the free (commutative) ring on a set $X$, and let $Q =
T(Y)$ be the free $T$-module on a set $Y$.
For $n \in \mathbb{N}$, let
\[\varphi_n : \mathbb{Z}[X] \to \mathbb{Z}[X]\]
denote the ring homomorphism defined by sending each generator $x \in X$ to
$x^n$.

\begin{definition}\label{def:frobenius_phi}
  Given $U \le_o V \le_o G$, we define the Frobenius lift
  \[\phi_U^V : Q^{\otimes_T G/V} \to Q^{\otimes_T G/U}\]
  as follows. We have $Q^{\otimes_T G/V} = T(Y)^{\otimes_T G/V} \cong T(Y^{\times G/V})$, and so
we can write any element of $Q^{\otimes_T G/V}$ in the form
\[\sum_{j \in J} \left( r_j \bigotimes_{gV \in G/V} y^j_{gV} \right) \]
where $J$ is a finite indexing set, $r_j \in T$ and $y^j_{gV} \in Y$. There is a canonical
surjection $G/U \xrightarrowdbl{} G/V$ given by $gU \mapsto gV$. This induces a
map $Y^{\times G/V} \to Y^{\times G/U}$, and we can extend this to define a map
\begin{align*} \phi_U^V : Q^{\otimes_T G/V} &\to Q^{\otimes_T G/U}\\
  \sum_{j \in J} \left( r_j \bigotimes_{gV \in G/V} y^j_{gV} \right) &\mapsto \sum_{j \in J} \left( \varphi_{\abs{V:U}} (r_j) \bigotimes_{gU \in G/U} y^j_{gV} \right) \text{.}\end{align*}
\end{definition}

Note that $\phi_U^V$ is additive but not a $T$-module homomorphism. Also note
that it depends on the choice of generators of $T$ and $Q$. %

\begin{remark}
For $g \in G$ and $m \in Q^{\otimes_T G/V}$ we have the following identity:
\[g \cdot \phi_U^V(m) = \phi_{\prescript{g}{}{U}}^{\prescript{g}{}{V}}(g \cdot m) \text{.}\]
For $U \le V \le W$ we have
\[\phi^V_U \phi^W_V = \phi^W_U \text{.}\]
\end{remark}

\section{Definition of the Witt vectors} \label{sec:witt_definition}

We start by briefly recalling how the usual $p$-typical Witt vectors of a
commutative ring are defined, following a similar approach to the exposition in
\cite{hesselholt_big_2015}. Next we will see how to generalise this to give a
uniqueness theorem for the $G$-typical Witt vectors with coefficients, Theorem~\ref{thm:witt_properties}.
We define the operators and monoidal structure on the Witt vectors (including versions of the usual
Frobenius and Verschiebung), and finally prove that our construction generalises
those prior variants of Witt vectors that we discussed in Section~\ref{sec:prior_work}.

\subsection{\texorpdfstring{$p$}{p}-typical Witt vectors} \label{sec:classical_witt}

Let $p$ be a prime and $n$ a natural number. Given a commutative ring $R$, the
ring of $n$-truncated, $p$-typical Witt vectors of $R$ is a commutative ring $W_{n, p}(R)$. A good example to keep in mind is that
\[W_{n, p}(\mathbb{F}_p) \cong \mathbb{Z}/p^n \mathbb{Z} \text{.}\]
Rather than describe the ring $W_{n, p}(R)$ explicitly, it is often defined by
giving a list of properties that specify it uniquely.

\begin{thmdef}
  There is a unique functor
  \[W_{n, p} : \text{CRing} \to \text{CRing}\]
  satisfying the following.
\begin{enumerate}[(i)]
\item{The underlying set of $W_{n, p}(R)$ is $\prod_{0 \le i < n} R$. Given a %
    ring homomorphism $f : R \to S$, the map $W_{n, p}(f) : W_{n, p}(R) \to
    W_{n, p}(S)$ is $\prod_{0 \le i < n} f$ at the level of sets.}
\item{For each $0 \le j < n$, define the map $w_j : W_{n, p}(R) \to R$ via
      \[(a_i) \mapsto \sum_{0 \le i \le j} p^i a_i^{p^{j-i}} \text{.}\]
    The product
    \[(w_0, \dotsc, w_{n-1}) : W_{n, p}(R) \to \prod_{0 \le j < n} R\]
    is called the ghost map, and denoted $w$. This map is a ring homomorphism (where the right hand
    side has the usual product ring structure).
  }
\end{enumerate}
\end{thmdef}

We will not prove this in detail, but we will give a brief sketch, since
we will use some of the same ideas later in the paper. We already know the
underlying set of $W_{n, p}(R)$, so we just
need to define the ring structure on this set in a way that satisfies the above properties.
The key technical ingredient is the following lemma.

\begin{lemma}[Dwork lemma for $p$-typical Witt vectors]\label{lem:classical_dwork}
  Suppose there exists a ring homomorphism $\phi_p : R \to R$ such that
  $\phi_p(r) \equiv r^p \text{ modulo } pR$ for all $r \in R$ (that is, a
  lift of the Frobenius homomorphism on $R/pR$). Then an element $a
  \in \prod_{0 \le j < n} R$ is in the image of the ghost map $w : W_{n, p}(R)
  \to \prod_{0 \le j < n} R$ iff it satisfies
  \[a_j \equiv \phi_p(a_{j-1}) \text{ modulo } p^j R\]
  for all $1 \le j < n$.
\end{lemma}
\begin{proof}
  See \cite{hesselholt_big_2015} Lemma~1.1. Commonly attributed to Dwork.
\end{proof}
\begin{corollary}
  Under the
  conditions of the lemma, the image of the ghost map is a subring of $\prod_{0
    \le j < n} R$.
\end{corollary}
\begin{proof}
  Since $\phi_p$ is a ring homomorphism, the congruence condition shows that $\text{im}(w)$
  contains the multiplicative unit and is closed under subtraction and
  multiplication, so it is a subring.
\end{proof}

Consider $R = \mathbb{Z}[X]$, the free commutative ring on a
set $X$. It is easy to check that in this case the ghost map $w : W_{n,
  p}(\mathbb{Z}[X]) \to \prod_{0 \le j < n} \mathbb{Z}[X]$ is injective (this
only needs the fact that $\mathbb{Z}[X]$ is torsion free). If we can check that its image is
a subring then there is clearly a unique ring structure on $W_{n,
p}(\mathbb{Z}[X])$ so that the ghost map is a ring homomorphism. If we define a ring homomorphism $\phi_p : \mathbb{Z}[X] \to
\mathbb{Z}[X]$ by $x \mapsto x^p$ for each $x \in X$, then this satisfies
$\phi_p(r) \equiv r^p \text{ mod } p\mathbb{Z}[X]$. So
Lemma~\ref{lem:classical_dwork} applies, and the Corollary shows that the image
of the ghost map is a subring.

For a general commutative ring $R$, we can use functoriality to extend from the
free ring case and define a ring structure on $W_{n,
  p}(R)$. With some care we can do this in a well-defined way, and show that these are the unique ring structures satisfying the properties above.

\subsection{Defining properties of the \texorpdfstring{$G$}{G}-typical Witt vectors with coefficients} \label{sec:defining_properties}

Recall that we refer to the complement of a family of subgroups as a truncation set.

\begin{definition} A truncation set for a profinite group $H$ is a set $S$
  of open subgroups of $H$, such that
  \begin{enumerate}[(i)]
  \item{$S$ is upwards closed (i.e.\ if $U \in S$ and $U'$ an open subgroup of $H$
      containing $U$ then
      $U' \in S$), and}
  \item{$S$ is closed under conjugation.}
  \end{enumerate}
\end{definition}

Let $G$ be a profinite group, $H$ an open subgroup and $S$ a truncation set for $H$.
Let $R$ be a commutative ring and $M$ an $R$-module. We
will define a Hausdorff topological abelian group of Witt vectors $W^S_{H \le G}(R; M)$.

As in the case of the usual Witt vectors, the group $W^S_{H \le G}(R; M)$ is hard to describe
very explicitly. Instead we show that there is a unique construction (up to unique
isomorphism) satisfying various conditions.
These were inspired by the various
previous definitions of Witt vectors, particularly those properties used in
\cite{dotto_witt_2025} Section~1.1 to describe Witt vectors with coefficients. %
One subtlety is that our specification will require making some arbitrary choices,
although these won't affect the final result.

For the usual Witt vectors, we know the underlying set of $W_{n, p}(R)$ is
$\prod_{0 \le i < n} R$, and
the ring operations are governed by the ghost map $w : \prod_{0 \le i < n} R \to
\prod_{0 \le j < n} R$. In our case, we can't immediately describe the
underlying space so precisely, but it will at least be a topological quotient of
\[\prod_{V \in \can{S}} M^{\otimes_R G/V} \text{,}\]
where $\can{S} \subseteq S$ consists of a choice of distinguished representative for
each $H$-conjugacy class of subgroups in $S$.

The addition operation and this quotient are again governed by a
ghost map. For each distinguished subgroup $V \in \can{S}$, we need to fix a
choice of coset representatives for $G/V$.
Then using these choices, we have an explicit formula for a continuous map
\[w : \prod_{V \in \can{S}} M^{\otimes_R G/V} \to \left( \prod_{U \in S} M^{\otimes_R G/U} \right)^H\]
called the ghost map; we delay describing this until
Section~\ref{sec:ghost_map}.

We can now give a formal uniqueness result defining the $S$-truncated
$G$-typical Witt vectors with coefficients, although the proof that they exist relies on
results from the next several subsections.

\begin{theorem} \label{thm:witt_properties}
  There is a functor
  \[W^S_{H \le G} : \text{Mod} \to \text{Ab}_\text{Haus}\]
  and there are quotient maps of underlying spaces
  \[q : \prod_{V \in \can{S}} M^{\otimes_R G/V} \xrightarrowdbl{} W^S_{H \le G}(R; M)\]
  natural in $(R; M)$, satisfying the following properties:
  \begin{enumerate}[(i)]
  \item The ghost map $w$ factors through $q$, inducing a natural additive map
    \[w : W^S_{H \le G}(R; M) \to \left( \prod_{U \in S} M^{\otimes_R G/U} \right)^H\]
    that we will also call the ghost map.
  \item For each $(T; Q)$ free (i.e.\ a free module over a free ring) the
    ghost map out of the Witt vectors is an embedding of topological abelian groups. That is, the ghost map induces an isomorphism
  \[W^S_{H \le G}(T; Q) \cong \text{im}(w) \le \left( \prod_{U \in S}
      Q^{\otimes_T G/U} \right)^H \text{.}\]
\item The functor $W^S_{H \le G}$ preserves reflexive coequalisers.
  \end{enumerate}
  These properties specify the functor $W^S_{H \le G}$ together with the natural
  map $w
  : W^S_{H \le G}(R; M) \to \left( \prod_{U \in S} M^{\otimes_R G/U} \right)^H$
  uniquely up to unique natural isomorphism, even if we make different choices
  of the distinguished subgroups or coset representatives used to define $w : \prod_{V
    \in \can{S}} M^{\otimes_R G/V} \to \left( \prod_{U \in S} M^{\otimes_R G/U} \right)^H$.
\end{theorem}
\begin{proof}
  In Section~\ref{sec:dwork} we study $\text{im}(w) \subseteq \left(\prod_{U \in
      S} Q^{\otimes_T G/U}\right)^H$ for $(T; Q)$ free. In particular, Corollary~\ref{cor:ghost_image_subgroup_indep} shows that the image is a closed subgroup and independent of the choices used in its
  definition, and Lemma~\ref{lem:ghost_quotient} shows that in this case $w$ is
  a quotient onto its image.

  First consider existence of $W^S_{H \le G}$ and $q$ satisfying these properties. In Section~\ref{sec:witt_final_definition} we show
  (Lemma~\ref{lem:finite_module_kan}) that any functor defined
  on the full subcategory of free objects of $\text{Mod}$ admits an essentially unique
  extension to a reflexive coequaliser-preserving functor on all of
  $\text{Mod}$. So we can define $W^S_{H \le G}$ on free modules via $W^S_{H \le
  G}(T; Q) \coloneqq \text{im}(w)$, and extend to get a functor on all
modules (Definition~\ref{def:witt}).
For free coefficients $w : \prod_{V \in \can{S}} Q^{\otimes_T G/V} \to \left(
  \prod_{U \in S} Q^{\otimes_T G/U} \right)^H$ factors through the Witt vectors as a
quotient followed by an embedding; by Lemma~\ref{lem:finite_module_kan}, for general coefficients this extends to a factorisation of $w$ through the Witt vectors as a quotient $q$ followed by
an additive map (Definitions~\ref{def:witt_quotient} and \ref{def:witt_ghost}).

Now consider the uniqueness result. This follows because reflexive
coequaliser-preserving extensions are essentially unique.
Explicitly, suppose we had a functor $W'^S_{H \le G}$ and map $w' :
W'^S_{H \le G}(R; M) \to \left( \prod_{U \in S} M^{\otimes_R G/U} \right)^H$
also coming from the above defining properties but potentially using different choices of coset
representatives. Since the image of the ghost map doesn't depend on the choices,
$w$ and $w'$ induce an isomorphism $u : W^S_{H \le G}(T; Q) \cong \text{im}(w) \cong
W'^S_{H \le G}(T; Q)$ for $(T; Q)$ free, unique such that $w'u = w$. And now by
Lemma~\ref{lem:finite_module_kan} this extends to a unique natural isomorphism $u : W^S_{H
  \le G}(R; M) \cong W'^S_{H \le G}(R; M)$ such that $w'u = w$. %
\end{proof}

\subsection{The ghost map} \label{sec:ghost_map}

To complete the definition in Section~\ref{sec:defining_properties} we need to describe the ghost map
\[w : \prod_{V \in \can{S}} M^{\otimes_R G/V} \to \left(\prod_{U \in S} M^{\otimes_R
    G/U}\right)^H \text{.}\]
Note that $w$ will not in general be a homomorphism; merely a continuous map of topological spaces (considering
$M^{\otimes_R n}$ to be discrete, and using the usual product and subspace
topologies). However it will descend to a continuous abelian group homomorphism out of the
Witt vectors.

Let $U$ and $V$ be open subgroups of $H$.
Observe that $U$ acts on $H/V$ by left multiplication, and the fixed
points $(H/V)^U$ are the cosets $hV$ such that $U \le \prescript{h}{}{V}$, or
equivalently $U^h \le V$. In particular $(H/V)^U$ is only non-empty when $U$ is
subconjugate to $V$, and this holds for finitely many $V$, so the below sums are well-defined.

\begin{definition}[Ghost map] \label{def:ghost_map}
  For $U \in S$, the $U$-component of the ghost map is the map
  $w_U : \prod_{V \in \can{S}} M^{\otimes_R G/V} \to M^{\otimes_R G/U}$ given by
        \[n \mapsto \sum_{V \in \can{S}} \left( \sum_{hV \in (H/V)^U} h \cdot f_{G/V}(n_V^{\otimes_R V/U^h}) \right) \text{.}\]
  The isomorphism $f_{G/V} : M^{\otimes_R
    G/V \times V/U^h} \to M^{\otimes_R G/U^h}$ (described in Definition~\ref{def:f}) depends on
  our fixed choice of coset representatives for $G/V$.
  Note the choice of representative $h \in hV$ doesn't matter, by Remark~\ref{rem:fidentity}.

  We define the ghost map
  \[w : \prod_{V \in \can{S}} M^{\otimes_R G/V} \to \left( \prod_{U \in S} M^{\otimes_R
      G/U} \right)^H \]
  to be the product of the ghost components $w_U$, over all $U \in S$.
\end{definition}

To see that this is well-defined, we need to check that the image of the product of the ghost
components really does lie in the $H$-fixed points of $\prod_{U \in S} M^{\otimes_R G/U}$.

\begin{lemma}For $a \in H$, we have
  \[a \cdot w_U(n) = w_{\prescript{a}{}{U}}(n) \text{.}\]
\end{lemma}
\begin{proof}
  We have
  \begin{align*}
    a \cdot w_U(n) &= \sum_{V \in \can{S}} \left( \sum_{hV \in (H/V)^U} ah \cdot f_{G/V}(n_V^{\otimes_R V/U^h})) \right)\\
                   &= \sum_{V \in \can{S}} \left( \sum_{hV \in (H/V)^{\prescript{a}{}{U}}} h \cdot f_{G/V}(n_V^{\otimes_R V/U^{a^{-1}h}}) \right)\\
                   &= w_{\prescript{a}{}{U}}(n) \text{,}
  \end{align*}
  where the second equality holds because $U$ fixes $hV$ iff $\prescript{a}{}{U}$ fixes
  $ahV$.
\end{proof}

Since we will use it a lot, let us give a name to the codomain of the ghost map.

\begin{definition}
  Define the ghost group
  \[\text{gh}^S_{H \le G}(R; M) \coloneqq \left( \prod_{U \in S} M^{\otimes_R G/U} \right)^H \text{.}\]
\end{definition}

This extends to a functor $\text{gh}^S_{H \le G} :
\text{Mod} \to \text{Ab}_\text{Haus}$, and the ghost map gives a natural transformation
\[w : \prod_{V \in \can{S}} M^{\otimes_R G/V} \to \text{gh}^S_{H \le G}(R; M) \text{.}\]
Note that we can also write
\[\text{gh}^S_{H \le G}(R; M) \cong \prod_{V \in \can{S}} (M^{\otimes_R
  G/V})^{N_H(V)/V} \text{.}\]

\begin{lemma}
  The ghost map defined in Definition~\ref{def:ghost_map} generalises the ghost maps of the usual (possibly truncated) Witt vectors of a commutative
  ring, the $G$-typical Witt vectors of a commutative ring,
  and the Witt vectors with coefficients
  (in the case of a module over a commutative ring).
\end{lemma}
\begin{proof}
  These are all straightforward to check. See Section~\ref{sec:generalise} for a
  precise description of how we generalise these constructions---in particular,
  to generalise the Witt vectors with coefficients we need to make the correct
  choice of coset representatives.

  For illustrative purposes, let us discuss the case of the $n+1$-truncated $p$-typical Witt
  vectors of a commutative ring. Let $S$ be the collection of all subgroups of
  $C_{p^n}$; then we expect $W^S_{C_{p^n} \le C_{p^n}}(R; R) \cong W_{n+1, p}(R)$.
  Note $R^{\otimes_R G/V} \cong R$, and under this isomorphism $h
  \cdot f_{G/V}(({-})^{\otimes_R V/U^h})$ is the map raising an element of
  $R$ to the $\abs{V : U^h}$ power. So for
  $0 \le k \le n$ the
  $C_{p^k}$-component of our
  ghost map becomes the map $w_{C_{p^k}} : \prod_{C_{p^l} \le C_{p^n}} R \to R$
  defined by
    \[a \mapsto \sum_{C_{p^k} \le C_{p^l} \le C_{p^n}} \sum_{hC_{p^l}\in C_{p^n}/C_{p^l}} a_{C_{p^l}}^{p^{l-k}} = \sum_{k \le l \le n} p^{n-l} a_{C_{p^l}}^{p^{l-k}} \text{.}\]
  We can clean this up by labelling the $C_{p^k}$ component of the product by
  the index $j = n-k$; then we get $w_j : \prod_{0 \le i \le n} R \to R$
  defined by
    \[a \mapsto \sum_{0 \le i \le j} p^i a_i^{p^{j-i}} \text{.}\]
  This agrees with the description of the $j$-component of the ghost map for $W_{n+1, p}(R)$
  given in Section~\ref{sec:classical_witt}
\end{proof}

\subsection{The Dwork lemma}\label{sec:dwork}

The group of Witt vectors with free coefficients will be isomorphic to the
image of the ghost map. We need to prove that the image is a subgroup, in order to get an
abelian group structure on the Witt vectors. We also want to show that the image
doesn't depend on any of the arbitrary choices used in the definition of the
ghost map.

Recall Lemma~\ref{lem:classical_dwork}, the Dwork lemma for the usual $p$-typical Witt vectors of a ring $R$. This says that if $R$ can be equipped with a
Frobenius lift $\phi_p : R \to R$ (a ring homomorphism such that $\phi_p(r)
\equiv r \text{ mod } pR$) then we can describe the image of the ghost map
using a collection of congruences. We want to prove an analogue for the
$G$-typical Witt vectors with coefficients. We will only work with the case of a
free module over a free ring $(T; Q)$, and use the maps $\phi_U^V : Q^{\otimes_T
  G/V} \to Q^{\otimes_T G/U}$ described in Definition~\ref{def:frobenius_phi}
in place of the $\phi_p$. We will use a more complicated collection of
congruences, modelled after Theorem~2.7.3 of
\cite{dress_burnside_1988}.\footnote{That theorem doesn't
  involve any sort of Frobenius lift. However it only applies to
  $W_G(\mathbb{Z})$, and $\text{id} : \mathbb{Z} \to \mathbb{Z}$ is a Frobenius lift for the integers. The theorem itself is essentially
a consequence of Burnside's lemma; the last line of our proof will proceed
similarly, but we will need to do quite a bit of work to reduce to that.}

The essence of the Dwork lemma is to show that
for any $n \in \prod_{V \in \can{S}} Q^{\otimes_T G/V}$, the value of $w_U(n)$
modulo the image of the transfer map $\text{tr}_U^{N_H(U)} : Q^{\otimes_T G/U}
\to Q^{\otimes_T G/U}$ is determined by the values of $w_V(n)$ for all subgroups
$V$ of lower index than $U$ in $H$. We prove this in the following lemma. Then in the
subsequent lemma we will see that any element of $\text{gh}^S_{H \le G}(T; Q)$ satisfying such congruence conditions is in the image of the
ghost map, since the conditions are of exactly the right form to let us
inductively construct a preimage (indeed for $U \in \can{S}$ then $w_U(n)$ is
$\text{tr}_U^{N_H(U)}(n_U)$ plus a term depending only on $n_V$ for lower index $V$).

In fact in the following we prove a slightly more general result: given open subgroups
$U \le_o K \le_o H$, we
describe how $w_U(n)$ is constrained by the values of $w_V(n)$ for $V$ a
subgroup of $K$ of lower
index than $U$. We will need this version later when we
define operators on Witt vectors.

\def\lvUr{{\langle vU \rangle}}

\begin{lemma} \label{lem:congruence}
  Let $G$ be a profinite group, $H$ an
  open subgroup, and $S$ a truncation set for $H$. Let $K \in S$ be an open subgroup of
  $H$, and $U \in S$ an open
  subgroup of $K$.
  Let $T = \mathbb{Z}[X]$ be the free ring on a set $X$, and let $Q =
  T(Y)$ the free $T$-module on a set $Y$.
  Then for any $n \in \prod_{V \in \can{S}} Q^{\otimes_T G/V}$, the sum
  \begin{equation}\label{eq:congruence1}\sum_{vU \in N_K(U)/U} \phi^{\lvUr}_U (w_{\lvUr}(n))\end{equation}
  is in the image of $\text{tr}_U^{N_K(U)} : Q^{\otimes_T G/U} \to Q^{\otimes_T
    G/U}$ (where $\lvUr$ denotes the subgroup of $N_K(U)$ generated by the
  elements of $vU$---observe that since $S$ is upwards closed and $U \in S$, we
  have $\lvUr \in S$ so $w_{\langle vU \rangle}$ is defined).

  Note in most applications we will use the case $K = H$.
\end{lemma}
\begin{proof}
Substituting in the definition of the components of the ghost map and using
additivity of $\phi_U^\lvUr$, we find that (\ref{eq:congruence1}) equals
\[\sum_{vU \in N_K(U)/U} \left( \sum_{V \in \can{S}} \ \sum_{hV \in (H/V)^{\lvUr}}
  \phi_U^{\lvUr} (h \cdot f_{G/V}(n_V^{\otimes_T V/\langle vU \rangle^h})) \right) \text{.}\]
Interchanging sums, we get
\[\sum_{V \in \can{S}} \  \sum_{hV \in (H/V)^U} \left( \sum_{vU \in N_{K
        \cap \prescript{h}{}{V}}(U)/U}
  \phi_U^{\lvUr} (h \cdot f_{G/V}(n_V^{\otimes_T V/\langle vU
    \rangle^h})) \right) \text{.}\]
Using our identities for how the group action interacts with the other maps, we have
\begin{align*}
  & \sum_{vU \in N_{K \cap \prescript{h}{}{V}}(U)/U} \phi_U^{\lvUr} (h \cdot f_{G/V}(n_V^{\otimes_T V/\langle vU \rangle^h}))\\
  & \sum_{vU \in N_{K \cap \prescript{h}{}{V}}(U)/U} h \cdot \phi_{U^h}^{\lvUr^h} (f_{G/V}(n_V^{\otimes_T V/\langle vU \rangle^h}))\\
  &= h \cdot \left(  \sum_{vU^h \in N_{K^h \cap V}(U^h)/U^h} \phi_{U^h}^{\langle v U^h \rangle} (f_{G/V}(n_V^{\otimes_T V/\langle v U^h \rangle})) \right) \text{.}
\end{align*}
We can now decompose according to the orbits of the action of $N_K(U)$ on $(H/V)^U$. Note that
the stabiliser of $hV \in (H/V)^U$ is $\prescript{h}{}{V} \cap N_K(U) = N_{K \cap \prescript{h}{}{V}}(U)$,
so we compute that (\ref{eq:congruence1}) equals the sum over $V$ in $\can{S}$ of
\begin{align*}
  & \sum_{hV \in (H/V)^U/{N_K(U)}} \  \sum_{a \in
    N_K(U)/N_{K \cap \prescript{h}{}{V}}(U)} ah \cdot \left( \sum_{vU^h \in N_{K^h \cap V}(U^h)/U^h}
    \phi_{U^h}^{\langle v U^h \rangle} (f_{G/V}(n_V^{\otimes_T V/\langle v U^h \rangle})) \right)\\
  &= \sum_{hV \in (H/V)^U/{N_K(U)}} \text{tr}^{N_K(U)}_{N_{K \cap \prescript{h}{}{V}}(U)} \left( h \cdot \sum_{vU^h \in N_{K^h \cap V}(U^h)/U^h}
    \phi_{U^h}^{\langle v U^h \rangle} (f_{G/V}(n_V^{\otimes_T V/\langle v U^h \rangle})) \right) \text{.}
\end{align*}
Hence it suffices to prove that
\[ h \cdot \sum_{vU^h \in N_{K^h \cap V}(U^h)/U^h}
  \phi_{U^h}^{\langle v U^h \rangle} (f_{G/V}(n_V^{\otimes_T V/\langle v U^h \rangle}))\]
is in the image of $\text{tr}_U^{N_{K \cap \prescript{h}{}{V}}(U)}$. Since
$h \cdot
\text{tr}_{U^h}^{N_{K^h \cap V}(U^h)}({-}) = \text{tr}_U^{N_{K \cap \prescript{h}{}{V}}(U)} \, h \cdot ({-})$, we can assume without loss of
generality that $h = e$ (just replace $U$ by $U^h$ and $K$ by $K^h$).
So we want to show that
\begin{equation}\label{eq:congruence2}
  \sum_{vU \in N_{K \cap V}(U)/U} \phi_U^{\lvUr} (f_{G/V}(n_V^{\otimes_T V/\lvUr}))
\end{equation}
is in the image of $\text{tr}_U^{N_{K \cap V}(U)}$.

The polynomial ring $T =
\mathbb{Z}[X]$ has a $\mathbb{Z}$-module basis given by the (monic) monomials in
the variables $X$.
Using this, we can write $n_V \in Q^{\otimes_T G/V} \cong \mathbb{Z}[X](Y^{\times G/V})$ in the form
\[n_V = \sum_{j \in J} \left( c_j \theta_j \bigotimes_{gV \in G/V} y^j_{gV} \right)\]
for $c_j \in \mathbb{Z}$, $\theta_j \in \mathbb{Z}[X]$ monomials, and $y^j_{gV} \in Y$. We can expand out
\[n_V^{\otimes_T V/\lvUr} = \sum_{\chi : V/\lvUr \to J}  \
  \bigotimes_{s\lvUr \in V/\lvUr} \left( c_{\chi(s)} \theta_{\chi(s)}
  \bigotimes_{gV \in G/V} y^{\chi(s)}_{gV} \right) \in Q^{\otimes_T G/V \times V/\langle vU \rangle}\text{,}\]
where $\chi$ runs through all functions (of sets) from $V/\lvUr$ to $J$.

We have
\[f_{G/V}(n_V^{\otimes_T V/\lvUr}) = \sum_{\chi : V/\lvUr \to J}
  \left[ \left( \prod_{s\lvUr \in V/\lvUr} c_{\chi(s)} \theta_{\chi(s)} \right) \bigotimes_{g_i s \lvUr
  \in G/\lvUr} y^{\chi(s)}_{g_iV} \right]\]
(now an element of $Q^{\otimes_T G/\langle vU \rangle}$) where in the tensor
product we use the fact that any coset in $G/\lvUr$ can be
uniquely written as $g_i (s \lvUr)$ where $g_i$ is one of the
distinguished coset representatives for $G/V$ and $s\lvUr$ is a coset in $V/\lvUr$.
And so we calculate
\[\phi_U^{\lvUr}(f_{G/V}(n_V^{\otimes_T V/\lvUr})) = \sum_{\chi : V/\lvUr \to J}
  \left[ \left( \prod_{s\lvUr \in V/\lvUr} c_{\chi(s)}
    \theta_{\chi(s)}^{\abs{\lvUr : U}} \right) \bigotimes_{g_i t U
    \in G/U} y^{\chi(t)}_{g_iV} \right]\]
(an element of $Q^{\otimes_T G/U}$) where now $tU$ is a coset in $V/U$.

We can think of a function $V/\lvUr \to J$ as a function $V/U \to J$ with the
property that it
factors through the canonical surjection $V/U \xrightarrowdbl{} V/\lvUr$.
There is an action of $N_V(U)$ on the set $\text{Set}(V/U, J)$ of functions $V/U \to J$, via $(v \cdot \chi)(sU) = \chi(sUv)$. Given $U \le W \le N_V(U)$, a function $\chi : V/U \to J$ factors through
$V/U \xrightarrowdbl{} V/W$ iff $W \le \text{Stab}_{N_V(U)}(\chi)$. So we can expand out (\ref{eq:congruence2}) and
interchange summation to get
\[\sum_{\chi \in \text{Set}(V/U, J)} \  \sum_{vU \in (K \cap \text{Stab}(\chi))/U} \left[ \left(
    \prod_{s\lvUr \in V/\lvUr} c_{\chi(s)} \theta_{\chi(s)}^{\abs{\lvUr:U}} \right) \bigotimes_{g_i t U
    \in G/U} y^{\chi(t)}_{g_iV} \right]\]
(where by $\text{Stab}(\chi)$ we always mean $\text{Stab}_{N_V(U)}(\chi)$).
Since $\chi$ factors through $V/\text{Stab}(\chi)$, we can define
\begin{align*}
  C_\chi &\coloneqq \prod_{s\text{Stab}(\chi) \in V/\text{Stab}(\chi)} c_{\chi(s)} \in \mathbb{Z}\\
  \Theta_\chi &\coloneqq \prod_{s\text{Stab}(\chi) \in V/\text{Stab}(\chi)} \theta_{\chi(s)} \in \mathbb{Z}[X]\\
  Y_\chi &\coloneqq \bigotimes_{g_itU \in G/U} y_{g_iV}^{\chi(t)} \in Q^{\otimes_T G/U}
\end{align*}
and then (\ref{eq:congruence2}) simplifies to
\[\sum_{\chi \in \text{Set}(V/U, J)} \left[ \left( \sum_{vU \in (K \cap \text{Stab}(\chi))/U}
  C_\chi^{\abs{\text{Stab}(\chi):\lvUr}} \right)
\Theta_\chi^{\abs{\text{Stab}(\chi):U}} Y_\chi \right]
  \text{.}\]
This decomposes as the sum over $\chi$ in the set of orbits $\text{Set}(V/U, J)/{N_{K \cap V}(U)}$
of
\begin{align*} & \sum_{w \in
    N_{K \cap V}(U)/(K \cap \text{Stab}(\chi))} \left[ \left( \sum_{vU \in (K \cap \text{Stab}(w \cdot \chi))/U}
  C_{w \cdot \chi}^{\abs{\text{Stab}(w \cdot \chi):\lvUr}} \right) \Theta_{w \cdot
                 \chi}^{\abs{\text{Stab}(w \cdot \chi):U}} Y_{w \cdot \chi} \right] \\
               &= \sum_{w \in
                 N_{K \cap V}(U)/(K \cap \text{Stab}(\chi))} \left[ \left( \sum_{vU \in (K \cap \text{Stab}(\chi))/U}
                 C_{w \cdot \chi}^{\abs{\prescript{w}{}{\text{Stab}(\chi)}:\prescript{w}{}{\lvUr}}} \right) \Theta_{w \cdot
                 \chi}^{\abs{\prescript{w}{}{\text{Stab}(\chi)}:U}} Y_{w \cdot \chi} \right] \text{.}
\end{align*}
For $w \in N_V(U)$, we have the identities
\begin{align*}
C_{w \cdot \chi} &= C_\chi\\
\Theta_{w \cdot \chi} &= \Theta_\chi\\
Y_{w \cdot \chi} &= w \cdot Y_\chi %
\end{align*}
and (\ref{eq:congruence2}) becomes the sum over $\chi \in \text{Set}(V/U,
J)/{N_{K \cap V}(U)}$ of
\begin{align*}
  &\sum_{w \in
    N_{K \cap V}(U)/(K \cap \text{Stab}(\chi))} \left[ \left( \sum_{vU \in (K \cap \text{Stab}(\chi))/U}
    C_{\chi}^{\abs{\text{Stab}(\chi):\lvUr}} \right) \Theta_{\chi}^{\abs{\text{Stab}(\chi):U}} (w \cdot Y_{\chi}) \right] \\
  &= \text{tr}_{K \cap \text{Stab}(\chi)}^{N_{K \cap V}(U)} \left[
    \left( \sum_{vU \in (K \cap \text{Stab}(\chi))/U}
    C_{\chi}^{\abs{\text{Stab}(\chi):\lvUr}} \right) \Theta_{\chi}^{\abs{\text{Stab}(\chi):U}} Y_{\chi} \right] \text{.}
  \end{align*}
So it is enough to show that
\[\left( \sum_{vU \in (K \cap \text{Stab}(\chi))/U}
    C_{\chi}^{\abs{\text{Stab}(\chi):\lvUr}} \right)
  \Theta_{\chi}^{\abs{\text{Stab}(\chi):U}} Y_{\chi}\]
is in the image of $\text{tr}_U^{K \cap \text{Stab}(\chi)}$. But
$Y_\chi$ is fixed by $\text{Stab}(\chi)$, so it suffices to show that
\begin{equation}\label{eq:congruence3}
  \sum_{vU \in (K \cap \text{Stab}(\chi))/U}
  C_{\chi}^{\abs{\text{Stab}(\chi):\lvUr}} \equiv 0 \quad \text{mod } \abs{K \cap \text{Stab}(\chi):U} \text{.}
\end{equation}
We will prove that this congruence holds for any integer $C_{\chi}$. The Redfield-P\'{o}lya enumeration theorem (a
straightforward corollary of Burnside's lemma) shows that for a finite group $G$ acting
on a finite set $X$, we have %
\[\sum_{g \in G} m^{c(g)} \equiv 0 \quad \text{mod } \abs{G}\]
for any integer $m$, where $c(g)$ is the number of cycles of $g$ considered as a
permutation of $X$ (equivalently the number of orbits of the action on $X$ by the subgroup generated
by $g$). Applying this to $(K \cap \text{Stab}(\chi))/U$ acting on $\text{Stab}(\chi)/U$ by right multiplication gives (\ref{eq:congruence3}).
\end{proof}

We can now show that the image of the ghost map is precisely the set of elements that satisfy these congruences.

\begin{lemma}[Dwork lemma] \label{lem:dwork}
  Let $(T; Q)$ be free. For an element $a \in \text{gh}^S_{H \le G}(T; Q)$, the
  following are equivalent:
  \begin{enumerate}[(i)]
  \item the element $a$ is in the image of the ghost map
    \[w : \prod_{V \in \underline{S}} Q^{\otimes_T G/V} \to \text{gh}^S_{H \le
        G}(T; Q) \text{,}\]
  \item for all subgroups $U \in S$,
  \begin{equation}\label{eq:dwork1}\sum_{vU \in N_H(U)/U} \phi^{\lvUr}_U (a_{\lvUr})\end{equation}
  is in the image of $\text{tr}_U^{N_H(U)} : Q^{\otimes_T G/U} \to Q^{\otimes_T G/U}$.
  \end{enumerate}
\end{lemma}
\begin{proof}
Setting $K = H$ in Lemma \ref{lem:congruence} we see that any element in the image of the ghost map satisfies
condition (ii), so we just need to prove the converse. Suppose $a
\in \text{gh}^S_{H \le G}(T; Q)$ satisfies
(ii), then we want to construct $n \in \prod_{V \in \can{S}} Q^{\otimes_T G/V}$ such that
$w(n) = a$.

We will pick the components of $n = (n_W)_{W \in \can{S}}$ by induction on the index
of $W$ in $H$.
Let $k \in \mathbb{N}$, and suppose that for every distinguished subgroup $V \in
\can{S}$ with $\abs{H:V} < k$ we have already chosen a value for $n_V$. Moreover
suppose $w_V(n) = a_V$ for all such $V$ (note $w_V(n)$ is well-defined
since it only depends on those $n_{V'}$ where $V'$ has index lower or equal to
the index of $V$). We want to choose a value of $n_W$ for every distinguished
subgroup $W \in \can{S}$
with $\abs{H:W} = k$, such that $w_W(n) = a_W$.

Note that since the
image of the ghost map is fixed by
$H$, we automatically have $w_{U}(n) = a_{U}$ for all (not necessarily distinguished) $U \in S$ with $\abs{H:U} < k$.
Explicitly, if $V \in \can{S}$ with $\abs{H:V} < k$ then for any conjugate subgroup
$\prescript{h}{}{V}$ we have $w_{\prescript{h}{}{V}}(n) = h
\cdot w_V(n) = h \cdot a_V = a_{\prescript{h}{}{V}}$.

Let $W \in \can{S}$ be an index $k$ distinguished subgroup. The definition of the ghost component $w_W$ is
\[w_W(n) = \sum_{V \in \can{S}} \left( \sum_{hV \in (H/V)^W}
    h \cdot f_{G/V}(n_V^{\otimes_T V/W^h}) \right) \text{.}\]
The $V=W$ term of the sum simplifies to
\begin{align*}
  \sum_{hW \in (H/W)^W} h \cdot f_{G/W}(n_W^{\otimes_T W/W^h}) &= \sum_{hW \in N_H(W)/W} h \cdot n_W\\
                                                              &= \text{tr}_W^{N_H(W)}(n_W) \text{,}
\end{align*}
so if we define
\[r_W \coloneqq \sum_{V \in \can{S},\, V \ne W} \left( \sum_{hV \in (H/V)^W}
    h \cdot f_{G/V}(n_V^{\otimes_T V/W^h}) \right)\]
then we have
\[w_W(n) = \text{tr}_W^{N_H(W)}(n_W) + r_W\text{,}\]
where $r_W$ only depends on those values of $n_V$ that we have already chosen.
Since $\phi^W_W$ is the identity, splitting off the $vW = W$ term in Lemma
\ref{lem:congruence} (with $K = H$, $U = W$ and $n_W$ temporarily set to $0$) tells us that
\[r_W + \sum_{vW \in N_H(W)/W,\, vW \ne W} \phi^{\langle vW \rangle}_W (w_{\langle vW \rangle}(n))\]
is in the image of $\text{tr}_W^{N_H(W)}$. By condition (ii) we have that
\[a_W + \sum_{vW \in N_H(W)/W,\, vW \ne W} \phi^{\langle vW \rangle}_W (a_{\langle vW \rangle})\]
is in the image of $\text{tr}_W^{N_H(W)}$, and by assumption we have ensured
$w_{\langle vW \rangle}(n)
= a_{\langle vW \rangle}$ for $vW \ne W$. So $a_W - r_W$ must be in the image of $\text{tr}_W^{N_H(W)}$.
Choose $n_W$ such that $\text{tr}_W^{N_H(W)}(n_W) = a_W - r_W$, and
then $w_W(n) = \text{tr}_W^{N_H(W)}(n_W) + r_W = a_W$ as desired.

Similarly choose a value of $n_{W}$ for all other $W \in \can{S}$ with
$\abs{H : W} = k$, such
that $w_{W}(n) = a_{W}$; note we can do this simultaneously for potentially
infinitely many different $W$. Then the
induction holds. At the end of this inductive process we will have determined
$n$ such that $w_U(n) = a_U$ for all $U \in S$, i.e.\ $w(n) = a$.
\end{proof}

\begin{corollary}\label{cor:ghost_image_subgroup_indep}
  The image $\text{im}(w)$ is a closed subgroup of $\text{gh}^S_{H \le G}(T;
  Q)$, and it is independent of our choices of conjugacy
  class representatives $\can{S}$ or coset representatives for subgroups in $\can{S}$.
\end{corollary}
\begin{proof}
  Since $\phi^{\langle  vU \rangle}_U({-})$ is additive we see
  that the image of the ghost map is closed under subtraction, so a subgroup.
  For each $U \in S$, the expression (\ref{eq:dwork1}) only references finitely many
  coordinates of $a \in \text{gh}^S_{H \le G}(T; Q)$, so the intersection of the
  relevant conditions for all $U$ specifies a closed subset. These conditions don't depend on our choices of $\can{S}$ or coset representatives for $G/V$, so the image is independent of these choices.
\end{proof}

We also note that the surjection $w : \prod_{V \in \can{S}} Q^{\otimes_T G/V}
\xrightarrowdbl{} \text{im}(w)$ is a topological quotient map.

\begin{lemma} \label{lem:ghost_quotient}
  For $(T; Q)$ free, the map
  \[w : \prod_{V \in \can{S}} Q^{\otimes_T G/V} \xrightarrowdbl{} \text{im}(w)\]
  is open, hence a quotient map (where $\text{im}(w) \subseteq \text{gh}^S_{H \le G}(T; Q)$ has the
  subspace topology).
\end{lemma}
\begin{proof}
  Let $S'$ be a finite sub-truncation set of $S$ and take conjugacy class
  representatives $\can{S'} = S' \cap \can{S}$. Let $\pi_{S'} : \prod_{V \in \can{S}}
Q^{\otimes_T G/V} \to \prod_{V \in \can{S'}} Q^{\otimes_T G/V}$ be the
projection map. The
topology on $\prod_{V \in \can{S}} Q^{\otimes_T G/V}$ has a basis given by sets of the form
$\pi_{S'}^{-1}(n')$ for some finite $S'$ and some $n'\in \prod_{V \in \can{S'}} Q^{\otimes_T G/V}$. So it suffices to show that
the image $w(\pi_{S'}^{-1}(n'))$ of a basic open set is open in $\text{im}(w)$.
We will show that each element of the set $w(\pi_{S'}^{-1}(n'))$ has an open
neighbourhood contained in the set. Consider some $n \in
\pi_{S'}^{-1}(n')$ (that is, an element of the product that agrees with $n'$ at
the components indexed by $S'$). Let $\tilde{R}_{S'} :
\text{gh}^S_{H \le G}(T; Q) \to \text{gh}^{S'}_{H \le G}(T; Q)$ denote the
projection on ghost groups, so we get a commutative diagram
\[\begin{tikzcd}
    \prod_{V \in \can{S}} Q^{\otimes_T G/V} \ar[r,"w"] \ar[d, "\pi_{S'}"] & \text{gh}^S_{H \le G}(T; Q) \ar[d, "\tilde{R}_{S'}"]\\
    \prod_{V \in \can{S'}} Q^{\otimes_T G/V} \ar[r, "w"] & \text{gh}^{S'}_{H \le G}(T; Q) \text{.}
\end{tikzcd}\]
Then we claim that $\tilde{R}_{S'}^{-1}(w(n')) \cap \text{im}(w)$ is an open
neighbourhood of $w(n)$ contained in $w(\pi_{S'}^{-1}(n'))$. Certainly it is open (in the
subspace topology on $\text{im}(w)$) and contains $w(n)$. It remains to show that it
is contained in $w(\pi_{S'}^{-1}(n'))$. Given $a \in
\tilde{R}_{S'}^{-1}(w(n')) \cap \text{im}(w)$ we need to check that $a \in
w(\pi_{S'}^{-1}(n'))$. Since $a$ is in the image of $w$, it satisfies the
congruence conditions of the Dwork lemma, Lemma~\ref{lem:dwork}. But the proof
of the Dwork lemma constructs an element of $w^{-1}(a)$ inductively, in such
a way that we can start with $n'$ and extend it to a preimage of $a$.
Then this preimage is an element of $\pi_{S'}^{-1}(n')$ as desired.

Deduce that $w(\pi_{S'}^{-1}(n'))$ is open in $\text{im}(w)$, so $w
: \prod_{V \in \can{S}} Q^{\otimes_T G/V} \to \text{im}(w) \subseteq
\text{gh}^S_{H \le G}(T; Q)$ is an open map.
\end{proof}

While Lemma~\ref{lem:dwork} is most similar in form to Theorem~2.7.3 from
\cite{dress_burnside_1988}, it is also closely related to the original
Dwork lemma, as illustrated in the following remark.

\begin{lemma}
  Lemma~\ref{lem:dwork} generalises Lemma~\ref{lem:classical_dwork}, the Dwork
  lemma for the usual $p$-typical Witt vectors of commutative rings, in the case
  of a free commutative ring with the standard Frobenius lift.\footnote{Since our ghost
    map generalises the usual one, this lemma will obviously be true in the sense that
    the images of the ghost maps are the same. The point is to see how to show
    directly that the conditions in both Dwork lemmas are equivalent.}
\end{lemma}
\begin{proof}
We are interested in the case of Lemma~\ref{lem:dwork} where $T = Q = \mathbb{Z}[X]$, $H = G =
C_{p^n}$, and $S$ is the set of all subgroups. We want to show that this generalises
Lemma~\ref{lem:classical_dwork} applied to $W_{n+1, p}(T)$, with Frobenius
lift the ring homomorphism $\varphi_p : T \to T$ defined by $x \mapsto x^p$ for $x
\in X$.

Note $T^{\otimes_T r} \cong T$, and under this isomorphism
\[\phi_{C_{p^{n-k}}}^{C_{p^{n-k+l}}} : T^{\otimes_T C_{p^n}/C_{p^{n-k+l}}} \to
  T^{\otimes_T C_{p^n}/C_{p^{n-k}}}\]
becomes the ring homomorphism $\varphi_{p^l} : T \to T$ taking $x \mapsto x^{p^l}$.
Given $U = C_{p^{n-k}} \le C_{p^n}$,
observe that $\langle i + C_{p^{n-k}} \rangle \le C_{p^n}/C_{p^{n-k}}$ is equal to the subgroup $C_{p^{n-k+l}}/C_{p^{n-k}}$ for
$p^{l}-p^{l-1}$ different values of $0 \le i < p^k-1$ (for $1 \le l \le k$, and equals
$C_{p^{n-k}}/C_{p^{n-k}}$ once).

Consider some $a \in \text{gh}^S_{H \le G}(T; Q) \cong \prod_{0 \le i
  \le n} T$ (where the $i$th component is $T \cong T^{\otimes_T
  C_{p^n}/C_{p^{n-i}}}$). The conditions in
Lemma~\ref{lem:dwork} ask that
\begin{equation} \label{eq:dwork_equiv_classical} a_k + \sum_{l=1}^k (p^l-p^{l-1})\varphi_{p^l}(a_{k-l}) \in p^kT \text{,} \end{equation}
for all $0 \le k \le n$. We want to show that these are equivalent to the
conditions in the classical
Dwork lemma, which asks that $a_{k} \equiv \varphi_p(a_{k-1})$ mod $p^{k}T$ for $1
\le k \le n$.

When $k=0$ then (\ref{eq:dwork_equiv_classical}) just says $a_0 \in T$, which is
trivially always true. Now suppose (\ref{eq:dwork_equiv_classical}) holds for
some $k$ with $0 \le k < n$. Since $\varphi_p$ is a ring homomorphism, we have $\varphi_p(p^kT) \subseteq p^kT$. Also
note that $\varphi_p \circ \varphi_{p^l} = \varphi_{p^{l+1}}$. So applying $\varphi_p$ to
(\ref{eq:dwork_equiv_classical}) at $k$ tells us that
\[\varphi_p(a_k) + \sum_{l=1}^k (p^l - p^{l-1})\varphi_{p^{l+1}}(a_{k-l}) \in p^k T \text{.}\]
But condition (\ref{eq:dwork_equiv_classical}) at $k+1$ asks that
\begin{align*}& a_{k+1} + \sum_{l=1}^{k+1} (p^l - p^{l-1})\varphi_{p^l}(a_{k+1-l})\\ &= a_{k+1} -
  \varphi_p(a_{k}) + p \left( \varphi_p(a_{k}) + \sum_{l=1}^k (p^l - p^{l-1})
    \varphi_{p^{l+1}}(a_{k-l}) \right) \in p^{k+1}T\end{align*}
so holds iff $a_{k+1} \equiv \varphi_p(a_k)$ mod $p^{k+1}T$.
By induction on $k$, we see that Lemma~\ref{lem:classical_dwork} and
Lemma~\ref{lem:dwork} are equivalent in this special case.
\end{proof}

We can actually get a rather concrete understanding of the subgroup $\text{im}(w)
\le \text{gh}^S_{H \le G}(T; Q)$. In particular when $S$ is finite then
$\text{im}(w)$ is a subgroup of a free abelian group, so free abelian, and the
following lemma will allow us to write down a basis. The idea is to consider a
variant of the ghost map where we replace
the tensor power $f_{G/V}(({-})^{\otimes_T V/U})$ with the Frobenius lift
$\phi^V_U$, giving an additive map.

\begin{lemma} \label{lem:witt_free_description}
  Let $T = \mathbb{Z}[X]$, $Q = T(Y)$. Define a continuous additive map
  \[w^f : \left(\prod_{U \in S} Q^{\otimes_T G/U}\right)_H \to \left(\prod_{U \in S}
    Q^{\otimes_T G/U}\right)^H = \text{gh}^S_{H \le G}(T; Q)\]
via
\[w^f_U(n) = \sum_{W \in S} \sum_{hW \in (H/W)^U} h \cdot \phi^W_{U^h}(n_W) \text{.}\]
The map $w^f$ is an embedding, with image equal to the image of the ghost
map $w : \prod_{V \in \can{S}}
Q^{\otimes_T G/V} \to \text{gh}^S_{H \le G}(T; Q)$.

We can also write
\[\left(\prod_{U \in S} Q^{\otimes_T G/U}\right)_H \cong \prod_{V
    \in \can{S}} ( Q^{\otimes_T G/V} )_{N_H(V)}\text{,}\]
and under this isomorphism we see that for $n \in \prod_{V \in
  \can{S}} (Q^{\otimes_T G/V})_{N_H(V)}$ we have $w^f_U(n) = \sum_{V \in \can{S}}
\sum_{hV \in (H/V)^U} h \cdot \phi^V_{U^h}(n_V)$. Note $(Q^{\otimes_T
  G/V})_{N_H(V)} \cong T((Y^{\times G/V})_{N_H(V)})$ is a free abelian group,
so when $S$ is finite this lets us write down a basis of $\text{im}(w)$.
\end{lemma}
\begin{proof}
  It is straightforward to check that $w^f$ is well-defined as a map out of the
  group of $H$-orbits. Next we show that $\text{im}(w^f) \subseteq
  \text{im}(w)$. By Lemma~\ref{lem:dwork} it suffices to show that for all $U
  \in S$ and $n \in \left( \prod_{U \in S} Q^{\otimes_T G/U} \right)_H$,
  \begin{equation}\sum_{vU \in N_H(U)/U} \phi^{\langle vU \rangle}_U(w^f_{\langle vU \rangle}(n))\label{eq:congruence4}\end{equation}
  is in the image of $\text{tr}^{N_H(U)}_U$. We can prove this by manipulation
  very similar to the start of the proof of Lemma~\ref{lem:congruence}.
  Expanding out and interchanging summation gives
  \begin{align*}\sum_{vU \in N_H(U)/U} \phi^{\langle vU \rangle}_U(w^f_{\langle vU \rangle}(n)) &= \sum_{vU \in N_H(U)/U} \, \sum_{W \in S} \, \sum_{hW \in (H/W)^{\langle vU \rangle}} h \cdot \phi^{{\langle vU \rangle}^h}_{U^h}\phi^W_{\langle vU \rangle^h}(n_W)\\
    &= \sum_{W \in S} \, \sum_{hW \in (H/W)^U} \, \sum_{vU \in N_{H \cap
      \prescript{h}{}{W}}(U)/U} h \cdot \phi^W_{U^h}(n_W)\end{align*}
and decomposing into $N_H(U)$-orbits shows that (\ref{eq:congruence4}) equals
  \[\sum_{W \in S} \sum_{hW \in (H/W)^U/{N_H(U)}} \text{tr}^{N_H(U)}_{N_{H \cap \prescript{h}{}{W}}(U)} \left(\sum_{vU \in N_{H \cap \prescript{h}{}{W}}(U)/U} h \cdot \phi^W_{U^h}(n_W) \right) \text{.}\]
  So it suffices to prove that
  \[\sum_{vU \in N_{H \cap \prescript{h}{}{W}}(U)/U} h \cdot \phi^W_{U^h}(n_W)\]
  is in the image of $\text{tr}_U^{N_{H \cap \prescript{h}{}{W}}(U)}$, which is
  true since $h \cdot \phi^W_{U^h}(n_W) = \phi^{\prescript{h}{}{W}}_U(h \cdot
  n_W)$ is fixed by $N_{H  \cap \prescript{h}{}{W}}(U)$
  (the map $\phi^{\prescript{h}{}{W}}_U$ commutes with the action of $N_{H \cap
    \prescript{h}{}{W}}(U)$, and $h \cdot n_W \in
  Q^{\otimes_T G/\prescript{h}{}{W}}$ so is fixed by $N_{H \cap
    \prescript{h}{}{W}}(U) \le \prescript{h}{}{W}$)

  Now we can show that in fact $\text{im}(w^f) = \text{im}(w)$. This follows by
  essentially exactly the same proof as Lemma~\ref{lem:dwork}. For $n \in
  \prod_{V \in \can{S}} ( Q^{\otimes_T G/V} )_{N_H(V)}$ we have
  \begin{align*}w^f_W(n) &= \sum_{V \in \can{S}} \sum_{hV \in (H/V)^W} h \cdot \phi^V_{W^h}(n_V)\\
    &= \text{tr}^{N_H(W)}_{W}(n_W) + \sum_{V \in \can{S}, V \ne W} \left( \sum_{hV \in (H/V)^W} h \cdot \phi^V_{W^h}(n_V) \right) \text{,}
  \end{align*}
  so $w_W^f(n)$ is the sum of $\text{tr}^{N_H(W)}_W(n_W)$ and a term that only
  depends on $n_V$ for $V$ of smaller index than $W$. The approach of
  Lemma~\ref{lem:dwork} shows that given any element of $\text{gh}^S_{H \le
    G}(T; Q)$ satisfying the Dwork congruences (i.e.\ any element in
  $\text{im}(w)$), we can inductively construct a preimage under $w^f$. We
  conclude that $\text{im}(w^f) = \text{im}(w)$.

  The homomorphism $w^f$ is injective. Indeed suppose $n \in \prod_{V \in
    \can{S}} \left( Q^{\otimes_T G/V} \right)_{N_H(V)}$ is non-zero. Let $W \in
  \can{S}$ be of minimal index in $H$ such that $n_W \ne 0 \in (Q^{\otimes_T
      G/V})_{N_H(V)}$. Then $w^f_W(n) = \text{tr}^{N_H(W)}_W(n_W) \ne 0$,
  since $\text{tr}^{N_H(W)}_W : (Q^{\otimes_T G/W})_{N_H(W)} \to (Q^{\otimes_T
    G/W})^{N_H(W)}$ is injective.

  Now we know that $w^f : \prod_{V \in \can{S}} (Q^{\otimes_T G/V})_{N_H(V)} \to
  \text{im}(w)$ is a continuous additive bijection. Finally, exactly the same
  proof as Lemma~\ref{lem:ghost_quotient} shows that it is an open map so we
  conclude that it is an embedding.
\end{proof}

\begin{remark} \label{rem:witt_free_choices}
  Note that the Frobenius lift $\phi^V_U$ depends on the choice of generators of
  $T$ and $Q$, so the isomorphism $w^f : \left(\prod_{U \in S} Q^{\otimes_T
      G/V}\right)_H \cong \text{im}(w)$ is not natural with respect to general maps
  between free objects of $\text{Mod}$. However unlike the usual ghost map $w$,
  the map $w^f$ (with the domain described in this way) does not depend on any
  choice of distinguished subgroups or coset representatives.
\end{remark}

\subsection{Extension from free modules}\label{sec:witt_final_definition}

We have now nearly finished the definition of the group of Witt vectors. Using the Dwork lemma, we can define the Witt vectors
with free coefficients $W^S_{H \le G}(T; Q)$ to be the image of the ghost map $\text{im}(w) \le
\text{gh}^S_{H \le G}(T; Q)$.
It remains to show that this uniquely extends to a reflexive
coequaliser-preserving functor $W^S_{H \le G} : \text{Mod} \to \text{Ab}_\text{Haus}$.

Let $(R; M) \in \text{Mod}$, and suppose we have
\[(\overline{T}; \overline{Q}) \, \substack{\xrightarrow{f}\\[-0.2em]
    \xleftarrow{}\\[-0.2em] \xrightarrow[g]{}} \, (T; Q) \xrightarrowdbl{\epsilon} (R; M)\]
a reflexive coequaliser diagram where $(\overline{T}; \overline{Q})$ and $(T;
Q)$ are free. We call this a free resolution of $(R; M)$. Every object of
$\text{Mod}$ has a canonical free resolution originating from the free-forgetful
adjunction:
\[FUFU(R; M) \,\substack{\xrightarrow{} \\[-0.2em] \xleftarrow{}\\[-0.2em] \xrightarrow{}} \, FU(R; M) \xrightarrowdbl{} (R; M) \text{.}\]
Let $\text{Mod}_F$ be the full subcategory of $\text{Mod}$ spanned by the free
objects.
The existence of free resolutions shows that
reflexive coequaliser-preserving functors on $\text{Mod}$ are uniquely defined
by their values on $\text{Mod}_F$. In the next couple of lemmas we will show that
in fact any functor $\text{Mod}_F \to \mathcal{C}$ extends uniquely to a reflexive
coequaliser-preserving functor $\text{Mod} \to \mathcal{C}$, as long as the category
$\mathcal{C}$ has all reflexive coequalisers. First we check that any functor
out of $\text{Mod}_F$ must already preserves those reflexive coequalisers
consisting of free objects.

\begin{lemma} \label{lem:refl_coeq_free_split}
  A free resolution of a free object in $\text{Mod}$ is split.
\end{lemma}
\begin{proof}
  \newcommand{\ol}[1]{\overline{#1}}
  \newcommand{\mc}{(T; Q)}
  \newcommand{\mb}{(T_0; Q_0)}
  \newcommand{\ma}{(T_1; Q_1)}

  Suppose we have a reflexive coequaliser diagram of free objects
  \[\ma \, \substack{\xrightarrow{f}\\[-0.2em]
      \xleftarrow{}\\[-0.2em] \xrightarrow[g]{}} \, \mb
    \xrightarrowdbl{\epsilon} \mc \text{.}\]
  Denote the common section of $f$ and $g$ by $s : \mb
  \to \ma$. We want to show this is a split coequaliser, by defining maps $u : \mc \to \mb$ and $t : \mb \to \ma$ such
  that $\epsilon u = 1_{\mc}$, $u \epsilon = gt$ and $ft = 1_{\mb}$.

  Recall from Section~\ref{sec:cat_of_modules} that we have a free-forgetful
  adjunction between $\text{Mod}$ and $\text{Set} \times \text{Set}$. If $\mc$
  is the free object on a pair of sets $(X, Y)$ then defining a map $\mc \to \mb$
  is equivalent to defining a map of pairs of sets $(X, Y) \to U \mb$.

  So since $\epsilon$ is surjective (on underlying sets of rings and modules)
  we can choose a map $u : \mc \to \mb$ such that $\epsilon u = 1_{\mc}$, by
  sending each generator to a preimage under $\epsilon$.

  Next we can define $t$. We start by defining the ring component of $t$,
  via showing where to send each generator $x \in T_0$.
  Observe that $\epsilon(u(\epsilon(x))) = \epsilon(x)$, i.e.\ $u(\epsilon(x))$ and $x$ are
  identified by the reflexive coequaliser quotient $T_0 \xrightarrowdbl{\epsilon} T$.
  So there exists some $x' \in T_1$ such that $g(x') = u(\epsilon(x))$ and
  $f(x') = x$. Define the ring component of $t$ to send $x$ to $x'$, and
  similarly for the other generators of the free ring $T_0$. Then the ring
  component of $t$ satisfies $u \epsilon = gt$ and $ft = 1_{T_0}$. We can use
  exactly the same argument to define the module component of $t$.
\end{proof}

In the following lemma we prove that we can use free resolutions to uniquely extend a functor $G: \text{Mod}_F \to
\mathcal{C}$ to a reflexive
coequaliser-preserving functor $\hat{G} : \text{Mod} \to \mathcal{C}$.
Moreover this is part of an adjunction: given any functor $H : \text{Mod} \to \mathcal{C}$, natural
transformations from $G$ to the restriction of $H$ to free objects
are in bijection with natural
transformations from $\hat{G}$ to $H$.

\begin{lemma}\label{lem:finite_module_kan}
  Let $\iota : \text{Mod}_F \to \text{Mod}$ be the
  inclusion functor, and $\mathcal{C}$ a category that admits reflexive
  coequalisers. Then the left Kan extension
  \[\text{Lan}_\iota : \text{Fun}(\text{Mod}_F, \mathcal{C}) \to
    \text{Fun}(\text{Mod}, \mathcal{C})\]
  exists (that is, the restriction functor $\iota^\ast : \text{Fun}(\text{Mod},
  \mathcal{C}) \to \text{Fun}(\text{Mod}_F, \mathcal{C})$ has a left adjoint). The functor $\text{Lan}_\iota$ is full and faithful, and has essential image the full subcategory $\text{Fun}_{rc}(\text{Mod}, \mathcal{C})$ of reflexive
  coequaliser-preserving functors $\text{Mod} \to \mathcal{C}$. This
  exhibits $\text{Fun}_{rc}(\text{Mod}, \mathcal{C})$ as a coreflective
  subcategory of $\text{Fun}(\text{Mod}, \mathcal{C})$.
\end{lemma}
\begin{proof}
  We will construct $\text{Lan}_\iota : \text{Fun}(\text{Mod}_F, \mathcal{C})
  \to \text{Fun}(\text{Mod}, \mathcal{C})$ explicitly, then show that it has the
  desired properties.

  Let $G :
  \text{Mod}_F \to \mathcal{C}$. Define $\hat{G} : \text{Mod} \to \mathcal{C}$ via
  \[\hat{G}(R; M) = \text{coeq}(GFUFU(R; M) \,\substack{\xrightarrow{} \\[-0.2em] \xrightarrow{}} \, GFU(R; M)) \text{.}\]

  Since split coequalisers are absolute, the previous lemma shows that
  applying $G$ to a free resolution of a free object gives a (split) coequaliser. So if $(T; Q)$ is free then
  \[GFUFU(T; Q) \,\substack{\xrightarrow{} \\[-0.2em]
      \xrightarrow{}} \, GFU(T; Q) \to G(T; Q)\]
  is a coequaliser diagram, and hence $\hat{G}(T; Q)$ is naturally
  isomorphic to $G(T; Q)$. So $\hat{G}$ really is an extension of $G$ (up to
  isomorphism, or we may choose coequalisers such that it is an extension on the
  nose).

  Next we show that $\hat{G}$ preserves all reflexive coequalisers. Suppose that
  \[(R_1; M_1) \, \substack{\xrightarrow{f\,}\\[-0.2em]
      \xleftarrow{}\\[-0.2em] \xrightarrow[g\,]{}} \, (R_0; M_0) \xrightarrowdbl{\epsilon\,} (R; M)\]
  is a reflexive coequaliser. Using the canonical free resolutions, we
  have a diagram
  \[\begin{tikzcd}
      FUFU(R_1; M_1) \ar[r, shift left] \ar[r, shift right] \ar[d, shift left] \ar[d,
      shift right]& FUFU(R_0; M_0) \ar[d, shift left] \ar[d, shift right] \ar[r]&
      FUFU(R; M) \ar[d, shift left] \ar[d, shift right]\\
      FU(R_1; M_1) \ar[r, shift left] \ar[r, shift right] \ar[d] & FU(R_0; M_0) \ar[r]
      \ar[d]& FU(R; M) \ar[d]\\
      (R_1; M_1) \ar[r, shift left] \ar[r, shift right] & (R_0; M_0) \ar[r] & (R; M) \text{.}
    \end{tikzcd}\]
  We know that the forgetful functor $U$ preserves reflexive coequalisers, and
  $F$ is a left adjoint, so $FU$ preserves reflexive coequalisers. Hence all the
  rows and columns of the diagram are reflexive coequalisers. The top two rows
  consist of free objects, so are split coequalisers.

  Apply $\hat{G}$ to the diagram. Since the restriction of $\hat{G}$ to free objects is
  naturally isomorphic to $G$, we get a diagram
  \[\begin{tikzcd}
      GFUFU(R_1; M_1) \ar[r, shift left] \ar[r, shift right] \ar[d, shift left] \ar[d,
      shift right]& GFUFU(R_0; M_0) \ar[d, shift left] \ar[d, shift right] \ar[r]&
      GFUFU(R; M) \ar[d, shift left] \ar[d, shift right]\\
      GFU(R_1; M_1) \ar[r, shift left] \ar[r, shift right] \ar[d] & GFU(R_0; M_0) \ar[r]
      \ar[d]& GFU(R; M) \ar[d]\\
      \hat{G}(R_1; M_1) \ar[r, shift left] \ar[r, shift right] & \hat{G}(R_0; M_0) \ar[r] & \hat{G}(R; M) \text{.}
    \end{tikzcd}\]
  The columns are coequalisers by the definition of $\hat{G}$.
  The top two rows are split coequalisers. Since colimits commute with colimits, we deduce that the
  bottom row is a coequaliser as desired.

  So $G$ has a reflexive coequaliser-preserving extension $\hat{G} \in \text{Fun}_{rc}(\text{Mod}, \mathcal{C})$.
  Define $\text{Lan}_{\iota} : \text{Fun}(\text{Mod}_F, \mathcal{C}) \to
  \text{Fun}(\text{Mod}, \mathcal{C})$ on objects by $G \mapsto \hat{G}$, and on morphisms via the
  canonically induced maps between coequalisers. We claim that this is the left
  Kan extension---that is, it is left adjoint to the functor
  $\iota^\ast : \text{Fun}(\text{Mod}, \mathcal{C}) \to \text{Fun}(\text{Mod}_F,
  \mathcal{C})$. We have already shown that the identity is canonically
  isomorphic to $\iota^\ast \text{Lan}_\iota$; this is the unit of the
  adjunction. Given $H : \text{Mod} \to \mathcal{C}$ and $(R; M) \in
  \text{Mod}$, the universal property of the coequaliser
  \[(\text{Lan}_\iota \iota^\ast H)(R; M) = \text{coeq}(HFUFU(R; M)
    \,\substack{\xrightarrow{} \\[-0.2em]
      \xrightarrow{}} \, HFU(R; M))\]
  gives a factorisation of $HFU(R; M) \to H(R; M)$ through the map $HFU(R; M)
  \to (\text{Lan}_\iota \iota^\ast H)(R; M)$; the collection of resulting maps
  $(\text{Lan}_\iota \iota^\ast H)(R; M) \to H(R; M)$ gives
  the counit of the adjunction. It is straightforward to check the triangle
  identities, verifying that we have an adjunction $\text{Lan}_\iota \dashv \iota^\ast$.

  Since the unit is a natural isomorphism, $\text{Lan}_\iota$ is full and faithful.
  We have seen that every functor in the image of
  $\text{Lan}_\iota$ preserves reflexive coequalisers. But also the counit is
  clearly an isomorphism at any $H : \text{Mod} \to \mathcal{C}$ that preserves reflexive coequalisers,
  so the essential image of $\text{Lan}_\iota$ is precisely the full subcategory
  $\text{Fun}_{rc}(\text{Mod}, \mathcal{C})$ of reflexive coequaliser-preserving
  functors. This makes $\text{Fun}_{rc}(\text{Mod}, \mathcal{C})$ into a
  coreflective subcategory of $\text{Fun}(\text{Mod}, \mathcal{C})$.
\end{proof}

\begin{remark} \label{rem:finite_module_kan}
  An analogous statement holds if instead of $\text{Mod}$ and $\text{Mod}_F$ we
  consider the category $\text{Ab}$ of abelian groups and the full subcategory
  $\text{Ab}_F$ of free abelian groups, or the categories $\text{CRing}$ and
  $\text{CRing}_F$ of commutative rings and free commutative rings. The proofs use exactly the same
  ideas. In particular free resolutions of free abelian groups or free
  commutative rings are split; this can be proved analogously to
  Lemma~\ref{lem:refl_coeq_free_split} or deduced from it.
\end{remark}

This was the last ingredient we need for the proof of Theorem~\ref{thm:witt_properties},
the uniqueness theorem for
the $S$-truncated $G$-typical Witt vectors with coefficients.

\begin{definition}[$S$-truncated $G$-typical Witt vectors with coefficients] \label{def:witt}
  Define $W^S_{H \le G} : \text{Mod} \to \text{Ab}_\text{Haus}$ to be a
  reflexive coequaliser-preserving extension of the functor $\text{Mod}_F \to
  \text{Ab}_\text{Haus}$ given by
  \[(T; Q) \mapsto \text{im}\left(w : \prod_{V \in \can{S}} Q^{\otimes_T G/V} \to
    \text{gh}^S_{H \le G}(T; Q)\right) \text{.}\]
As noted in the proof of Lemma~\ref{lem:finite_module_kan} we may choose $W^S_{H \le G}$ such that it
is genuinely an on-the-nose extension, so we have $\iota^\ast W^S_{H \le G}(T;
Q) = \text{im}(w)$.
\end{definition}

\begin{remark}
  We write $W^S_G(R; M)$ (omitting the subgroup $H \le G$) as shorthand for $W^S_{G \le G}(R; M)$.
  We
  write $W_{H \le G}(R; M)$ (omitting the truncation set $S$) to mean the
  untruncated Witt vectors, i.e.\ $W^{S_0}_{H \le G}(R; M)$ where $S_0$ is the set of all open subgroups of
  $H$.
\end{remark}

We saw in Corollary~\ref{cor:ghost_image_subgroup_indep} that $\text{im}(w) \le
\text{gh}^S_{H \le G}(T; Q)$ is a closed subgroup, and
Lemma~\ref{lem:ghost_quotient} showed that the surjection $w : \prod_{V \in
  \can{S}} Q^{\otimes_T G/V} \xrightarrowdbl{} \text{im}(w) = W^S_{H \le G}(T; Q)$ is a quotient map.
Since tensor powers and products preserve reflexive coequalisers (Lemmas~\ref{lem:tensorcoequaliser} and \ref{lem:ab_product_preserves_reflexive_coequalisers}) we see that the functor
\[(R; M) \mapsto \prod_{V \in \can{S}} M^{\otimes_R G/V}\]
preserves reflexive coequalisers. This lets us define the Witt vector quotient map:

\begin{definition}[Witt vector quotient map] \label{def:witt_quotient}
  We define the natural map of underlying spaces $q : \prod_{V \in \can{S}} M^{\otimes_R G/V}
  \xrightarrowdbl{} W^S_{H \le G}(R; M)$ to be the extension of the quotient
  $\prod_{V \in \can{S}} Q^{\otimes_T G/V} \xrightarrowdbl{w} \text{im}(w) =
  \iota^\ast W^S_{H \le G}(T; Q)$ defined for free coefficients.
  Since reflexive coequalisers preserve quotients
  (Lemma~\ref{lem:space_coeq_quotient}), $q$ is a topological quotient map.
\end{definition}

And we can show that the ghost map factors through this quotient, giving us an
additive ghost map out of the Witt vectors:

\begin{definition}[Witt vector ghost map] \label{def:witt_ghost}
  The ghost map factorises as
\[\prod_{V \in \can{S}} Q^{\otimes_T G/V} \xrightarrowdbl{q} \iota^\ast W^S_{H \le G}(T;
  Q) = \text{im}(w) \hookrightarrow \iota^\ast \text{gh}^S_{H \le G}(T; Q)\]
for $(T; Q)$ free, where the
  inclusion map is additive. Recall that the adjunction proved in Lemma~\ref{lem:finite_module_kan} shows that for any functors $G : \text{Mod}_F \to
  \mathcal{C}$ and $H : \text{Mod} \to \mathcal{C}$, natural transformations
  from $G$ to the restriction of $H$ to free objects are in bijection with
  natural transformations from the reflexive coequaliser-preserving extension
  $\hat{G}$ to $H$; moreover this bijection is given by extending the natural
  transformation in the obvious way (consider the definition in terms of
  applying the left Kan extension then postcomposing the counit).
  So the natural transformation $\iota^\ast
  W^S_{H \le G}(T; Q) \hookrightarrow \iota^\ast \text{gh}^S_{H \le G}(T; Q)$ (of
  $\text{Ab}_{\text{Haus}}$-valued functors)
  uniquely extends to a natural transformation  $W^S_{H \le G}(R; M) \to
  \text{gh}^S_{H \le G}(R; M)$.
  The composition
  \[\prod_{V \in \can{S}} M^{\otimes_R G/V} \xrightarrowdbl{q} W^S_{H \le G}(R; M) \to \text{gh}^S_{H \le G}(R; M)\]
  gives a natural transformation of $\text{Top}_{\text{Haus}}$-valued functors
  that matches the ghost map $w : \prod_{V
    \in \can{S}} M^{\otimes_R G/V} \to \text{gh}^S_{H \le G}(R; M)$ for $(R; M)$
  free; but in fact by Lemma~\ref{lem:finite_module_kan} the
  extension from free objects is unique, so this composition must match the ghost map for
  all $(R; M)$.

  In summary, the ghost map descends to an additive map out of
  the quotient $\prod_{V \in \can{S}} M^{\otimes_R G/V} \xrightarrowdbl{q}
  W^S_{H \le G}(R; M)$. Following \cite{dotto_witt_2025} we will also refer to this map $W^S_{H \le G}(R; M) \to \text{gh}^S_{H \le
    G}(R; M)$ as the ghost map, and denote it by $w$.
  Whether we mean this map or the map $w : \prod_{V \in
    \can{S}} M^{\otimes_R G/V} \to \text{gh}^S_{H \le G}(R; M)$ should be clear from context.
\end{definition}

\begin{remark}
  The Dwork lemma shows that the image of the ghost map is independent of the
  choices of $\can{S}$ and coset representatives that we made, so both the Witt vectors themselves and the ghost map $w : W^S_{H \le G}(R; M) \to \text{gh}^S_{H
    \le G}(R; M)$ don't depend
  on these choices. However the expression of the underlying space of $W^S_{H \le G}(R; M)$ as a
  quotient via the map $q$ does depend on the arbitrary choices.
\end{remark}

Unwinding some definitions, we can write this quotient
more explicitly as follows
(analogously to how the Witt vectors are defined in \cite{dotto_witt_2025}
Definition~1.3): %

\begin{remark}
  Let $(R; M) \in \text{Mod}$, and
  \[(\overline{T}; \overline{Q}) \, \substack{\xrightarrow{f}\\[-0.2em]
      \xleftarrow{}\\[-0.2em] \xrightarrow[g]{}} \, (T; Q) \xrightarrowdbl{\epsilon} (R; M)\]
  a free resolution of $(R; M)$ (that is, a reflexive coequaliser diagram where
  $(\overline{T}; \overline{Q})$ and $(T; Q)$ are free). Define an equivalence
  relation $\sim$ on $\prod_{V \in \can{S}} M^{\otimes_R G/V}$ by
  $a \sim b$ if there exists $z \in \prod_{V \in \can{S}}
  \overline{Q}^{\otimes_R G/V}$ and $q, u \in \prod_{V \in \can{S}}
  Q^{\otimes_R G/V}$ such that
  \begin{align*}
    a = \epsilon_\ast(q) &\quad b = \epsilon_\ast(u)\\
    f_\ast(w(z)) = w(q) &\quad g_\ast(w(z)) = w(u) \text{.}
  \end{align*}
  Then the underlying topological space of the group of Witt vectors is
  \[W^S_{H \le G}(R; M) \cong \bigg( \prod_{V \in \can{S}} M^{\otimes_R G/V} \bigg)
    \!\raisebox{-.5em}{$\Big/$} \raisebox{-.65em}{$\sim$} \, \text{.}\]
\end{remark}

Our primary approach for proving identities involving the Witt vectors will be
to show that they hold for free coefficients, and then use
Lemma~\ref{lem:finite_module_kan} to show that they in fact hold in general. As
a demonstration of this, we show that we can compute $W^S_{H \le G}$ in terms
of $W^S_{H} \coloneqq W^S_{H \le H}$.

\begin{lemma} \label{lem:HGHHiso}
  A choice of coset representatives for $G/H$ gives us a
  natural isomorphism
  \[W^S_{H \le G}(R; M) \cong W^S_{H \le H}(R; M^{\otimes_R G/H}) \text{.}\]
\end{lemma}
\begin{proof}
  The choice of coset representatives induces
  isomorphisms
  \[f_{G/H} : M^{\otimes_R G/H \times H/U} \cong M^{\otimes_R G/U}\text{.}\]
  The product of these isomorphisms gives an $H$-equivariant isomorphism
  \[\prod_{U \in
    S} M^{\otimes_R G/H \times H/U} \to  \prod_{U \in S} M^{\otimes_R
    G/U}\text{,}\]
and
  restricting to $H$-fixed points gives a natural isomorphism
  \[\theta : \text{gh}^S_{H \le H}(R; M^{\otimes_R G/H})
    \to \text{gh}^S_{H \le G}(R; M) \text{.}\]

  Let $P : \text{Mod} \to \text{Mod}$ be the functor $(R; M) \mapsto (R;
  M^{\otimes_R G/H})$. Composing the ghost map and our isomorphism gives a
  natural transformation $\theta w_P : W^S_{H \le H} P \Rightarrow \text{gh}^S_{H
    \le G}$.
  It's straightforward to check that
  when $(T; Q)$ is free, the isomorphism $\theta$ sends the image
  of the ghost map $W^S_{H \le H}(T; Q^{\otimes_T G/H}) \to \text{gh}^S_{H \le
    H}(T; Q^{\otimes_T G/H})$ to the image of the ghost map $W^S_{H \le G}(T; Q)
  \to \text{gh}^S_{H \le G}(T; Q)$: use the characterisation of the image in
  Lemma~\ref{lem:dwork}, and observe that the Frobenius lifts
  $\phi^V_U$ and the transfer both commute with the isomorphisms $f_{G/H}$.
  Since the ghost map is injective for free coefficients, we get
 a natural transformation $\iota^\ast
  (W^S_{H \le H} P) \Rightarrow \iota^\ast W^S_{H \le G}$.
  Lemma~\ref{lem:tensorcoequaliser} shows that $W^S_{H \le H}P$ preserves
  reflexive coequalisers. So by Lemma~\ref{lem:finite_module_kan} our
  natural transformation of functors defined on free objects extends uniquely to a natural transformation
  \[\Theta : W^S_{H \le H}P \Rightarrow W^S_{H \le G}\]
  lifting $\theta$ along the ghost maps.

  Similarly we can lift $\theta^{-1}$ along the ghost maps to get a natural transformation
  \[\Theta' : W^S_{H \le G} \Rightarrow W^S_{H \le H}P \text{.}\]
  The compositions $\Theta \Theta'$ and $\Theta' \Theta$
  are both the identity at free objects, so by uniqueness of extension they must be the identity at all objects.
  So $\Theta$ and $\Theta'$ specify a natural isomorphism $W^S_{H \le H}(R; M^{\otimes_R
    G/H}) \cong W^S_{H \le G}(R; M)$.
\end{proof}
\begin{remark}
  In the setting of \cite{dotto_witt_2022} where $G$ is cyclic (or procyclic),
  there are obvious choices of coset representatives for $G/H$.
  Implicitly using the corresponding isomorphisms allows the authors to only work with $W_G$.
  However it turns out that the Witt vector operators are more naturally defined between the
  $W_{H \le G}$ (for varying $H$), and so in our setting with no canonical choice
  of coset representatives available to us it will be easier to work with the $W_{H \le
  G}$ directly, at the cost of complicating the notation.
\end{remark}

\subsection{Initial computations}

When $(T; Q)$ is free we have a complete description of the Witt vectors.

\begin{proposition} \label{prop:witt_free}
   For $(T, Q)$ free we have a (not natural) isomorphism of topological abelian groups
  \[W^S_{H \le G}(T; Q) \cong \text{im}(w) \cong \left( \prod_{U \in S} Q^{\otimes_T G/V}
    \right)_H \cong \prod_{V \in \can{S}} (Q^{\otimes_T G/V})_{N_H(V)}\text{.}\]
\end{proposition}
\begin{proof}
  This is an immediate consequence of Lemma~\ref{lem:witt_free_description}.
\end{proof}

\begin{remark}
  This isomorphism can be thought of as analogous to tom Dieck splitting for fixed points of
  equivariant suspension spectra (and indeed when $T = \mathbb{Z}$ it precisely
  corresponds to using tom
  Dieck splitting to compute the zeroth equivariant stable homotopy groups of the norm of a
  suspension spectrum).
\end{remark}

\begin{remark}
  This generalises the isomorphism of \cite{dotto_witt_2025} Corollary~A.9 in
  the case of free coefficients equipped with an external Frobenius
  defined analogously to the Frobenius lift in this paper. Note this is a
  different isomorphism to the isomorphism of abelian
  groups in Proposition~1.14 of \cite{dotto_witt_2025}. The appropriate
  generalisation of that isomorphism will be given in Lemma~\ref{lem:witt_ab_group_free_coeff}.
\end{remark}

For general coefficients it can be hard to explicitly describe either the underlying space or the additive structure of the Witt vectors. However we can analyse some special cases.

\begin{lemma} \label{lem:witt_trivial_cases}
  We have
  \[W^\emptyset_{H \le G}(R; M) = 0\]
  and
  \[W^{\{H\}}_{H \le G}(R; M) \cong M^{\otimes_R G/H}\]
  (as topological abelian groups).
\end{lemma}
\begin{proof}
  When the truncation set is empty then $\iota^\ast W^\emptyset_{H \le G}(T; Q)$
  is trivial.
  The unique reflexive coequaliser-preserving extension
  $W^\emptyset_{H \le G}({-}; {-})$ is the constant functor to the trivial group.

  In the case $S = \{H\}$ then the ghost map is just the identity
  \[M^{\otimes_R G/H} \xrightarrow{\text{id}} M^{\otimes_R G/H} = \left( M^{\otimes_R
        G/H}\right)^H\text{.}\]
  Hence for $(T; Q)$ free we have
  $\iota^\ast W^{\{H\}}_{H \le G}(T; Q) = Q^{\otimes_T G/H}$ and (using Lemma~\ref{lem:tensorcoequaliser}) the
  reflexive coequaliser-preserving extension is
  \[W^{\{H\}}_{H \le G}(R; M) = M^{\otimes_R G/H} \text{.}\]
\end{proof}

Note the forgetful
functor $\text{Ab}_\text{Haus} \to \text{Top}_\text{Haus}$ preserves reflexive coequalisers, so the underlying space functor $W^S_{H \le G} : \text{Mod} \to \text{Top}_\text{Haus}$ is the
unique reflexive coequaliser-preserving extension of the restricted underlying
space functor $\iota^\ast W^S_{H \le G} : \text{Mod}_F \to
\text{Top}_\text{Haus}$. This means that to compute the underlying space it
suffices to analyse what happens in the free case. However recall that the
isomorphism of Proposition~\ref{prop:witt_free} is not natural with respect to maps of
free objects (Remark~\ref{rem:witt_free_choices}), so that proposition will not
be very helpful here.

We check what happens when $M = R$ (analogously to \cite{dotto_witt_2025} Example~1.5.2).

\begin{lemma}\label{lem:ring_case_set_iso}
  We have a natural homeomorphism of topological spaces
  \[W^S_{H \le G}(R; R) \cong \prod_{V \in \underline{S}} R \text{.}\]
\end{lemma}
\begin{proof}
  Note $R^{\otimes_R G/U} \cong R$, and under this isomorphism the transfer map $\text{tr}_U^{N_H(U)} : R^{\otimes_R G/U}
  \to R^{\otimes_R G/U}$
  becomes multiplication by $\abs{N_H(U) : U}$ on $R$.

  First consider the case when $R$ is torsion-free. We claim the ghost map $w :
  \prod_{V \in \can{S}} R \to \text{gh}^S_{H \le G}(R; R)$ is injective. Suppose
  for contradiction that $n \in \prod_{V \in \can{S}} R$ is non-zero with $w(n)
  = 0$. We can choose $W \in \can{S}$ with $\abs{H:W}$ minimal such that $n_W
  \ne 0$. But then $0 = w_W(n) = \text{tr}^{N_H(W)}_W(n_W) = \abs{N_H(W):W}n_W$
  (see the calculation of the $V=W$ term of $w_W(n)$ in Lemma~\ref{lem:dwork},
  and observe all other terms vanish). This is a contradiction in a torsion-free ring, so $w$ is injective as claimed.

  We know the ghost map factorises as
  \[\prod_{V \in \can{S}} R \xrightarrowdbl{q} W^S_{H \le G}(R; R)
    \xrightarrow{w} \text{gh}^S_{H \le G}(R; R) \text{,}\]
  so for $R$ torsion-free the quotient $\prod_{V \in \can{S}} R \xrightarrowdbl{} W^S_{H
    \le G}(R; R)$ must also be injective, hence a homeomorphism.

  For general $R$, observe we can resolve $(R; R)$ by the
  free objects $(\mathbb{Z}[R]; \mathbb{Z}[R])$ and $(\mathbb{Z}[\mathbb{Z}[R]],
  \mathbb{Z}[\mathbb{Z}[R]])$. A coequaliser of homeomorphisms is a
  homeomorphism, so applying the torsion-free case we conclude that the
  (natural) map
  \[q : \prod_{V \in \can{S}} R \xrightarrowdbl{} W^S_{H \le G}(R; R)\]
  is always a homeomorphism.
\end{proof}

Indeed this is the result that we expect, since it agrees with the underlying
set of the $G$-typical Witt vectors of \cite{dress_burnside_1988}. We will show
later (Proposition~\ref{prop:generalise_burnside}) that the abelian group
structure also agrees (and we can even recover the ring multiplication).

Above we computed the Witt vectors for truncation sets of size $0$ and $1$. As the truncation set $S$ gets larger, it rapidly becomes hard to describe the
underlying space of the Witt vectors explicitly. The following is the last case
where we can do so fairly easily for general coefficients (following a similar
approach to \cite{dotto_witt_2025} Proposition~1.9).

\begin{lemma}
  Suppose $S$ only contains the whole group $H$ and some collection of
  maximal proper subgroups of $H$ (that is, the poset of subgroups in $S$ has height 2). Then we
have a natural homeomorphism of underlying spaces
\[W^S_{H \le G}(R; M) \cong M \times \prod_{V \in \can{S} \setminus \{H\}}
  (M^{\otimes_R G/V})_{N_H(V)} \text{.}
\]
\end{lemma}
\begin{proof}
To see this, first suppose $(T; Q)$ is free. In that case $\iota^\ast W_{H \le
  G}^S(T; Q)$ is homeomorphic to the image of the ghost map $w : \prod_{V \in \can{S}}
Q^{\otimes_T G/V} \to \text{gh}^S_{H \le G}(T; Q) \cong Q \times \prod_{V \in \can{S} \setminus \{H\}} (Q^{\otimes_T G/V})^{N_H(V)}$.
Given $m = (m_V) \in \prod_{V \in \can{S}} Q^{\otimes_T G/V}$, the
$H$-component of the ghost map is given by $w_H(m) = m_H$, and for $V$ a proper
subgroup in $S$ we have $w_V(m) = m_H^{\otimes_T G/V} +
\text{tr}_V^{N_H(V)}(m_V)$. The transfer map $\text{tr}_V^{N_H(V)} : Q^{\otimes_T
  G/V} \to (Q^{\otimes_T G/V})^{N_H(V)}$ factors as $Q^{\otimes_T G/V}
\xrightarrowdbl{} (Q^{\otimes_T G/V})_{N_H(V)} \hookrightarrow (Q^{\otimes_T
  G/V})^{N_H(V)}$ (the second map is injective since $(T; Q)$
is free). %
So the ghost map factors as
\[\begin{tikzcd}\prod_{V \in \can{S}} Q^{\otimes_T G/V} \ar[twoheadrightarrow, r]
    \ar[rd, "w" swap] & Q \times \prod_{V \in \can{S} \setminus \{H\}}
  (Q^{\otimes_T G/V})_{N_H(V)} \ar[hookrightarrow, d]\\
  & Q \times \prod_{V \in \can{S} \setminus \{H\}}
  (Q^{\otimes_T G/V})^{N_H(V)} \text{.}
\end{tikzcd}\]
and we deduce that $\iota^\ast W^S_{H \le G}(T; Q) \cong Q \times \prod_{V \in \can{S}
  \setminus \{H\}} (Q^{\otimes_T G/V})_{N_H(V)}$ as topological spaces,
naturally with respect to maps of free objects.

Tensor powers, orbits and products preserve reflexive coequalisers, so resolving
$(R; M)$ with free objects shows that we have a natural homeomorphism $W^S_{H \le G}(R; M) \cong M \times \prod_{V \in \can{S} \setminus \{H\}}
(M^{\otimes_R G/V})_{N_H(V)}$ in general.
\end{proof}

\subsection{Operators and monoidal structure on Witt vectors} \label{sec:operators}

All kinds of Witt vectors come with natural maps---most famously the Frobenius and Verschiebung operators. Our construction is
no exception.

We will define these operators by describing corresponding maps on ghost components, then
using the universal properties of reflexive coequaliser-preserving functors described in Lemma~\ref{lem:finite_module_kan} to
lift to maps of Witt vectors.

We start with the Frobenius and Verschiebung operators. Let $G$ be a profinite
group, $H$ an open subgroup of $G$, and $S$ a truncation set for $H$. Let $K$ be
an open
subgroup of $H$. The Frobenius and Verschiebung operators will go between
$W^S_{H \le G}(R; M)$ and $W^{S \mid_K}_{K \le G}(R; M)$, where  $S\!\!\mid_K = \{U \le K \mid U
\in S\}$ is the restriction of $S$ to $K$.

\begin{proposition} \label{prop:frobenius}
  There is a Frobenius operator
  \[F_K^H : W^S_{H \le G}(R; M) \to W^{S\mid_K}_{K \le G}(R; M) \text{,}\]
  natural in the choice of coefficients $(R; M)$. It is the unique
  natural transformation such that
  \[\begin{tikzcd}
      W^S_{H \le G}(R; M) \ar[r, "F_K^H"] \ar[d, "w"] & W^{S \mid_K}_{K \le G}(R; M)
      \ar[d, "w"]\\
      \text{gh}^S_{H \le G}(R; M) \ar[r, "\tilde{F}_K^H"] & \text{gh}^{S \mid_K}_{K \le G}(R; M)
    \end{tikzcd} \]
  commutes, where $\tilde{F}_K^H$ is defined to be the composition
  \[\left( \prod_{U \in S} M^{\otimes_R G/U} \right)^H \xrightarrow{\text{res}}
    \left( \prod_{U \in S} M^{\otimes_R G/U} \right)^K \xrightarrowdbl{}
    \left(\prod_{U \in S\mid_K} M^{\otimes_R G/U} \right)^K \text{.}\]
  The first map includes $H$-fixed points into $K$-fixed points, and the
  second map projects those components corresponding to subgroups of $K$.
  \end{proposition}
  \begin{proof}
Precomposing with the ghost map, we get a natural transformation
\[\tilde{F}_K^H w : W^S_{H \le G}({-}; {-}) \Rightarrow \text{gh}^{S \mid_K}_{K
    \le G}({-}; {-})\text{.}\]
By Lemma~\ref{lem:finite_module_kan} we just need to check that for $(T; Q)$ free,
the result of applying
$\tilde{F}_K^H w$ to an element of $W^S_{H \le G}(T; Q)$ lies in the image of
the ghost map $w : W^{S\mid_K}_{K \le G}(T; Q) \to
\text{gh}^{S \mid_K}_{K \le G}(T; Q)$. Then the restriction of $\tilde{F}_K^Hw$
to $\text{Mod}_F$ factors uniquely through the inclusion $\iota^\ast W^{S \mid_K}_{K \le
  G} \hookrightarrow \iota^\ast \text{gh}^{S\mid_K}_{K \le G}$, giving a natural
transformation $\iota^\ast W^S_{H \le G} \Rightarrow \iota^\ast W^{S \mid_K}_{K
  \le G}$. This extends to a natural transformation $W^S_{H \le G} \Rightarrow W^{S\mid_K}_{K \le
  G}$, which is the unique lift of $\tilde{F}_K^Hw$ along $w$.

To check this we can use the Dwork lemma. Let $a \in \text{gh}^S_{H \le G}(T; Q)$ be in the image of $w : W^S_{H \le G}(T; Q) \to
\text{gh}^S_{H \le G}(T; Q)$. We need to show that $\tilde{F}_K^H(a)$
satisfies the conditions of Lemma~\ref{lem:dwork}, proving that it is in the
image of $w : W^{S\mid_K}_{K \le G}(T; Q) \to
\text{gh}^{S \mid_K}_{K \le G}(T; Q)$. That is, given a subgroup $U \in S\!\!\mid_K$ we need to show that
\[\sum_{vU \in N_K(U)/U} \phi_U^{\langle  vU \rangle}(\tilde{F}_K^H(a)_{\langle  vU \rangle})\]
is in the image of $\text{tr}_U^{N_K(U)} : Q^{\otimes_T G/U} \to
Q^{\otimes_T G/U}$.
But since $\langle vU \rangle \in {S\!\!\mid_K}$ we have $\tilde{F}_K^H(a)_{\langle
  vU \rangle} = a_{\langle vU \rangle}$, and Lemma~\ref{lem:congruence} tells us that
\[\sum_{vU \in N_K(U)/U} \phi_U^{\langle  vU \rangle}(a_{\langle  vU \rangle})\]
is in the image of $\text{tr}_U^{N_K(U)}$ as desired.

So $\tilde{F}_K^H w$ lifts to a natural transformation
$F_K^H : W^S_{H \le G}(R; M) \to W^{S\mid_K}_{K \le G}(R; M)$.
\end{proof}

\begin{proposition} \label{prop:verschiebung}
  There is a Verschiebung operator
  \[V_K^H : W^{S\mid_K}_{K \le G}(R; M) \to W^S_{H \le G}(R; M)\text{,}\]
  natural in $(R; M)$. It is the unique natural transformation such that
  \[\begin{tikzcd}
    W^{S\mid_K}_{K \le G}(R; M) \ar[r, "V_K^H"] \ar[d, "w"] & W^S_{H \le G}(R; M) \ar[d, "w"]\\
    \text{gh}^{S\mid_K}_{K \le G}(R; M) \ar[r, "\tilde{V}_K^H"] & \text{gh}^S_{H \le G}(R; M)
    \end{tikzcd}\]
  commutes, where $\tilde{V}_K^H$ is defined to be the composition
\[
  \left( \prod_{U \in S\mid_K} M^{\otimes_R G/U} \right)^K \hookrightarrow
  \left( \prod_{U \in S} M^{\otimes_R G/U} \right)^K \xrightarrow{\text{tr}_K^H}
  \left( \prod_{U \in S} M^{\otimes_R G/U} \right)^H \text{.}
\]
\end{proposition}
\begin{proof}
Observe we can express $\tilde{V}_K^H$ in components as
  \[\tilde{V}_K^H(a)_W = \sum_{hK \in (H/K)^W} h \cdot a_{W^h}\]
  (recalling that $(H/K)^W$ is the set of cosets $hK$ such that $W^h \le K$).

Precomposing with the ghost map, we get a natural transformation
\[\tilde{V}_K^H w : W^{S\mid_K}_{K \le G}({-}; {-}) \Rightarrow \text{gh}^S_{H
    \le G}({-}; {-})\text{.}\]
Let $(T; Q)$ be free. Again it suffices to check that the image of
$\tilde{V}_K^H w : W^{S\mid_K}_{K \le G}(T; Q) \to \text{gh}^S_{H \le G}(T; Q)$ lies in the image of
the ghost map $w : W^S_{H \le G}(T; Q) \to \text{gh}^S_{H \le G}(T; Q)$, and then we will obtain a unique lift of $\tilde{V}_K^Hw$ along
$w$ to give $V_K^H : W^{S \mid_K}_{K
  \le G}({-}; {-}) \Rightarrow W^S_{H \le G}({-}; {-})$.

Let $a \in \text{gh}^{S \mid_K}_{K \le G}(T; Q)$ be in the image of
$w : W^{S \mid_K}_{K \le G}(T; Q) \to \text{gh}^{S \mid_K}_{K \le G}(T; Q)$. We want to show that given a subgroup $U \in S$, the sum
\begin{equation}\label{eq:verschiebungdwork} \sum_{vU \in N_H(U)/U} \phi_U^{\langle vU \rangle}(\tilde{V}^H_K(a)_{\langle vU \rangle})\end{equation}
is in the image of $\text{tr}_U^{N_H(U)} : Q^{\otimes_T G/U} \to Q^{\otimes_T
  G/U}$.
Using our expression for the components of $\tilde{V}^H_K$ gives
\begin{align*}
  & \sum_{vU \in N_H(U)/U} \phi^\lvUr_U \left(\sum_{hK \in (H/K)^\lvUr} h \cdot a_{\lvUr^h}\right)\\
  &= \sum_{vU \in N_H(U)/U} \sum_{hK \in (H/K)^\lvUr} h
\cdot \phi^{\langle vU \rangle^h}_{U^h}(a_{\lvUr^h})\\
  &= \sum_{hK \in (H/K)^U} h
    \cdot \sum_{vU^h \in N_K(U^h)/U^h} \phi^{\langle v U^h \rangle}_{U^h}(a_{\langle vU^h \rangle}) \text{.}
\end{align*}
By the Dwork lemma (noting that for $hK \in (H/K)^U$ we have $U^h \in S \!\!\mid_K$) we have that
\[\sum_{vU^h \in N_K(U^h)/U^h} \phi^{\langle v U^h \rangle}_{U^h}(a_{\langle
    vU^h \rangle}) = \text{tr}_{U^h}^{N_K(U^h)}(x_{U^h})\]
for some choice of $x_{U^h} \in Q^{\otimes_T G/U^h}$.
Then (\ref{eq:verschiebungdwork}) equals
\[\sum_{hK \in (H/K)^U} h \cdot \text{tr}_{U^h}^{N_K(U^h)}(x_{U^h}) \text{.}\]
Decompose according to the Weyl group action of $N_H(U)$ on $(H/K)^U$ to get
\begin{align*}
  \sum_{hK \in (H/K)^U} h \cdot \text{tr}_{U^h}^{N_K(U^h)}(x_{U^h}) 
  &= \sum_{hK \in (H/K)^U/{N_H(U)}} \ \sum_{s \in
                N_H(U)/N_{\prescript{h}{}{K}}(U)} sh \cdot \text{tr}_{U^h}^{N_K(U^h)}(x_{U^h})\\
  &= \sum_{hK \in (H/K)^U/{N_H(U)}} \text{tr}^{N_H(U)}_{N_{\prescript{h}{}{K}}(U)} \, \text{tr}_{U}^{N_{\prescript{h}{}{K}}(U)}(h \cdot x_{U^h})\\
    &= \sum_{hK \in (H/K)^U/{N_H(U)}} \text{tr}_U^{N_H(U)}(h \cdot x_{U^h})
\end{align*}
so (\ref{eq:verschiebungdwork}) is in the image of $\text{tr}_U^{N_H(U)}$ as desired.

So $\tilde{V}_K^H w$ lifts to a natural transformation
$V_K^H : W^{S\mid_K}_{K \le G}(R; M) \to W^S_{H \le G}(R; M)$.
\end{proof}

\begin{remark} \label{rem:verschiebung_quotient}
  Considering (the underlying space of) the Witt vectors as a quotient of $\prod_{V \in \can{S}} M^{\otimes_R G/V}$, we sometimes have an alternative
  description of the Verschiebung. Suppose that none of the $H$-conjugacy classes of
  subgroups in $S \!\!\mid_K$ split in $K$; that is, the $H$-conjugacy classes
  are also $K$-conjugacy classes. Then we can use the distinguished conjugacy
  class representatives $\can{S\!\!\mid_K} = \can{S} \cap S\!\!\mid_K$. If we use the same choices of
  coset representatives to define the quotient maps $\prod_{V \in \can{S}} M^{\otimes_R G/V}
  \xrightarrowdbl{} W^S_{H \le G}(R; M)$ and $\prod_{V \in \can{S\mid_K}} M^{\otimes_R
    G/V} \xrightarrowdbl{} W^{S\mid_K}_{K \le G}(R; M)$, then $V_K^H$
  agrees with the map on quotients induced by the inclusion $\prod_{V \in
    \can{S\mid_K}} M^{\otimes_R G/V} \hookrightarrow \prod_{V \in \can{S}} M^{\otimes_R
    G/V}$. To see this, observe that it is true in the free case (write down a
  square involving $\tilde{V}^H_K$ and check it is commutative), and then by
  Lemma~\ref{lem:finite_module_kan} this extends to the general case.
\end{remark}
\begin{remark} \label{rem:verschiebung_witt_free}
  For $(T; Q)$ free we have $W^S_{H \le G}(T; Q) \cong \left( \prod_{U \in S}
    Q^{\otimes_T G/U} \right)_H$ (Proposition~\ref{prop:witt_free}). Under
  these isomorphisms $V^H_K$ becomes the map
  \[\left( \prod_{U \in S \mid_K} Q^{\otimes_T G/U} \right)_K \hookrightarrow \left(
      \prod_{U \in S} Q^{\otimes_T G/U} \right)_K \xrightarrowdbl{} \left(
      \prod_{U \in S} Q^{\otimes_T G/U} \right)_H \text{.}\]
  This is a straightforward check on ghost components.
  We can also see this as the map
  \[\prod_{V \in \can{S \mid_K}} (Q^{\otimes_T G/V})_{N_K(V)} \to \prod_{V \in \can{S}} (Q^{\otimes_T G/V})_{N_H(V)}\]
  where the $V$-component of the image of an element $n$ is the sum of the
  $V'$-components of $n$ where $V'$ runs through those subgroups in $\can{S\!\!\mid_K}$ $H$-conjugate to $V$.
\end{remark}

We also have a conjugation operator on the Witt vectors. Let $g \in G$. Given
$S$ a truncation set for $H$, define $\prescript{g}{}{S} = \{\prescript{g}{}{U}
\mid U \in S\}$ to be the conjugate truncation
set for $\prescript{g}{}{H}$.

\begin{proposition} \label{prop:conjugation}
  There is a conjugation operator
  \[c_g : W^S_{H \le G}(R; M) \to
    W^{\prescript{g}{}{S}}_{\prescript{g}{}{H} \le G}(R; M) \text{,}\]
  natural in $(R; M)$. It is the unique natural transformation such that
  \[\begin{tikzcd}
      W^S_{H \le G}(R; M) \ar[d, "w"] \ar[r, "c_g"] &
      W^{\prescript{g}{}{S}}_{\prescript{g}{}{H} \le G}(R; M) \ar[d, "w"]\\
      \text{gh}^S_{H \le G}(R; M) \ar[r, "\tilde{c}_g"] & \text{gh}^{\prescript{g}{}{S}}_{\prescript{g}{}{H} \le G}(R; M)
    \end{tikzcd}\]
  commutes, where $\tilde{c}_g$ is defined to be the map
  \[\left( \prod_{U \in S} M^{\otimes_R G/U} \right)^H \xrightarrow{g \cdot ({-})}
    \left( \prod_{U \in \prescript{g}{}{S}} M^{\otimes_R G/U} \right)^{\prescript{g}{}{H}}\]
  induced by the maps $g \cdot ({-}) : M^{\otimes_R G/U} \to M^{\otimes_R
    G/{\prescript{g}{}{U}}}$.
  \end{proposition}
  \begin{proof}
Taking a similar approach to the last two propositions, it is straightforward to use
Lemma~\ref{lem:finite_module_kan} and the Dwork lemma to
check that $c_g w$ has a unique lift $c_g$ along $w$.
\end{proof}
\begin{remark}
  For $(T; Q)$ free this is the map
  \[W^S_{H \le G}(T; Q) \cong \left( \prod_{U \in S} Q^{\otimes_T G/U} \right)_H \xrightarrow{g \cdot
      ({-})} \left( \prod_{U \in \prescript{g}{}{S}} Q^{\otimes_T G/U}
    \right)_{\prescript{g}{}{H}} \cong
    W^{\prescript{g}{}{S}}_{\prescript{g}{}{H} \le G}(T; Q)\text{.}\]
\end{remark}

We will need some identities involving the Frobenius, Verschiebung and
conjugation operators, generalising those for previous versions of Witt vectors.
The identities are reminiscent of the Mackey functor axioms, and
indeed for $G$ a finite group the Witt
vectors do define a Mackey functor (see Lemma~\ref{lem:witt_mackey}).

\begin{proposition} \label{prop:fvc_identities}
  We have the following identities:
  \begin{enumerate}[(i)]
    \item The maps $F^H_H$ and $V^H_H$ are the identity on $W^{S}_{H \le
        G}(R; M)$, as is $c_h$ for $h \in H$.
    \item For subgroups $J \le_o K \le_o H$ we have $F^K_J F^H_K = F^H_J$ and $V^H_K
      V^K_J = V^H_J$.
    \item We have $c_g c_{g'} = c_{g g'}$ for $g, g' \in G$. In particular this
      gives an action of the Weyl group $N_G(H)/H$ on $W^S_{H \le G}(R;
      M)$.
    \item For $K \le_o H$ we have $V^{\prescript{g}{}{H}}_{\prescript{g}{}{K}} c_g
      = c_g V^H_K$ and $F^{\prescript{g}{}{H}}_{\prescript{g}{}{K}} c_g = c_g F^H_K$.
    \item For $J \le_o K \le_o H$ we have
      \[F^H_J V^H_K = \sum_{JhK \in J \backslash H/K} V^J_{J \cap
          \prescript{h}{}{K}} c_h F^K_{J^h \cap K} \text{.}\]
  \end{enumerate}
  The analogous identities for $\tilde{F}^H_K$, $\tilde{V}^H_K$ and
  $\tilde{c}_g$ mapping between ghost groups also hold.
\end{proposition}
\begin{proof}
  For all of these we can check the analogous identities on ghost components,
  and then conclude by uniqueness of lifting. (i) is immediate since
  $\tilde{F}^H_H$, $\tilde{V}^H_H$ and $\tilde{c}_h$ are the identity on
  $\text{gh}^S_{H \le G}(R; M)$ for $h \in H$. For (ii),
  observe that $\tilde{F}^K_J \tilde{F}^H_K = \tilde{F}^H_J$ and $\tilde{V}^H_K
  \tilde{V}^K_J = \tilde{V}^H_J$ (the latter a consequence of the fact that
  transfers of group actions satisfy $\text{tr}^H_K \text{tr}^K_J =
  \text{tr}^H_J$). Similarly (iii) follows as $\tilde{c}_g\tilde{c}_{g'} =
  \tilde{c}_{g g'}$, and (iv) follows since $\tilde{V}^{\prescript{g}{}{H}}_{\prescript{g}{}{K}} \tilde{c}_g
  = \tilde{c}_g \tilde{V}^H_K$ and
  $\tilde{F}^{\prescript{g}{}{H}}_{\prescript{g}{}{K}} \tilde{c}_g = \tilde{c}_g
  \tilde{F}^H_K$. The only identity that requires some work is (v), the double
  coset formula. We need to show that for $J \le_o K \le_o H$ we have
  \[\tilde{F}_J^H\tilde{V}_K^H = \sum_{JhK \in J \backslash H / K} \tilde{V}^J_{J \cap
      \prescript{h}{}{K}} \tilde{c}_h \tilde{F}^K_{J^h \cap K}\]
  as maps $\text{gh}_{K \le G}^{S \mid_K}(R; M) \to \text{gh}_{J \le G}^{S \mid_J}(R; M)$. Let
  $U \in S \!\!\mid_J$ and $a \in \text{gh}_{K \le G}^{S \mid_K}(R; M)$. Considering the
  $U$-component of the left hand side, we have
  \[\left( \tilde{F}_J^H \tilde{V}_K^H(a) \right)_U = \left( \tilde{V}_K^H(a) \right)_U = \sum_{hK \in (H/K)^U} h \cdot
    a_{U^h} \text{.}\]
  On the right hand side we have
  \begin{align*}
    \left(\sum_{JhK \in J \backslash H / K} \tilde{V}^J_{J \cap \prescript{h}{}{K}} \tilde{c}_h \tilde{F}^K_{J^h \cap K}(a)\right)_U &= \sum_{JhK \in J \backslash H / K} \ \sum_{j(J \cap \prescript{h}{}{K}) \in (J/(J \cap \prescript{h}{}{K}))^U} j \cdot \left(\tilde{c}_h\tilde{F}^K_{J^h \cap K}(a)\right)_{U^j}\\
                                                                                                              &= \sum_{JhK \in J \backslash H / K} \ \sum_{j(J \cap \prescript{h}{}{K}) \in (J/(J \cap \prescript{h}{}{K}))^U} j h \cdot \left(\tilde{F}^K_{J^h \cap K}(a)\right)_{U^{jh}}\\
                                                                                                              &= \sum_{JhK \in J \backslash H / K} \ \sum_{j(J \cap \prescript{h}{}{K}) \in (J/(J \cap \prescript{h}{}{K}))^U} j h \cdot a_{U^{jh}} \text{.}
  \end{align*}
  Observe that $J$ acts on $H/K$ by left multiplication, where $hK$ has orbit $JhK$ and stabiliser $J
  \cap \prescript{h}{}{K}$. So
  \begin{align*}
    (H/K)^U &= \{hK \in H/K \mid U^h \le K\}\\
            &= \{jhK \mid JhK \in J\backslash H/K, j(J \cap \prescript{h}{}{K}) \in J/(J \cap \prescript{h}{}{K}), U^{jh} \le K\}\\
            &= \{jhK \mid JhK \in J\backslash H/K, j(J \cap \prescript{h}{}{K}) \in J/(J \cap
              \prescript{h}{}{K}), U^{j} \le J \cap \prescript{h}{}{K}\}\\
            &= \{jhK \mid JhK \in J\backslash H/K, j(J \cap \prescript{h}{}{K}) \in (J/(J \cap
              \prescript{h}{}{K}))^U\} \text{.}
  \end{align*}
  Conclude that all components of the left and right hand sides agree, proving the identity.
\end{proof}

There is a truncation operator relating the Witt
vectors defined with different truncation sets.
Suppose $S$ and $S'$ are both truncation sets for $H$, with $S' \subseteq S$.

\begin{proposition}
  \label{prop:truncation}
  There is a truncation operator
  \[R_{S'} : W^S_{H \le G}(R; M) \to W^{S'}_{H \le G}(R; M) \text{,}\]
  natural in $(R; M)$. It is the unique natural transformation such that
  \[\begin{tikzcd}
    W^S_{H \le G}(R; M) \ar[r, "R_{S'}"] \ar[d, "w"] & W^{S'}_{H \le G}(R; M)
    \ar[d, "w"]\\
    \text{gh}^S_{H \le G}(R; M) \ar[r, "\tilde{R}_{S'}"] & \text{gh}^{S'}_{H \le
      G}(R; M)
  \end{tikzcd}\]
commutes, where $\tilde{R}_{S'}$ is defined to be the projection map
\[\left( \prod_{U \in S} M^{\otimes_R G/U}
  \right)^H \xrightarrowdbl{}
  \left(\prod_{U \in S'} M^{\otimes_R G/U}\right)^H \text{.}\]
\end{proposition}
\begin{proof}
As usual the Dwork lemma combined with Lemma~\ref{lem:finite_module_kan} shows that
$\tilde{R}_{S'}w$ has a unique lift $R_{S'}$ along the ghost map.
\end{proof}
\begin{remark} \label{rem:truncation_quotient}
If we consider (the underlying space of) the Witt vectors as a quotient of
$\prod_{V \in \can{S}} M^{\otimes_R G/V}$, we can give an alternative
description of the truncation map. If we use the distinguished conjugacy class
representatives $\can{S'} = \can{S} \cap S'$ and the same choices of
coset representatives to define the quotient maps $\prod_{V \in \can{S}} M^{\otimes_R G/V}
\xrightarrowdbl{} W^S_{H \le G}(R; M)$ and $\prod_{V \in \can{S'}} M^{\otimes_R
G/V} \xrightarrowdbl{} W^{S'}_{H \le G}(R; M)$, then $R_{S'}$ agrees with the map on quotients induced by the projection $\prod_{V \in \can{S}}
M^{\otimes_R G/V} \xrightarrowdbl{} \prod_{V \in \can{S'}} M^{\otimes_R
  G/V}$.
\end{remark}
\begin{remark} \label{rem:truncation_witt_free}
  For $(T; Q)$ free the truncation map is the
  projection map
  \[W^S_{H \le G}(T; Q) \cong \left( \prod_{U \in S} Q^{\otimes_T G/U} \right)_H \xrightarrowdbl{} \left(
    \prod_{U \in S'} Q^{\otimes_T G/U} \right)_H \cong W^{S'}_{H \le G}(T; Q) \text{.}\]
We can also see this as a projection map
\[\prod_{V \in \can{S}} (Q^{\otimes_T G/V})_{N_H(V)} \xrightarrowdbl{} \prod_{V
  \in \can{S'}} (Q^{\otimes_T G/V})_{N_H(V)}\]
if we pick subgroup representatives as in the previous remark.
\end{remark}

The truncation maps commute with all the other structure, in the following
sense.

\begin{proposition} \label{prop:trunc_commute}
  We have
  \[R_{S'} V^H_K = V^H_K R_{S'\mid_K} \quad R_{S'\mid_K} F^H_K = F^H_K R_{S'}
    \quad R_{\prescript{g}{}{S'}} c_g = c_g R_{S'} \text{.}\]
\end{proposition}
\begin{proof}
  The identities
  \[\tilde{R}_{S'} \tilde{V}^H_K = \tilde{V}^H_K \tilde{R}_{S'\mid_K} \quad \tilde{R}_{S'\mid_K} \tilde{F}^H_K = \tilde{F}^H_K \tilde{R}_{S'}
    \quad \tilde{R}_{\prescript{g}{}{S'}} \tilde{c}_g = \tilde{c}_g \tilde{R}_{S'}\]
  hold on ghost components, so we are done by uniqueness of lifting.
\end{proof}

Analogously to \cite{dotto_witt_2022} we would like to have a Teichm\"uller map $M^{\otimes_R G/H} \to W^S_{H \le
  G}(R; M)$. We only expect a continuous map of spaces, not necessarily an
additive map (though it should at least preserve zero). When $S = \emptyset$, we have $W^\emptyset_{H \le G}(R; M) = 0$ so
the Teichm\"uller map must be zero. Otherwise $S$ is non-empty and since it is
upwards-closed it must contain $H$. For $(T; Q)$ free, there is an obvious map given by
\[\tau' : Q^{\otimes_T G/H} \hookrightarrow \prod_{V \in \can{S}} Q^{\otimes_T G/V}
  \xrightarrowdbl{q} W^S_{H \le G}(T; Q) \text{.}\]
Since this is defined using our representation of $W^S_{H \le G}(T; Q)$ as a
quotient, it depends on some choices of coset representatives. Luckily it turns
out not to depend on all of the arbitrary choices we made---only our choice of
coset representatives for $G/H$. To see this, observe that on ghost components
we have
\begin{equation} \label{eq:tau_prime_ghost} w_U(\tau'(m)) = f_{G/H}(m^{\otimes_T
    H/U}) \text{,}\end{equation}
where $f_{G/H}$ is defined using our coset representatives for $G/H$.
In the following proposition we describe a Teichm\"uller map $\tau_{G/H}$ for each choice of coset representatives for
$G/H$, defined as the lift of the map to the ghost group given
by (\ref{eq:tau_prime_ghost}).

\begin{proposition} \label{prop:teichmuller}
  Given $\{g_i H\}$ a choice of coset representatives for $G/H$, there is a
  continuous (not necessarily additive) Teichm\"uller map
  \[\tau_{G/H} : M^{\otimes_R G/H} \to W^S_{H \le G}(R; M) \text{,}\]
  natural in $(R; M)$.
  This is the unique natural transformation (of functors $\text{Mod} \to \text{Top}_\text{Haus}$) such that
  \[\begin{tikzcd}
      M^{\otimes_R G/H} \ar[rd, "\tilde{\tau}_{G/H}" swap] \ar[r, "\tau_{G/H}"] & W^S_{H \le G}(R; M)
      \ar[d, "w"]\\
      & \text{gh}^S_{H \le G}(R; M)
    \end{tikzcd}\]
  commutes, where $\tilde{\tau}_{G/H} : M^{\otimes_R G/H} \to \left( \prod_{U \in S}
    M^{\otimes_R G/U} \right)^H = \text{gh}^S_{H \le G}(R; M)$ is defined by
  \[\tilde{\tau}_{G/H}(m)_U = f_{G/H}(m^{\otimes_R H/U})\text{.}\]
  Note $f_{G/H}$ should use our choice $\{g_i H\}$ of coset representatives.
\end{proposition}
\begin{proof}
Equation~\ref{eq:tau_prime_ghost}
shows that for $(T; Q)$ free, the image of
\[\tilde{\tau}_{G/H} : Q^{\otimes_T G/H} \to \text{gh}^S_{H \le G}(T; Q)\]
is contained in the image of the
ghost map (we see that this is true if we define the ghost map using the same choice of
coset representatives for $G/H$, but the image doesn't depend on this choice). So we get a natural transformation of functors
$\text{Mod}_F \to \text{Top}_\text{Haus}$ lifting the restriction of $\tilde{\tau}_{G/H}$
along the ghost map $w : W^S_{H \le G}(T; Q) \hookrightarrow \text{gh}^S_{H \le G}(T; Q)$.
As usual we deduce from Lemma~\ref{lem:finite_module_kan} that this natural
transformation extends to a unique natural transformation between reflexive
coequaliser-preserving functors $\text{Mod} \to \text{Top}_\text{Haus}$, and since tensor
powers and Witt vectors both preserve reflexive coequalisers this gives a unique
lift of $\tilde{\tau}_{G/H}$ along the ghost map.
\end{proof}

We record some properties of the Teichm\"uller map.
Fix a choice of coset representatives $\{g_iH\} = G/H$.

\begin{proposition} \label{prop:witt_teichmuller_properties}
  The map $\tau_{G/H}$ has the following properties:
  \begin{enumerate}[(i)]
  \item{
      We have $\tau_{G/H}(0) = 0$.
    }
  \item
    The map $R$ interacts well with $\tau_{G/H}$, in the sense that the diagram
    \[\begin{tikzcd}
        M^{\otimes_R G/H} \ar[r, "\tau_{G/H}"]
        \ar[rd, "\tau_{G/H}" swap] &
        W^S_{H \le G}(R; M)
        \ar[d, "R_{S'}"]\\
        & W^{S'}_{H \le G}(R; M)
      \end{tikzcd}\]
    commutes.
  \item
    The map $\tau_{G/H}$ is equivariant, in the sense that
    \[\tau_{G/\prescript{g}{}{H}}(g \cdot m) = c_g \tau_{G/H}(m)\]
    for any $g \in G$ (where we use the map $\tau_{G/\prescript{g}{}{H}}$
    corresponding to the coset representatives $G/\prescript{g}{}{H} = \{g_i g^{-1} (\prescript{g}{}{H})\}$).
  \item The map
    \[M^{\otimes_R G/H} \xrightarrow{\tau_{G/H}} W^{\{H\}}_{H \le G}(R; M)\]
    is a monoidal additive isomorphism, independent of the choice of coset representatives.
  \item
    Suppose we have coset representatives $\{g_i\}$ for $G/H$ and $\{h_j\}$ for
    $H/K$. Observe that $\{g_i h_j\}$ is a set of coset representatives for
    $G/K$. Then the diagram
    \[\begin{tikzcd}
        M^{\otimes_R G/H} \ar[r, "\tau_{G/H}"] \ar[d,
        "({-})^{\otimes_R H/K}"] &
        W^S_{H \le G}(R; M)
        \ar[dd, "F_K^H"]\\
        M^{\otimes_R G/H \times H/K} \ar[d, "f_{G/H}"]&\\
        M^{\otimes_R G/K} \ar[r, "\tau_{G/K}"] & W^{S \mid_K}_{K \le G}(R; M)
      \end{tikzcd}\]
    commutes (where $f_{G/H}$, $\tau_{G/H}$ and $\tau_{G/K}$ are defined using
    the above coset representatives).
  \end{enumerate}
\end{proposition}
\begin{proof}
  \begin{enumerate}[(i)]
  \item{
      This follows by naturality, since the map of modules $(\mathbb{Z}; 0) \to
      (R; M)$ gives a commutative diagram
      \[\begin{tikzcd}
          0^{\otimes_{\mathbb{Z}} G/H} \ar[r, "\tau_{G/H}"] \ar[d] & W^S_{H \le
            G}(\mathbb{Z}; 0) \ar[d]\\
          M^{\otimes_{\mathbb{Z}} G/H} \ar[r, "\tau_{G/H}"] & W^S_{H \le G}(R;
          M)
        \end{tikzcd}\]
      and $W^S_{H \le G}(\mathbb{Z}; 0) = 0$.
    }
  \item{
      Recall Proposition~\ref{prop:truncation}. The truncation map $R_{S'}$ is induced by the map $\tilde{R}_{S'}$ on
      ghost components, which is the projection map onto those components
      indexed by the new truncation set $S'$.
      Using this it is straightforward to check (ii) on ghost components.
    }
  \item{
      Let $m \in M^{\otimes_R G/H}$ and $g \in G$.
      Recall the properties of the conjugation operator from
      Proposition~\ref{prop:conjugation}. We check
      \[w_{\prescript{g}{}{U}}(c_g(\tau_{G/H}(m))) = g \cdot w_U(\tau_{G/H}(m)) =
        g \cdot f_{G/H}(m^{\otimes_R H/U})\]
      and
      \[w_{\prescript{g}{}{U}}(\tau_{G/\prescript{g}{}{H}}(g \cdot m)) =
        f_{G/\prescript{g}{}{H}}((g \cdot m)^{\otimes_R
          \prescript{g}{}{H}/{\prescript{g}{}{U}}}) \text{.}\]
      One can check that these are equal (when we use the appropriate choices of
      coset representatives), and so $c_g \cdot \tau_{G/H} =
      \tau_{G/\prescript{g}{}{H}}(g \cdot ({-}))$.
    }
  \item{
      Given $m
      \in M^{\otimes_R G/H}$ we have
      \[w_H(\tau_{G/H}(m)) = f_{G/H}(m^{\otimes_R H/H}) = m\]
      so $w \, \tau_{G/H} : M^{\otimes_R G/H} \to \text{gh}^{\{H\}}_{H \le G}(R; M)$ is
      the identity map, under the identification $\text{gh}^{\{H\}}_{H \le
        G}(R; M) \coloneqq (M^{\otimes_R G/H})^H = M^{\otimes_R G/H}$. But by
      Lemma~\ref{lem:witt_trivial_cases} we know that the ghost map $w : W^{\{H\}}_{H \le G}(R; M) \to
      \text{gh}^{\{H\}}_{H \le G}(R; M) = M^{\otimes_R G/H}$ is an additive
      isomorphism, so $\tau_{G/H}$ is the inverse
      isomorphism (and hence also independent of the choice of coset representatives).
    }
  \item{
      We want to verify the commutative diagram
      \[\begin{tikzcd}
        M^{\otimes_R G/H} \ar[r, "\tau_{G/H}"] \ar[d, "({-})^{\otimes_R H/K}"] & W^S_{H \le G}(R; M)
        \ar[dd, "F_K^H"]\\
        M^{\otimes_R G/H \times H/K} \ar[d, "f_{G/H}"]&\\
        M^{\otimes_R G/K} \ar[r, "\tau_{G/K}"] & W^{S \mid_K}_{K \le G}(R; M) \text{,}
      \end{tikzcd}\]
      where $\tau_{G/H}$ and $f_{G/H}$ are defined using the coset representatives
      $\{g_i H\}$ and $\tau_{G/K}$ is defined using $\{g_i h_j K\}$.
      Recall the properties of the Frobenius operator from
      Proposition~\ref{prop:frobenius}. To check the diagram commutes on ghost components,
      observe that for $U \in S \!\!\mid_K$ and $m \in M^{\otimes_R G/H}$ we have
      \[w_U(F^H_K \tau_{G/H}(m)) = w_U(\tau_{G/H}(m)) = f_{G/H}(m^{\otimes_R H/U})\]
      and
      \[w_U(\tau_{G/K}(f_{G/H}(m^{\otimes_R H/K}))) = f_{G/K}(f_{G/H}(m^{\otimes_R
          H/K})^{\otimes_R K/U}) \text{.}\]
      The choice of coset representatives is such that these are equal. \qedhere
    }
  \end{enumerate}
\end{proof}

  If we are working with a fixed choice $(T; Q)$ of free coefficients then
  we can use the isomorphism $W^S_{H \le G}(T; Q) \cong \left( \prod_{U \in S} Q^{\otimes_T G/U}
  \right)_H$ of Proposition~\ref{prop:witt_free} to get an alternative Teichm\"uller-like map.
\begin{definition}\label{def:alternative_teichmuller}
  For $(T; Q)$ free we define the alternative Teichm\"uller map
  \[\tau_{G/H}^f
  : Q^{\otimes_T G/H} \to \left( \prod_{U \in S} Q^{\otimes_T G/U}
  \right)_H \cong W^S_{H \le G}(T; Q)\]
  to be the inclusion of the $U=H$ component of the product (for $S$ non-empty, or zero otherwise).
\end{definition}

\begin{remark} \label{rem:alternative_teichmuller}
  This map has
  the advantage that it is additive and does not depend on a choice of coset
  representatives, but the disadvantage that it is not natural in the choice of
  coefficients (recall Remark~\ref{rem:witt_free_choices}).
  The ghost components are $w_U(\tau^f_{G/H}(m))
  = \phi^H_U(m)$. The map $\tau^f_{G/H}$ satisfies analogous properties to
  those we proved for $\tau_{G/H}$ in
  Proposition~\ref{prop:witt_teichmuller_properties}, by essentially the same proofs. Note part (v) becomes the
  identity $F^H_K(\tau^f_{G/H}(m)) = \tau^f_{G/K}(\phi^H_K(m))$, and no longer
  depends on choices of coset representatives. This will be key to the proof of
  Proposition~\ref{prop:untruncated_witt_monoidal}, see
  Remark~\ref{rem:untruncated_witt_monoidal}. We do not believe an analogue of
  this map has previously appeared in the literature; indeed it is primarily
  useful for avoiding the complexities of choosing coset representatives,
  which are new to this setting.
\end{remark}

Our expressions for the components of $\tilde{V}^H_K$ and $\tilde{\tau}_{G/H}$
are reminiscent of the formula for the ghost map in Definition~\ref{def:ghost_map}. Indeed we can use them to give
a more elegant expression for the ghost map.

\begin{lemma} \label{lem:ghost_via_operators}
  The ghost map $w : \prod_{V \in \can{S}} M^{\otimes_R G/V} \to \text{gh}^S_{H
    \le G}(R; M)$ can be defined as
  \[n \mapsto \sum_{V \in \can{S}} \tilde{V}^H_V \tilde{\tau}_{G/V}(n_V) \text{.}\]

  In particular this implies that the quotient map $q : \prod_{V
    \in \can{S}} M^{\otimes_R G/V} \to W^S_{H \le G}(R; M)$ is given by
  \[n \mapsto \sum_{V \in \can{S}} V^H_V \tau_{G/V}(n_V) \text{.}\]

  By a (possibly uncountably infinite) sum over $V \in \can{S}$, we mean the
  limit of the net sending a finite subset of $\can{S}$ to
  the sum indexed by that subset. This agrees with the treatment in
  \cite{dress_burnside_1988}.\footnote{See Remark~\ref{rem:g_infinite} in the
    next section for more
    discussion of how the infinite group case of our construction might connect to other work.}
\end{lemma}
\begin{proof}
  The claim about the ghost map is immediate from the descriptions of
  $\tilde{V}^H_V$ and $\tilde{\tau}_{G/V}$ in components given in
  Propositions~\ref{prop:verschiebung} and \ref{prop:teichmuller}. Note that
  each component is only non-zero for finitely many terms of the sum.

  For the claim about the quotient map, first suppose for simplicity that $S$ is
  finite. Since $V^H_V$ and $\tau_{G/V}$ are defined as lifts of $\tilde{V}^H_V$ and
  $\tilde{\tau}_{G/V}$, the natural map
  \begin{align}
    \prod_{V \in \can{S}} M^{\otimes_R G/V} &\to W^S_{H \le G}(R; M) \nonumber \\
    n &\mapsto \sum_{V \in \can{S}} V^H_V \tau_{G/V}(n_V) \label{eq:ghost_sum_map} \end{align} %
  factors the ghost map $\prod_{V \in \can{S}} M^{\otimes_R G/V} \to
  \text{gh}^S_{H \le G}(R; M)$ along the Witt vector ghost map $W^S_{H \le G}(R; M) \to
  \text{gh}^S_{H \le G}(R; M)$. When $(R; M)$ is free the Witt vector ghost map is
  injective, so (\ref{eq:ghost_sum_map}) must be the usual quotient map $q$. And then
  Lemma~\ref{lem:finite_module_kan} tells us that it must also be the usual quotient
  map for general coefficients.

  With a little care, the same statement holds for $S$ infinite.
  By definition the sum $\sum_{V \in \can{S}} V_V^H \tau_{G/V}(n_V)$ is the limit of the
  net in $W^S_{H \le G}(R; M)$ given by
  \begin{equation}\label{eq:sum_net}I \mapsto \sum_{V \in I} V^H_V \tau_{G/V}(n_V)\end{equation}
  for $I$ a finite subset of $\can{S}$. We have a net in $\prod_{V \in \can{S}} M^{\otimes_R
    G/V}$ taking $I$ to the element $(n'_V)_{V \in \can{S}}$ given by $n'_V =
  n_V$ for $V \in I$ and $n'_V = 0$ otherwise. The limit of this net is clearly
  $n$. By the same logic as above, the
  image of this net under $q$ is (\ref{eq:sum_net}), and so since $q$ is
  continuous, the limit of
  (\ref{eq:sum_net}) is $q(n)$ as desired.
\end{proof}

\begin{remark} \label{rem:free_V_tau}
  Similarly we see that for $(T; Q)$ free the isomorphism
  \[\left( \prod_{U \in S}
    Q^{\otimes_T G/U} \right)_H \cong W^S_{H \le
    G}(T; Q)\]
of Proposition~\ref{prop:witt_free} is given by
  \[n \mapsto \sum_{U \in S} V^H_U \tau^f_{G/U}(n_U) \text{.}\]
\end{remark}

\begin{remark}
  Let $n, n' \in \prod_{V \in \can{S}} M^{\otimes_R G/V}$ such that for each
  subgroup $V \in \can{S}$ at least one of $n_V$ and $n'_V$ is zero. Let $n + n'$ denote the
  pointwise sum of $n$ and $n'$. Then since $\tau_{G/V}(0) = 0$ we have
  \begin{align*} q(n) + q(n') &= \sum_{V \in \can{S}} V^H_V \tau_{G/V}(n_V) + \sum_{V \in \can{S}} V^H_V \tau_{G/V}(n'_V)\\
                              &= \sum_{V \in \can{S}} V^H_V \tau_{G/V}(n_V + n'_V)\\
    &= q(n+n') \in W^S_{H \le G}(R; M) \text{.}\end{align*}
  The difficulty of general Witt vector addition comes from the fact that the
  Teichm\"uller map isn't additive, so if $n_V$ and $n'_V$ are both non-zero
  then $V^H_V \tau_{G/V}(n_V) + V^H_V \tau_{G/V}(n'_V)$ is not simply $V^H_V \tau_{G/V}(n_V + n'_V)$.
\end{remark}

The usual Witt vectors of a commutative ring are themselves a commutative ring.
We have defined the $G$-typical Witt vectors with coefficients as merely a topological abelian group; but
there is a lax monoidal structure that generalises the multiplication of the
usual Witt vectors, described in the following proposition. We will further refine our understanding of the monoidal
structure in Section~\ref{sec:witt_mackey}, where we show that (for $G$ finite) the untruncated
$G$-typical Witt vectors with coefficients give a \textit{strong} symmetric monoidal functor to the category
of $G$-Mackey functors.

\newcommand{\tildestar}{\operatorname{\tilde{\star}}}

\begin{proposition} \label{prop:witt_monoidal}
There is an external product
\[\star : W^S_{H \le G}(R; M) \otimes_{\mathbb{Z}} W^S_{H \le G}(R'; M') \to W^S_{H \le G}(R
  \otimes_{\mathbb{Z}} R'; M \otimes_{\mathbb{Z}} M')\]
on the Witt vectors. This is the unique natural transformation such that
\[\begin{tikzcd}
  W^S_{H \le G}(R; M) \otimes_{\mathbb{Z}} W^S_{H \le G}(R'; M') \ar[r, "\star"]
  \ar[d, "w \otimes_{\mathbb{Z}} w"]
  &W^S_{H \le G}(R
  \otimes_{\mathbb{Z}} R'; M \otimes_{\mathbb{Z}} M') \ar[d, "w"]\\
  \text{gh}^S_{H \le G}(R; M) \otimes_{\mathbb{Z}} \text{gh}^S_{H \le G}(R'; M') \ar[r, "\tildestar"]
  &\text{gh}^S_{H \le G}(R
  \otimes_{\mathbb{Z}} R'; M \otimes_{\mathbb{Z}} M')
\end{tikzcd}\]
commutes, where
\[\tildestar : \text{gh}^S_{H \le G}(R; M) \otimes_\mathbb{Z}
\text{gh}^S_{H \le G}(R'; M') \to \text{gh}^S_{H \le G}(R \otimes_\mathbb{Z} R'; M \otimes_\mathbb{Z}
M')\]
is the map induced by the canonical shuffle isomorphisms
\[s : M^{\otimes_R G/U} \otimes_{\mathbb{Z}} M'^{\otimes_{R'} G/U} \xrightarrow{\cong} (M
\otimes_{\mathbb{Z}} M')^{\otimes_{R \otimes_{\mathbb{Z}} R'} G/U} \text{.}\]

Let $u : \mathbb{Z} \to W^S_{H \le G}(\mathbb{Z}; \mathbb{Z})$ be the
additive map defined by $1 \mapsto \tau_{G/H}(1)$. Then the functor $W^S_{H \le G} : \text{Mod} \to \text{Ab}_{\text{Haus}}$ is lax
symmetric monoidal via $\star$ and unit $u$ (and $\text{gh}^S_{H \le G}$ is lax
symmetric monoidal via $\tildestar$ and unit $wu$).

The external product satisfies a Frobenius reciprocity-type identity, and we
have a formula for the external product of images of Teichm\"uller maps:
\begin{enumerate}[(i)]
  \item Given $a \in W^{S \mid_K}_{K \le G}(R; M)$ and $a' \in W^S_{H \le G}(R';
    M')$, we have
    \[V^H_K(a) \star a' = V^H_K(a \star F^H_K(a'))\]
    (similarly the flipped version with the two factors reversed).
  \item Given $m \in M^{\otimes_R G/H}$ and $m' \in M'^{\otimes_{R'} G/H}$, we have
    \[\tau_{G/H}(m) \star \tau_{G/H}(m') = \tau_{G/H}(s(m \otimes_\mathbb{Z} m')) \text{.}\]
  \end{enumerate}
The analogous identities for $\tildestar$, $\tilde{V}^H_K$, $\tilde{F}^H_K$ and
$\tilde{\tau}_{G/H}$ also hold.
\end{proposition}
\begin{proof}
We will write $\otimes$ for $\otimes_\mathbb{Z}$. As usual it suffices to show
that the lift
\[\star : W^S_{H \le G}(T; Q) \otimes W^S_{H \le G}(T'; Q') \to
  W^S_{H \le G}(T \otimes T'; Q \otimes Q')\]
exists for $(T; Q)$ and $(T'; Q')$ free (apply Lemma~\ref{lem:finite_module_kan} with $(R; M)$ fixed and with $(R'; M')$ fixed); similarly we only need to prove the identities in the free case.

Since $(T \otimes T'; Q \otimes Q')$ is free, $w : W^S_{H \le G}(T
\otimes T'; Q \otimes Q') \to \text{gh}^S_{H \le G}(T \otimes T'; Q \otimes
Q')$ is a subspace inclusion. So we just need to show that
the image of an element of $W^S_{H \le G}(T; Q) \otimes W^S_{H \le G}(T; Q)$ under $w \otimes w$ followed by $\tildestar$ is in the image of $w$, and then $\star$ is the unique lift of $\tildestar \circ (w \otimes w)$
along $w$.

By Lemma~\ref{lem:ghost_via_operators} it suffices to show that
$\tilde{V}^H_{V} \tilde{\tau}_{G/V}(q) \tildestar \tilde{V}^H_{V'} \tilde{\tau}_{G/V'}
(q')$ is in the image of $w$, for $V, V' \in \can{S}$, $q \in Q^{\otimes_T
  G/V}$ and $q' \in Q'^{\otimes_{T'} G/V'}$. To show this we will prove
the identities for $\tildestar$ analogous to those described for $\star$, and
then conclude the result.

First we want to check that for $a \in \text{gh}^S_{K \le G}(R; M)$ and $a' \in
\text{gh}^S_{H \le G}(R'; M')$ we have $\tilde{V}^H_K(a) \tildestar a' =
\tilde{V}^H_K(a \tildestar \tilde{F}^H_K(a'))$. Indeed
\begin{align*}
  (\tilde{V}^H_K(a) \tildestar a')_W &= s(\tilde{V}^H_K(a)_W \otimes a'_W)\\
                                     &= s\left( \sum_{hK \in (H/K)^W} (h \cdot a_{W^h}) \otimes
                                       a'_W \right)\\
                                     &=s\left( \sum_{hK \in (H/K)^W} h \cdot (a_{W^h} \otimes a'_{W^h}) \right)\\
                                     &= \sum_{hK \in (H/K)^W} h \cdot s(a_{W^h} \otimes a'_{W^h})\\
                                     &= \sum_{hK \in (H/K)^W} h \cdot s(a_{W^h} \otimes
                                       \tilde{F}^H_K(a')_{W^h})\\
                                     &= \tilde{V}^H_K(a \tildestar \tilde{F}^H_K(a'))_W
\end{align*}
as desired, where the third equality holds because $a'$ is fixed by $H$
and the fifth equality holds because $W^h \le K$ for $hK \in (H/K)^W$.

Next we need to check that for $m \in M^{\otimes_R G/H}$ and $m' \in
M'^{\otimes_{R'} G/H}$ we have $\tilde{\tau}_{G/H}(m) \tildestar
\tilde{\tau}_{G/H}(m') = \tilde{\tau}_{G/H}(s(m \otimes m'))$. Indeed we have
\begin{align*}(\tilde{\tau}_{G/H}(m) \tildestar \tilde{\tau}_{G/H}(m'))_W &=
                                                                            s(f_{G/H}(m^{\otimes_R H/W}) \otimes f_{G/H}(m'^{\otimes_{R'} H/W}))\\
                                                                          &= f_{G/H}(s(m \otimes m')^{\otimes_{R \otimes R'} H/W})\\
  &= \tilde{\tau}_{G/H}(s(m \otimes m'))_W \text{.}\end{align*}

We return to our earlier aim of showing that $\tilde{V}^H_{V}
\tilde{\tau}_{G/V}(q) \tildestar \tilde{V}^H_{V'} \tilde{\tau}_{G/V'} (q')$ is in the
image of the ghost map. Using the identities proved above and in
Propositions~\ref{prop:fvc_identities} and \ref{prop:witt_teichmuller_properties}, we can rewrite this expression in terms of the other
operators on ghost components, and since the other operators lift to the Witt
vectors we get a preimage under the ghost map. We have
\begingroup \allowdisplaybreaks
\begin{align*}\tilde{V}^H_{V} \tilde{\tau}_{G/V}(q) \tildestar \tilde{V}^H_{V'} \tilde{\tau}_{G/V'}(q')
  &= \tilde{V}^H_V \big(\tilde{\tau}_{G/V}(q) \tildestar \tilde{F}^H_V\tilde{V}^H_{V'} \tilde{\tau}_{G/V'}(q')\big)\\
  &= \tilde{V}^H_V \bigg(\tilde{\tau}_{G/V}(q) \tildestar \Big(\sum_{VhV' \in V\backslash H/V'} \tilde{V}^V_{V \cap \prescript{h}{}{V'}} \tilde{c}_h \tilde{F}^{V'}_{V^h \cap V'} \tilde{\tau}_{G/V'}(q')\Big)\bigg)\\
  &= \sum_{VhV' \in V\backslash H/V'} \tilde{V}^H_V \tilde{V}^V_{V \cap \prescript{h}{}{V'}} \big( \tilde{F}^V_{V \cap \prescript{h}{}{V'}} \tilde{\tau}_{G/V}(q) \tildestar \tilde{c}_h \tilde{F}^{V'}_{V^h \cap V'} \tilde{\tau}_{G/V'}(q')\big)\\
  &= \sum_{VhV' \in V\backslash H/V'} \tilde{V}^H_{V \cap \prescript{h}{}{V'}} \big( \tilde{F}^V_{V \cap \prescript{h}{}{V'}} \tilde{\tau}_{G/V}(q) \tildestar \tilde{F}^{\prescript{h}{}{V'}}_{V \cap \prescript{h}{}{V'}} \tilde{\tau}_{G/\prescript{h}{}{V'}}(h \cdot q')\big)\\
  &= \sum_{VhV' \in V\backslash H/V'} \tilde{V}^H_{V \cap \prescript{h}{}{V'}} \Big( \tilde{\tau}_{G/(V \cap \prescript{h}{}{V'})} f_{G/V}(q^{\otimes_R V/(V \cap \prescript{h}{}{V'})}) \\*
  &\hspace{8em} \tildestar \tilde{\tau}_{G/(V \cap \prescript{h}{}{V'})} f_{G/\prescript{h}{}{V'}}((h \cdot q')^{\otimes_{R'} \prescript{h}{}{V'}/(V \cap \prescript{h}{}{V'})})\Big)\\
  &= \sum_{VhV' \in V\backslash H/V'} \tilde{V}^H_{V \cap \prescript{h}{}{V'}} \tilde{\tau}_{G/(V \cap \prescript{h}{}{V'})} s \Big(f_{G/V}(q^{\otimes_R V/(V \cap \prescript{h}{}{V'})})\\*
  &\hspace{10em} \otimes f_{G/\prescript{h}{}{V'}}((h \cdot q')^{\otimes_{R'} \prescript{h}{}{V'}/(V \cap \prescript{h}{}{V'})})\Big)\\
  &= w\Bigg(\sum_{VhV' \in V\backslash H/V'} V^H_{V \cap \prescript{h}{}{V'}} \tau_{G/(V \cap \prescript{h}{}{V'})} s \Big(f_{G/V}(q^{\otimes_R V/(V \cap \prescript{h}{}{V'})})\\*
  &\hspace{10em} \otimes f_{G/\prescript{h}{}{V'}}((h \cdot q')^{\otimes_{R'} \prescript{h}{}{V'}/(V \cap \prescript{h}{}{V'})})\Big)\Bigg) \text{.}
\end{align*}
\endgroup

So the unique lift $\star$ of $\tildestar$ does exist. And since the identities
analogous to (i) and (ii) hold on ghost components, we conclude that (i) and
(ii) hold for $\star$.

Now we need to check that $\text{gh}^S_{H \le G}$ and $W^S_{H \le G}$ really are
lax symmetric
monoidal. Note $(wu)(1)_U
= 1$ for all $U \in S$. It is easy to check that the maps $\tildestar$ and
$wu$ make $\text{gh}^S_{H \le G}$ into a lax symmetric monoidal
functor (this follows from monoidal properties of
products and fixed points, together with the shuffle isomorphism). By uniqueness of lifting all the relevant symmetry, associativity and unitality
identities must also hold for $\star$ and $u$, so $W^S_{H \le G}$ is lax
symmetric monoidal.
\end{proof}
\begin{remark}
  When we work with free coefficients $(T; Q)$ and $(T'; Q')$ we similarly have
  \[\tau^f_{G/H}(q) \star \tau^f_{G/H}(q') = \tau^f_{G/H}(s(q \otimes q'))\]
  for $q \in Q^{\otimes_T G/H}$ and $q' \in Q'^{\otimes_{T'} G/H}$.
\end{remark}
\begin{remark} \label{rem:multiplication}
  As in \cite{dotto_witt_2022} (e.g.\ Corollary~1.28) this gives us more algebraic
  structure on the Witt vectors. Let $\mu_R : R
  \otimes R \to R$ be the multiplication map for $R$. Then $W^S_{H \le G}(R; R)$
  is a commutative ring with multiplication
  \[W^S_{H \le G}(R; R) \otimes W^S_{H \le G}(R; R) \xrightarrow{\star}
    W^S_{H \le G}(R \otimes R; R \otimes R) \xrightarrow{(\mu_R, \mu_R)_\ast}
    W^S_{H \le G}(R; R) \text{.}\]
  We will see in the next section that this recovers the ring structure on
  the $G$-typical Witt vectors of \cite{dress_burnside_1988}. Let $l_M : R
  \otimes M \to M$ be the $R$-module action map. Then $W^S_{H \le G}(R; M)$ is a $W^S_{H \le G}(R;
  R)$-module, with action
  \[W^S_{H \le G}(R; R) \otimes W^S_{H \le G}(R; M) \xrightarrow{\star}
    W^S_{H \le G}(R \otimes R; R \otimes M) \xrightarrow{(\mu_R, l_M)_{\ast}}
    W^S_{H \le G}(R; M) \text{.}\]
\end{remark}

\begin{proposition} \label{prop:fc_sym_mon}
The Frobenius and conjugation operators are monoidal. That is, for $m \in W^S_{H
  \le G}(R; M)$ and $m' \in W^S_{H \le G}(R'; M')$ we have
\[F^H_K(m \star m') = F^H_K(m) \star F^H_K(m') \quad c_g(m \star m') = c_g(m)
  \star c_g(m') \text{.}\]
\end{proposition}
\begin{proof}
  These are straightforward to check on ghost components.
\end{proof}

\subsection{Relation to previous definitions}\label{sec:generalise}

We are now ready to check that our construction really does generalise and/or
overlap with those from prior work described in Section~\ref{sec:prior_work}.

First we will show that our construction generalises the Witt vectors with
coefficients from \cite{dotto_witt_2022} and \cite{dotto_witt_2025} (in the
case of a module over a commutative ring).
Let $R$ be a commutative ring and $M$ an $R$-module. Let us recall the definition
of the Hausdorff topological abelian group $W(R; M)$ of big Witt vectors with
coefficients in $M$ from \cite{dotto_witt_2022}, which simplifies slightly for
$R$ commutative. Define $\hat{T}(R; M)$ to be the
completed tensor algebra $\prod_{i \ge 0} M^{\otimes_R i}$, where we think of
elements as power series $a_0 + a_1 t + a_2 t^2 + \dotsb$ with $a_i \in
M^{\otimes_R i}$. Let $\hat{S}(R; M)$ be the multiplicative subgroup of
elements with constant term $a_0 = 1$. Then $W(R; M)$ is defined to be the
abelianisation $\hat{S}(R; M)^{\text{ab}}$.

\begin{proposition}\label{prop:generalise_dknp}
  Let $\hat{\mathbb{Z}}$ denote the profinite completion of the integers (considered as an additive group).
  We have an isomorphism of topological abelian groups
  \[W_{\hat{\mathbb{Z}}}(R; M) \cong W(R; M) \text{,}\]
  where the left hand side is our $\hat{\mathbb{Z}}$-typical Witt vectors with
  coefficients, and the right hand side is the big Witt vectors with
  coefficients of \cite{dotto_witt_2022}. This isomorphism respects the ghost
  maps out of each side, as well as the monoidal structure.
\end{proposition}
\begin{proof}
  Let $(T; Q) \in \text{Mod}$ be free. We can write
  $W_{\hat{\mathbb{Z}}}(T; Q)$ as a quotient
  \[\prod_{i=1}^\infty Q^{\otimes_T i} \xrightarrowdbl{}
    W_{\hat{\mathbb{Z}}}(T; Q)\]
  where the index $i$ corresponds to the subgroup $i\hat{\mathbb{Z}} \le
  \hat{\mathbb{Z}}$ in our usual indexing. This quotient depends on a choice
  of coset representatives for each subgroup; we will use $\{0, 1, \dotsc,
  i-1\}$ as coset representatives for $\hat{\mathbb{Z}}/i \hat{\mathbb{Z}}$.

  Define a continuous map
  \begin{align*}\gamma : \prod_{i = 1}^\infty Q^{\otimes_T i} &\to \hat{S}(T; Q) \xrightarrowdbl{} \hat{S}(T; Q)^{\text{ab}} = W(T; Q)\\
    (n_i) &\mapsto \prod_{i=1}^\infty (1 - n_i t^i)
  \end{align*}
  We claim that the map $\gamma$ descends to an isomorphism of topological abelian groups $W_{\hat{\mathbb{Z}}}(T; Q) \to W(T; Q)$.
  By Proposition~1.12 of \cite{dotto_witt_2022}, any element of $\hat{S}(T; Q)$
  can be written in the form
  \[\prod_{i = 1}^\infty (1 - n_i t^i)\]
  so $\gamma$ is surjective.
  The analogue of the ghost map in \cite{dotto_witt_2022} is a continuous group
  homomorphism $\text{tlog} : W(T; Q) \to \prod_{i = 1}^\infty (Q^{\otimes_T
    i})^{C_i} \cong \text{gh}_{\hat{\mathbb{Z}}}(T; Q)$.
  For $x \in Q^{\otimes_T i}$, $\text{tlog}$ sends $1 - x t^i$ to
  \[\text{tr}_e^{C_i} x t^i +
  \text{tr}_{C_2}^{C_{2i}} x^{\otimes_T 2} t^{2i} + \text{tr}_{C_3}^{C_{3i}}
  x^{\otimes_T 3} t^{3i} + \dotsc\]
(where we use a power series notation for elements of
  the codomain). We see that the $j$th component of $\text{tlog}(\gamma(n))$ is
  \[\sum_{i \mid j} \text{tr}_{C_{j/i}}^{C_j} n_i^{\otimes_T j/i} \text{,}\]
  which agrees with our usual ghost map $w$ (note this relies on the choice of coset
  representatives we made earlier). So we have a commutative diagram
  \[\begin{tikzcd}
     \prod_{i = 1}^\infty Q^{\otimes_T i} \ar[r, "\gamma"] \ar[rd, "w" swap] & W(T; Q)
     \ar[d, "\text{tlog}"]\\
     & \prod_{i \ge 1} (Q^{\otimes_T i})^{C_i}
    \end{tikzcd}\]
  Since $(T; Q)$ is free, the topological group $W_{\hat{\mathbb{Z}}}(T; Q)$ is
  isomorphic to the image of $w$. But by Proposition~1.18 of
  \cite{dotto_witt_2022}, $W(T; Q)$ is isomorphic to the image of $\text{tlog}$.
  Since $\gamma$ is surjective, the images of $w$ and $\text{tlog}$ coincide,
  so we see that $\gamma$ descends to an
  isomorphism of topological groups $W_{\hat{\mathbb{Z}}}(T; Q) \cong W(T; Q)$.

  Since both sides preserve reflexive coequalisers (Proposition~1.14 of \cite{dotto_witt_2022}\footnote{In
    \cite{dotto_witt_2022} the authors work in the category of all bimodules over not
    necessarily commutative rings, whereas we work in the full subcategory spanned by modules
    over commutative rings; but the subcategory inclusion preserves reflexive
    coequalisers, since in both cases reflexive coequalisers can be
    computed by taking the coequaliser of the underlying sets.}), $\gamma$ extends
  to give a natural isomorphism for all choices of coefficients, proving the proposition.

  The external products are defined as lifts of the same maps on ghost components, so the monoidal
  structures agree.
\end{proof}

We also want to show that this isomorphism respects the Witt vector operators,
but need to be a little careful to make sure the domains and codomains line up.

The Verschiebung and Frobenius operators in \cite{dotto_witt_2022} are defined
between the Witt vectors with coefficients in $M$ and $M^{\otimes_R n}$:
\begin{align*}
  V_n &: W(R; M^{\otimes_R n}) \to W(R; M)\\
  F_n &: W(R; M) \to W(R; M^{\otimes_R n}) \text{.}
\end{align*}
The authors define a $C_n$ action on $W(R; M^{\otimes_R n})$ and maps
\[\tau_n : M^{\otimes_R n} \to W(R; M) \text{.}\]

The previous proposition gives us isomorphisms
\[\gamma_1 : W_{\hat{\mathbb{Z}}}(R; M) \cong W(R; M)\]
and
\[\gamma_2 : W_{\hat{\mathbb{Z}}}(R; M^{\otimes_R n}) \cong W(R; M^{\otimes_R n}) \text{.}\]
However our Frobenius and Verschiebung don't go between these groups. Instead,
we need to use the isomorphism
\begin{equation}\label{eq:delta_iso} \delta : W_{\hat{\mathbb{Z}}}(R; M^{\otimes_R n}) \cong W_{n\hat{\mathbb{Z}}}(R; M^{\otimes_R n}) \cong W_{n \hat{\mathbb{Z}} \le \hat{\mathbb{Z}}}(R; M)\end{equation}
where the first isomorphism holds since $\hat{\mathbb{Z}} \cong
n\hat{\mathbb{Z}}$ as additive groups, and the second isomorphism is defined
in Lemma~\ref{lem:HGHHiso}. The second isomorphism depends on a choice of coset
representatives for $\hat{\mathbb{Z}}/n \hat{\mathbb{Z}}$; we use the standard
choice $\{0, 1, \dotsc, n-1\}$.

\begin{proposition} \label{prop:generalise_dknp_operators}
  We can relate operators of the big Witt vectors with coefficients and our Witt vectors using these isomorphisms: we have
\begin{align*}
  V_n &= \gamma_1 V^{\hat{\mathbb{Z}}}_{n \hat{\mathbb{Z}}}\delta \gamma_2^{-1}\\
  F_n &= \gamma_2 \delta^{-1} F^{\hat{\mathbb{Z}}}_{n \hat{\mathbb{Z}}} \gamma_1^{-1}\\
  \tau_n &= \gamma_1 V^{\hat{\mathbb{Z}}}_{n \hat{\mathbb{Z}}} \tau_{n\hat{\mathbb{Z}}/\hat{\mathbb{Z}}}
\end{align*}
(where $\tau_{n\hat{\mathbb{Z}}/\hat{\mathbb{Z}}}$ is defined using the same
standard choice of coset representatives), and the $C_n$ action on $W(R; M^{\otimes_R n})$ agrees with the $C_n \cong
\hat{\mathbb{Z}}/n \hat{\mathbb{Z}}$ conjugation action on $W_{n \hat{\mathbb{Z}} \le
  \hat{\mathbb{Z}}}(R; M)$ under the isomorphism $\gamma_2 \delta^{-1}$.
\end{proposition}
\begin{proof}
  In Section~\ref{sec:operators} we defined the various operators on the
  Witt vectors as the unique lifts of certain maps on the ghost groups. The
  isomorphism $\delta$ is similarly defined in terms of an isomorphism of ghost groups. In
  \cite{dotto_witt_2022} the authors also describe the compatibility of their operators
  with certain maps on ghost groups. Using this it is straightforward to check
  everything. %

  The only point of subtlety is that we needed to use the right choice of coset
  representatives for $\hat{\mathbb{Z}}/n\hat{\mathbb{Z}}$, such that both definitions internally use the isomorphism
  $M^{\otimes_R ni} \cong (M^{\otimes_R n})^{\otimes_R i}$ given by
  \[m_0 \otimes_R \dotsb \otimes_R m_{ni-1} \cong (m_0
  \otimes_R \dotsb \otimes_R m_{n-1}) \otimes_R \dotsb \otimes_R (m_{n(i-1)}
  \otimes_R \dotsb \otimes_R m_{ni-1})\]
  (see just before Proposition~1.24 of \cite{dotto_witt_2022}).
\end{proof}

There are also truncated versions of the big Witt vectors with coefficients.
In \cite{dotto_witt_2022} these are denoted by $W_S(R; M)$, where $S$
is a set of positive natural numbers closed under taking divisors. Note that when we
identify $i \in \mathbb{N}_{>0}$ with the subgroup $i\hat{\mathbb{Z}}$, such a
set $S$ corresponds precisely to a truncation set for $\hat{\mathbb{Z}}$ in
our terminology.

\begin{proposition} \label{prop:generalise_dknp_truncated}
  The isomorphism $W_{\hat{\mathbb{Z}}}(R; M) \cong W(R; M)$ induces an isomorphism of truncated Witt vectors
  \[W_{\hat{\mathbb{Z}}}^S(R; M) \cong W_S(R; M) \text{.}\]
  We abuse notation to let $S$ refer to a set of subgroups on the left and
  a set of positive natural numbers on the right, where $i
  \in \mathbb{N}_{>0}$ corresponds to $i\hat{\mathbb{Z}} \le \hat{\mathbb{Z}}$.
  This isomorphism respects the operators in the truncated setting.
\end{proposition}
\begin{proof}
  We know that $W_{\hat{\mathbb{Z}}}^S(R; M)$ preserves reflexive coequalisers, and so
  does $W_S(R; M)$ (since $W(R; M)$ does, and $W_S(R; M)$ is defined as a
  quotient of $W(R; M)$).
  So in order to prove that the quotient maps $R : W(R; M) \xrightarrowdbl{} W_S(R; M)$
  and $R_S :
  W_{\hat{\mathbb{Z}}}(R; M) \xrightarrowdbl{} W^S_{\hat{\mathbb{Z}}}(R; M)$ are
  isomorphic, it suffices to consider the case of free coefficients.

  Indeed for $(T; Q)$ free, $W(T; Q)$ embeds in the ghost group $\prod_{i
    = 1}^{\infty} (Q^{\otimes_T i})^{C_i} \cong \text{gh}_{\hat{\mathbb{Z}}}(T; Q)$.
  The truncated Witt vectors $W_S(T; Q)$ embed in $\prod_{i \in S} (Q^{\otimes_T
    i})^{C_i} \cong \text{gh}^S_{\hat{\mathbb{Z}}}(T; Q)$, and by Lemma~1.41 of \cite{dotto_witt_2022}
  the quotient map $R : W(T; Q) \xrightarrowdbl{} W_S(T; Q)$ is the restriction of the
  projection map
  $\tilde{R}_S : \text{gh}_{\hat{\mathbb{Z}}}(T; Q) \xrightarrowdbl{}
  \text{gh}^S_{\hat{\mathbb{Z}}}(T; Q)$ (recall
  Proposition~\ref{prop:truncation} for the definition of $\tilde{R}_S$).
  But Proposition~\ref{prop:generalise_dknp} showed that the embedding of $W_{\hat{\mathbb{Z}}}(T; Q)$ into
  $\text{gh}_{\hat{\mathbb{Z}}}(T; Q)$ agrees with the embedding of $W(T; Q)$,
  and by Proposition~\ref{prop:truncation} the
  projection $\tilde{R}_S$ also restricts to give the quotient $R_S :
  W_{\hat{\mathbb{Z}}}(T; Q) \xrightarrowdbl{} W^S_{\hat{\mathbb{Z}}}(T; Q)$. So
  these two quotient maps are isomorphic.

  The isomorphism $W_{\hat{\mathbb{Z}}}^S(R; M) \cong W_S(R; M)$ respects the
  operators in the same manner as Proposition~\ref{prop:generalise_dknp_operators}, since in both cases the operators on untruncated Witt vectors descend to the quotient to give the operators on
  truncated Witt vectors. For example consider the Verschiebung in \cite{dotto_witt_2022}
  \[V_n : W_{S/n}(R; M^{\otimes_R n}) \to W_S(R; M)\]
  where $S/n \coloneqq \{k \in \mathbb{N}_{> 0} \mid nk\in S\}$. We have shown
  that we have isomorphisms $\gamma_1^S : W^S_{\hat{\mathbb{Z}}}(R; M) \cong
  W_S(R; M)$ and $\gamma_2^S: W_{\hat{\mathbb{Z}}}^{S/n}(R; M^{\otimes_R n})
  \cong W_{S/n}(R; M^{\otimes_R n})$. Analogous to Equation~\ref{eq:delta_iso} we
  have an isomorphism
  \[\delta^S : W^{S/n}_{\hat{\mathbb{Z}}}(R; M^{\otimes_R n}) \cong W^{S \mid_{n
     \hat{\mathbb{Z}}}}_{n \hat{\mathbb{Z}}}(R; M^{\otimes_R n}) \cong W_{n
   \hat{\mathbb{Z}} \le \hat{\mathbb{Z}}}^{S \mid_{n \hat{\mathbb{Z}}}}(R; M) \text{,}\]
since $S/n$ corresponds to the set of subgroups of $\hat{\mathbb{Z}}$ given by
$\{k \hat{\mathbb{Z}} \mid nk \hat{\mathbb{Z}}
\in S\}$, and on applying the multiplication-by-$n$ isomorphism $\hat{\mathbb{Z}} \cong
n\hat{\mathbb{Z}}$ we get the set of subgroups of $n\hat{\mathbb{Z}}$ given by
$\{k(n\hat{\mathbb{Z}}) \mid k(n\hat{\mathbb{Z}}) \in S\} = S\!\!\mid_{n \hat{\mathbb{Z}}}$.
Our Witt vectors have a Verschiebung map
\[V_{n\hat{\mathbb{Z}}}^{\hat{\mathbb{Z}}} : W_{n \hat{\mathbb{Z}} \le \hat{\mathbb{Z}}}^{S
    \mid_{n \hat{\mathbb{Z}}}}(R; M) \to W_{\hat{\mathbb{Z}}}^S(R; M) \text{.}\]
Since both the operators and all these isomorphisms commute with the
quotient maps from the untruncated Witt vectors, it follows from
Proposition~\ref{prop:generalise_dknp_operators} that
\[V_n = \gamma_1^S V^{\hat{\mathbb{Z}}}_{n\hat{\mathbb{Z}}} \delta^S
  (\gamma_2^S)^{-1} \text{.} \qedhere\]
\end{proof}

The $p$-typical Witt vectors with coefficients of \cite{dotto_witt_2025} can
be defined in terms of the truncated big Witt vectors with coefficients via
\[W_{n+1, p}(R; M) = W_{\{1, p, \dotsc, p^n\}}(R; M)\]
and
\[W_{\infty, p}(R; M) = W_{\{1, p, p^2, \dotsc\}}(R; M) \text{,}\]
and so by Proposition~\ref{prop:generalise_dknp_truncated} these can be recovered from our
truncated $\hat{\mathbb{Z}}$-typical Witt vectors with coefficients. However
it's worth observing that they also arise as untruncated $G$-typical Witt
vectors with coefficients.

\begin{proposition} \label{prop:generalise_dknp_p}
  We have
  \[W_{n+1, p}(R; M) \cong W_{C_{p^n}}(R; M)\]
  and
  \[W_{\infty, p}(R; M) \cong W_{\mathbb{Z}_p}(R; M) \text{,}\]
  where $\mathbb{Z}_p$ denotes the $p$-adic integers (considered as an additive group).
\end{proposition}
\begin{proof}
  We will prove the latter isomorphism; the former is similar. By
  Proposition~\ref{prop:generalise_dknp_truncated}, we already have
  $W_{\infty, p}(R; M) \cong W^S_{\hat{\mathbb{Z}}}(R; M)$ where $S
  = \{\hat{\mathbb{Z}}, p \hat{\mathbb{Z}}, p^2 \hat{\mathbb{Z}}, \dotsc\}$. So it suffices to
  check that $W_{\mathbb{Z}_p}(R; M) \cong W^S_{\hat{\mathbb{Z}}}(R; M)$. Let
  $H = \bigcap_{i \ge 0} p^i \hat{\mathbb{Z}}$. Considering the definition of
  $W^S_{\hat{\mathbb{Z}}}(R; M)$, it only depends on the group
  $\hat{\mathbb{Z}}$ via the quotients of
  $\hat{\mathbb{Z}}$ by subgroups in $S$, so will be left unchanged up to
  isomorphism if we replace $\hat{\mathbb{Z}}$ by $\hat{\mathbb{Z}}/H$ and $p^i
  \hat{\mathbb{Z}}$ by $p^i \hat{\mathbb{Z}}/H$.
  The Chinese remainder
  theorem gives an isomorphism
  \[\hat{\mathbb{Z}} \cong \prod_{q \text{ prime}} \mathbb{Z}_q \text{,}\]
  and noting that $p$ is invertible in $\mathbb{Z}_q$ for $q \ne p$ shows
  $H \cong \prod_{q \ne p} \mathbb{Z}_q$ and $\hat{\mathbb{Z}}/H \cong \mathbb{Z}_p$. Under this isomorphism
  $\{\hat{\mathbb{Z}}/H, p \hat{\mathbb{Z}}/H, p^2 \hat{\mathbb{Z}}/H, \dotsc\}$
  becomes $\{\mathbb{Z}_p, p \mathbb{Z}_p, p^2 \mathbb{Z}_p, \dotsc\}$, i.e.\ the
  set of all open subgroups of $\mathbb{Z}_p$. The proposition follows.
\end{proof}

Considering the case of a commutative ring $R$ seen as a module over itself, we
recover the ring of $G$-typical Witt vectors $W_G(R)$ of \cite{dress_burnside_1988}.

\begin{proposition}\label{prop:generalise_burnside}
  Let $S$ be the set of all open subgroups of a profinite group $G$, and let $\underline{S}$ be a set of conjugacy class representatives. By definition $W_G(R)$ has underlying set
  $\prod_{V \in \underline{S}} R$.

  Then the isomorphism of underlying sets
  \[W_{G}(R; R) \cong W_G(R)\]
  constructed in Lemma~\ref{lem:ring_case_set_iso} is in fact an isomorphism of
  commutative rings,
  where $W_G(R; R)$ has ring structure as described in
  Remark~\ref{rem:multiplication}. This isomorphism respects the Frobenius and Verschiebung operators.
\end{proposition}
\begin{proof}
  It is straightforward to check that the ghost maps for $W_G(R; R)$ and
  $W_G(R)$ agree, noting that the ghost map out of $W_G(R; R) \cong \prod_{V \in
    \can{S}} R$ is given explicitly by the formula in
  Definition~\ref{def:ghost_map}. Also note that the ring structure of
  Remark~\ref{rem:multiplication} makes the ghost components into ring homomorphisms.
  By the main theorem of \cite{dress_burnside_1988} %
  this uniquely defines the ring structure on values of the functor $W_G({-}) : \text{CRing} \to \text{CRing}$, so the two constructions are isomorphic.

  Let us show that the isomorphism respects the Frobenius and Verschiebung operators in a
  similar manner to Proposition~\ref{prop:generalise_dknp_operators}.
  Lemma~\ref{lem:HGHHiso} gives an isomorphism
  \[\delta : W_{H \le G}(R; R) \cong W_H(R; R^{\otimes_R G/H}) \cong W_H(R; R) \text{.}\]
  Since $R^{\otimes_R G/H \times H/U}$ and $R^{\otimes_R G/U}$ are both
  canonically isomorphic to $R$, the isomorphism $\delta$ is also canonical, so
  we can consider our Frobenius and Verschiebung as maps
  \[F_H^G : W_G(R; R) \to W_H(R; R) \quad \quad V_H^G : W_H(R; R) \to W_G(R; R) \text{.}\]
  By Lemma~\ref{lem:finite_module_kan} these are the unique lifts of certain
  maps on ghost components. But the Frobenius $f_H$ and Verschiebung $v_H$ of
  \cite{dress_burnside_1988} are lifts of the same maps (see (2.10.4)' in
  \cite{dress_burnside_1988} and the
  preceding note, describing how induction and restriction between Burnside
  rings interact with the ghost map). %
\end{proof}

Recall the completed Burnside ring $\hat{\Omega}(G)$ of a profinite group $G$ from \cite{dress_burnside_1988}, defined to
be the Grothendieck ring of those discrete $G$-spaces $X$ for which the set
$X^U$ of $U$-fixed points is finite for every open subgroup $U$ of $G$.

\begin{corollary} \label{cor:witt_z_burnside}
  We deduce $W_{H \le G}(\mathbb{Z}; \mathbb{Z}) \cong \hat{\Omega}(H)$ is the
  completed Burnside ring of $H$.
\end{corollary}
\begin{proof}
  We have $W_{H \le G}(\mathbb{Z}; \mathbb{Z}) \cong W_H(\mathbb{Z}; \mathbb{Z})
  \cong W_H(\mathbb{Z})$, and in \cite{dress_burnside_1988}
  we see essentially by definition that $W_H(\mathbb{Z}) \cong \hat{\Omega}(H)$ is the completed the Burnside ring.
\end{proof}

For completeness, we record that a special case of the Witt vectors with coefficients gives the usual Witt vectors of a ring.

\begin{proposition}
  We recover the usual ring of truncated big Witt vectors as defined in
  \cite{hesselholt_big_2015}, via
  \[W_S(R) \cong W^S_{\hat{\mathbb{Z}}}(R; R) \text{.}\]
  Similarly, for the $p$-typical Witt vectors we
  have
  \[W_{\infty, p}(R) \cong W_{\mathbb{Z}_p}(R; R)\]
  and for the $n$-truncated $p$-typical Witt vectors we have
  \[W_{n+1, p}(R) \cong W_{C_{p^{n}}}(R; R) \text{.}\]
  These isomorphisms respect the Frobenius and Verschiebung operators.
\end{proposition}
\begin{proof}
  These are special cases of Proposition~\ref{prop:generalise_dknp_truncated}
  and Proposition~\ref{prop:generalise_dknp_p}, or follows from standard
  uniqueness results analogously to
  Proposition~\ref{prop:generalise_burnside}.
\end{proof}

\section{Isotropy separation and Mackey functors} \label{sec:isotropy_and_mackey}

The operators of the usual $p$-typical Witt vectors satisfy various identities and exact
sequences. In this section we will consider the
analogues for our construction, and see how these recall the structure of equivariant stable homotopy groups.

We will see later that when $X$ is a connective spectrum we have $\pi_0(N_{\{e\}}^G X)^{\Phi V} \cong
\pi_0(X^{\wedge G/V}) \cong (\pi_0 X)^{\otimes_{\mathbb{Z}} G/V}$.
One perspective on the Witt vector computation is that it's an attempt to
recover information about the norm based on the geometric fixed points---indeed in
some sense $W_G(\mathbb{Z}; \pi_0 X)$ is built from copies of $(\pi_0
X)^{\otimes_\mathbb{Z} G/V}$ for varying $V$. In order
to gain some leverage on how a spectrum is related to its geometric fixed
points we use a technique called isotropy separation. After recalling isotropy
separation of spectra and the corresponding exact sequence of zeroth homotopy groups, we describe analogous ideas for studying Mackey functors. We then prove that the Witt
vectors together with their operators define Mackey functors, and find an
exact sequence relating different truncations, letting us understand the ``isotropy
separation'' of these Mackey functors.

We also continue our study of the monoidal structure of the Witt vectors,
showing that the untruncated Witt vectors give a strong symmetric monoidal
functor $\text{Mod} \to \text{Mack}_G(\text{Ab})$.

\subsection{Isotropy separation of spectra}

In this section we recall the technique of isotropy separation in order to
establish notation and collect together the required results. All of these results
appear in or can be easily derived from standard references in equivariant
stable homotopy theory \cite{mandell_equivariant_2002, hill_nonexistence_2016, schwede_lectures_2023}.

Let $G$ be a finite group, $H$ a subgroup of $G$ and $S$ a truncation set for $H$. Define a family of subgroups
\[\mathcal{F}(S) = \{U \le H \mid U \not\in S\}\text{.}\]
Recall that the classifying $H$-space
\[E \mathcal{F}(S)\]
has $(E \mathcal{F}(S))^U$ contractible for $U \in \mathcal{F}(S)$ and empty for $U
\not\in \mathcal{F}(S)$. Defining
\[{\tilde{E} \mathcal{F}(S)}\]
to be the
cofibre of the based map $E \mathcal{F}(S)_{+} \to S^0$, we see that
$(\tilde{E} \mathcal{F}(S))^U$ is homotopy equivalent to $S^0$ for $U \in S$ and
contractible otherwise. The cofibre sequence
\[E \mathcal{F}(S)_{+}
  \to S^0 \to {\tilde{E} \mathcal{F}(S)}\]
is called the isotropy separation sequence.

Let $Y$ be an $H$-spectrum. One way to define the geometric fixed points is as
follows.

\begin{definition}[Geometric fixed points]
  The geometric fixed points $Y^{\Phi H}$ of $Y$ are defined to be the spectrum
  \[(Y \wedge {\tilde{E} \mathcal{F}(\{H\})})^H \text{.}\]
\end{definition}

Given truncation sets $S' \subseteq S$ there is a canonical map $\tilde{E}
\mathcal{F}(S) \to {\tilde{E} \mathcal{F}(S')}$, constructed by applying
$\tilde{E}\mathcal{F}(S) \wedge {(-)}$ to the map
$S^0 \to \tilde{E} \mathcal{F}(S')$ from the isotropy separation sequence (note
$\tilde{E}\mathcal{F}(S) \wedge \tilde{E}\mathcal{F}(S') \simeq
\tilde{E}\mathcal{F}(S')$ and $\tilde{E}\mathcal{F}(S)$ is an idempotent).

\begin{definition}[Truncation map] \label{def:truncation}
The canonical map $\tilde{E}\mathcal{F}(S) \to \tilde{E}\mathcal{F}(S')$ induces
a map of spectra
\[R_{S'} : Y \wedge \tilde{E}\mathcal{F}(S) \to Y \wedge
  \tilde{E}\mathcal{F}(S') \text{.}\]
This induces maps of fixed points and maps of equivariant homotopy groups, which we will also denote $R_{S'}$.
\end{definition}

Note that given $S'' \subseteq S' \subseteq S$ we have that $R_{S''}R_{S'} = R_{S''}$
as maps $Y \wedge \tilde{E}\mathcal{F}(S) \to Y \wedge \tilde{E}\mathcal{F}(S'')$,
since the composition ${\tilde{E} \mathcal{F}(S)} \to \tilde{E} \mathcal{F}(S') \to {\tilde{E} \mathcal{F}(S'')}$ is the same as the
canonical map ${\tilde{E} \mathcal{F}(S)} \to \tilde{E} \mathcal{F}(S'')$.

Now suppose $Y$ is a $G$-spectrum but $S$ is a truncation set for the
subgroup $H$.
The $H$-spectrum $Y \wedge \tilde{E}\mathcal{F}(S)$ comes
equipped with a conjugation action of $H$ on fixed point spectra. However the action of $G$ on
$Y$ also induces more general conjugation maps $c_g$ for all $g \in G$.

\begin{definition}[Conjugation map] \label{def:conjugation}
  There is a conjugation map $c_g$ defined by the composite
  \[(Y \wedge \tilde{E}\mathcal{F}(S))^H \cong (c_g^\ast Y \wedge
    \tilde{E}\mathcal{F}(\prescript{g}{}{S}))^{\prescript{g}{}{H}} \cong
    (Y \wedge \tilde{E}\mathcal{F}(\prescript{g}{}{S}))^{\prescript{g}{}{H}}\]
  where the first isomorphism comes from the inner automorphism $c_g : G
  \to G$ given by conjugation by $g$, and the second isomorphism is induced by the
  left action of $g$ on $Y$.
\end{definition}

These
interact as you would expect with transfers, restrictions and truncation maps: we have $\text{tr}^{\prescript{g}{}{L}}_{\prescript{g}{}{K}} \,
c_g = c_g \, \text{tr}^L_K$,
$\text{res}^{\prescript{g}{}{L}}_{\prescript{g}{}{K}} \, c_g = c_g \, \text{res}^L_K$ and
$R_{\prescript{g}{}{S'}} \, c_g = c_g \, R_{S'}$.

We will need to study the homotopy groups of $(Y \wedge \tilde{E} \mathcal{F}(S))^H$. When $Y$ is connective, we can use the
isotropy separation sequence to obtain an exact sequence of zeroth homotopy groups.

\begin{lemma} \label{lem:isotropy}
  Let $K$ be a subgroup of $H$. Starting with the $H$-truncation set $S$, we get an $H$-truncation set
  \[S \setminus K \coloneqq \{U \in S \mid \text{$U$ is not subconjugate to $K$}\} \text{,}\]
  and a $K$-truncation set
  \[S\!\!\mid_K \coloneqq \{U \in S \mid U \le K\}\text{.}\]
  Then for $Y$ a connective $H$-spectrum, we get an exact sequence of homotopy groups
  \[\pi^K_0(Y \wedge {\tilde{E} \mathcal{F}(S \!\!\mid_K)}) \xrightarrow{\text{tr}_K^H} \pi^H_0(Y \wedge
    {\tilde{E} \mathcal{F}(S)}) \xrightarrow{R_{S \setminus K}} \pi^H_0(Y \wedge \tilde{E} \mathcal{F}(S \setminus K)) \to 0 \text{.}\]
\end{lemma}
\begin{proof}
  Given a subgroup $K \le H$, define
  \[\mathcal{F}(K, H) = \{U \le H \mid \text{$U$ is subconjugate to $K$}\} \text{.}\]

  Isotropy separation gives a cofibre sequence of $H$-spectra
  \[
    (Y \wedge {\tilde{E} \mathcal{F}(S)} \wedge E \mathcal{F}(K, H)_{+})^H \to
    (Y \wedge {\tilde{E} \mathcal{F}(S)})^H \to
    (Y \wedge {\tilde{E} \mathcal{F}(S)} \wedge \tilde{E} \mathcal{F}(K, H))^H \text{.}
  \]
  Considering fixed point spaces shows that
  \[{\tilde{E} \mathcal{F}(S)} \wedge \tilde{E} \mathcal{F}(K, H) \simeq
    {\tilde{E} \mathcal{F}(S \setminus K)} \text{,}\]
  so we get a cofibre sequence
  \[(Y \wedge {\tilde{E} \mathcal{F}(S)}) \wedge E \mathcal{F}(K, H)_{+})^H \to
    (Y \wedge {\tilde{E} \mathcal{F}(S)})^H \xrightarrow{R_{S \setminus K}} (Y \wedge
    {\tilde{E} \mathcal{F}(S \setminus K)})^H \text{.}\]

  When $Y$ is connective, the associated long exact sequence of homotopy groups ends with
  \[\pi^H_0(Y \wedge {\tilde{E} \mathcal{F}(S)} \wedge E \mathcal{F}(K, H)_{+}) \to
    \pi^H_0(Y \wedge {\tilde{E} \mathcal{F}(S)})
    \xrightarrow{R_{S \setminus K}} \pi^H_0(Y \wedge {\tilde{E} \mathcal{F}(S \setminus K)}) \to 0\]
  Let $\mathcal{O}_{K, H}$ denote the full subcategory of the orbit category
  $\mathcal{O}_H$ with objects $H/J$ for $J \in \mathcal{F}(K, H)$.
  We can compute
  \begin{align*}\pi^H_0(Y \wedge {\tilde{E} \mathcal{F}(S)} \wedge E \mathcal{F}(K,
    H)_{+}) &\cong \underset{H/J \in \mathcal{O}_{K, H}}{\text{colim}}\ \pi_0^H(Y
    \wedge {\tilde{E} \mathcal{F}(S)}
    \wedge H/J)\\
    &\cong \underset{H/J \in \mathcal{O}_{K, H}}{\text{colim}}\ \pi^J_0(Y \wedge
    {\tilde{E} \mathcal{F}(S)})\end{align*}
  where the first isomorphism comes from using the model of $E\mathcal{F}(K,
  H)_{+}$ given by Lemma~2.2 of \cite{luck_completion_2001}, and the second
  isomorphism uses the Wirthm\"uller isomorphism.
  Since $H/K$ is a weakly terminal object of $\mathcal{O}_{K, H}$, we see that $\pi^H_0(Y
  \wedge {\tilde{E} \mathcal{F}(S)} \wedge E \mathcal{F}(K, H)_{+})$ is a
  quotient of $\pi^K_0(Y \wedge {\tilde{E} \mathcal{F}(S)}) \cong \pi^K_0(Y
  \wedge {\tilde{E} \mathcal{F}(S \!\!\mid_K)})$.

  So we get the desired exact sequence. One can check that the %
  resulting map $\pi^K_0(Y \wedge {\tilde{E} \mathcal{F}(S \!\!\mid_K)}) \to \pi^H_0(Y \wedge
  {\tilde{E} \mathcal{F}(S)})$ is the usual transfer coming from the Mackey
  functor structure on equivariant stable homotopy groups.%
\end{proof}
\begin{remark} \label{rem:isotropy_orbits_form}
Note that for $h \in N_H(K)$ we have $\text{tr}^H_K c_h = c_h \text{tr}^H_K = \text{tr}^H_K$, since $c_h$ is the identity on $\pi^H_0(Y \wedge \tilde{E}F(S))$. So if we wish we
can replace the first term of the sequence with $\left(\pi^K_0(Y \wedge \tilde{E}\mathcal{F}(S\!\!\mid_K))\right)_{N_H(K)}$, where orbits are taken with respect to the
Weyl group action induced by the conjugation maps.
\end{remark}

\subsection{Isotropy separation of Mackey functors} \label{sec:mackey_isotropy}

Recall that a $G$-Mackey functor $\underline{M}$ can be defined as an assignment of an abelian
group $M(H)$ to each transitive $G$-set $G/H$, together with maps
$\text{tr}_K^H$, $\text{res}_K^H$, $c_g$ between these groups. These maps
satisfy certain axioms; see standard references \cite{bouc_green_1997,
  webb_guide_2000} for the axioms and other ways to define Mackey functors.\footnote{Note
  Webb writes $I^H_K$ for $\text{tr}^H_K$ and $R^H_K$ for $\text{res}^H_K$.} We
will work with the box product of Mackey functors defined by Lewis
in \cite{lewis_theory_nodate}, see also \cite{luca_algebra_1996} for descriptions
of the box product in terms of other definitions of Mackey functors.

Given a $G$-spectrum
$Y$ the assignment $G/H \mapsto \pi^H_0(Y)$ has a canonical Mackey functor
structure, and we denote the resulting Mackey functor by $\underline{\pi}_0(Y)$.
As a result, the isotropy separation techniques we discussed for spectra have
purely algebraic analogues for Mackey functors. The contents of this subsection
is well-known, but we recall it in detail in order to introduce notation compatible with the rest of the paper.

\begin{definition} \label{def:mackey_truncation}
  Let $\underline{M}$ be a $G$-Mackey functor and $S$ a truncation set for $H$.
  We define the $S$-truncation of $\underline{M}$ to be the $H$-Mackey functor
  \[\underline{M}^S \coloneqq \underline{M} \Osq \underline{\pi}_0(\Sigma^\infty
    \tilde{E}\mathcal{F}(S))\]
  where $\Osq$ denotes the box product of Mackey functors and we implicitly
  take the underlying $H$-Mackey functor of $\underline{M}$.
\end{definition}

\begin{remark}
  For $K \le H$, the underlying $K$-Mackey functor of $\underline{\pi}_0(\Sigma^\infty
  \tilde{E}\mathcal{F}(S))$ is isomorphic to $\underline{\pi}_0(\Sigma^\infty
  \tilde{E}\mathcal{F}(S \!\!\mid_K))$, so we have
  \[M^S(K) \cong M^{S \mid_K}(K) \text{.}\]
\end{remark}

\begin{remark} \label{rem:truncations_agree_spectra}
  A smash product of connective spectra corresponds to a box product on zeroth
  homotopy. So for any connective $G$-spectrum $Y$ we have
  \[
    (\underline{\pi}_0 Y)^S \coloneqq
    \underline{\pi}_0(Y) \Osq \underline{\pi}_0(\Sigma^{\infty} \widetilde{E}\mathcal{F}(S)) \cong
    \underline{\pi}_0(Y \wedge \widetilde{E}\mathcal{F}(S))
    \text{.}\]
  In particular we have
  \[\underline{M}^S \cong \underline{\pi}_0(H\underline{M} \wedge \widetilde{E}\mathcal{F}(S))\]
  where $H\underline{M}$ is the equivariant Eilenberg-MacLane spectrum.
\end{remark}

\begin{remark} \label{rem:truncation_box_with_witt}
  We can use tom Dieck splitting to compute $\underline{\pi}_0(\Sigma^\infty
  \tilde{E}\mathcal{F}(S))$.
  When $S$ is the set of all subgroups of $G$, we can model
  $\tilde{E}\mathcal{F}(S)$ by $S^0$ and so $\underline{\pi}_0(\Sigma^\infty
  \tilde{E}\mathcal{F}(S))$ is the Burnside Mackey functor; in general
  $\underline{\pi}_0(\Sigma^\infty \tilde{E}\mathcal{F}(S))$ is a quotient of
  the Burnside Mackey functor, where $\pi^K_0(\Sigma^\infty
  \tilde{E}\mathcal{F}(S))$ is the free abelian group whose generators
  correspond to $K$-conjugacy classes of subgroups of $K$ that are in $S$.
  This Mackey functor $\underline{\pi}_0 (\Sigma^\infty \tilde{E}\mathcal{F}(S))$ can alternatively be written as
  \[H/K \mapsto W^{S \mid_K}_{K \le H}(\mathbb{Z}; \mathbb{Z})\]
  with Mackey functor structure as we will describe in
  Section~\ref{sec:witt_mackey}.
\end{remark}

\begin{remark}\label{rem:mackey_truncation_conjugation}
We have truncation maps $R_{S'} : \underline{M}^S \to \underline{M}^{S'}$ and
conjugation maps $c_g : M^S(H) \to
M^{\prescript{g}{}{S}}(\prescript{g}{}{H})$ defined by applying
$\pi_0$ to the truncation and conjugation maps defined for spectra in
Definitions \ref{def:truncation} and \ref{def:conjugation}. Note when $g \in H$
these conjugation maps agree with the conjugation maps that are part of the
Mackey functor structure of $\underline{M}^S$.
\end{remark}

This construction gives a notion of geometric fixed points of a Mackey
functor.

\begin{definition}
  We define the $H$-geometric fixed points of a Mackey functor $\underline{M}$ by
  \[M^{\Phi H} \coloneqq M^{\{H\}}(H) \text{.}\]
\end{definition}

A special case of Lemma~\ref{lem:isotropy} gives us an exact sequence.

\begin{lemma} \label{lem:mackey_truncation_exact}
  Given $\underline{M}$ a $G$-Mackey functor, $S$ a truncation set for $H$ and
  $K \le H$, we have an exact sequence
  \[M^{S \mid_K}(K) \xrightarrow{\text{tr}^H_K} M^{S}(H) \xrightarrow{R_{S \setminus K}} M^{S \setminus K}(H) \to 0 \text{.}\]
\end{lemma}
\begin{proof}
  Apply Lemma~\ref{lem:isotropy} to the Eilenberg-MacLane spectrum $H\underline{M}$.
\end{proof}

A map of $G$-spectra is a weak equivalence iff it induces weak equivalences on
all geometric fixed point spectra. We have a partial analogue for Mackey
functors: a map of Mackey functors is surjective if it induces surjections on all geometric
fixed points.

\begin{lemma} \label{lem:mackey_isotropy_surjective}
  Let $\alpha : \underline{M} \to \underline{N}$ be a map of $G$-Mackey functors,
  and $S$ a truncation set for $H$. Suppose that $\alpha$
  induces a surjection on geometric fixed points $M^{\Phi K} \to
  N^{\Phi K}$ for all $K \in S$. Then $\alpha$ induces a surjection
  $\underline{M}^S \to \underline{N}^S$.\
\end{lemma}
\begin{proof}
  It suffices to prove that under these hypotheses $\alpha : M^S(H) \to N^S(H)$
  is surjective (then for any $K \in S$, applying the same
  argument to the truncation set $S \!\!\mid_K$ shows that $\alpha: M^{S}(K) \to N^S(K)$
  is also surjective). We
  proceed by induction over the size of $S$, letting $H$ vary.

  When $S = \emptyset$ then
  $\underline{M}^S = \underline{N}^S = 0$. When $\abs{S} = 1$ we must have $S =
  \{H\}$. By definition $M^S(H) =
  M^{\Phi H}$ and by assumption $\alpha : M^{\Phi H} \to N^{\Phi H}$ is
  surjective.

  So suppose $\abs{S} > 1$, and that the statement is true for all smaller
  truncation sets. There must be some proper subgroup $K \le H$ with $K \in S$.
  Then $\alpha$ induces a map of exact sequences
  \[\begin{tikzcd}
    M^{S \mid_K}(K) \ar[r] \ar[d] &  M^{S}(H) \ar[r] \ar[d] & M^{S \setminus
      K}(H) \ar[r] \ar[d] & 0 \ar[d] \\
    N^{S \mid_K}(K) \ar[r] &  N^{S}(H) \ar[r] & N^{S \setminus K}(H) \ar[r] & 0 \text{.}
  \end{tikzcd}\]
By induction the maps $M^{S \mid_K}(K) \to N^{S \mid_K}(K)$ and $M^{S \setminus
  K}(H) \to N^{S \setminus K}(H)$ are surjective, since both involve truncation
sets that are strictly smaller than $S$. The rightmost vertical
map $0 \to 0$ is injective, so by the four lemma we deduce that the map $M^S(H) \to
N^S(H)$ is surjective. Hence the induction holds.
\end{proof}

\subsection{Mackey structure of Witt vectors} \label{sec:witt_mackey}

We begin by observing that the ghost groups and the Witt vectors form Mackey
functors. We also recall that the truncation and more general conjugation operators interact
well with this structure.

\begin{lemma} \label{lem:witt_mackey}
  Let $G$ be a finite group, $H$ a subgroup and $S$ a truncation set for $H$.
  The assignment
  \[H/K \mapsto \text{gh}^{S \mid_K}_{K \le H}(R; M)\]
  together with the $\tilde{V}$, $\tilde{F}$ and $\tilde{c}_g$ operators (as
  defined in Section~\ref{sec:operators}) gives an
  $H$-Mackey functor. In our standard notation for Mackey functor
  operations, $\tilde{V}$ corresponds to transfer $\text{tr}$, $\tilde{F}$ corresponds to
  restriction $\text{res}$ and $\tilde{c}_h$ corresponds to conjugation $c_h$.

  Similarly, for the Witt vectors
  \[H/K \mapsto W^{S \mid_K}_{K \le H}(R; M)\]
  together with the Verschiebung, Frobenius and conjugation operators is a
  Mackey functor. The ghost map $w$ is a map of Mackey functors.
\end{lemma}
\begin{proof}
  All the relevant identities were proved in
  Proposition~\ref{prop:fvc_identities}. Since the operators on Witt vectors
  were defined as lifts of the corresponding maps on the ghost group, the ghost
  map commutes with the Mackey functor structure.
\end{proof}

\begin{definition} \label{def:witt_mackey}
  Let $\underline{W}^S_G(R; M)$ denote the $H$-Mackey functor $H/K \mapsto W^{S \mid_K}_{K \le
    G}(R; M)$ with structure maps as defined in the above lemma, and let $\underline{\text{gh}}^S_G(R; M)$ denote $H/K \mapsto
  \text{gh}^{S \mid_K}_{K \le G}(R; M)$.
\end{definition}

\begin{remark} \label{rem:g_infinite}
Note all the identities hold for arbitrary $G$; we restrict to $G$ finite
  simply because Mackey functors are generally only defined for finite groups. We
  expect that when $G$ is infinite we get a $G$-Mackey profunctor in the sense
  of Kaledin \cite{kaledin_mackey_2022}, but we have not studied this further. Mackey
  profunctors are the homotopy groups of quasifinitely genuine $G$-spectra as
  described in \cite{krause_polygonic_2023}, so one would expect that our Witt vectors
  would compute the zeroth homotopy of a quasifinitely genuine $G$-spectrum
  version of the norm.
\end{remark}

\begin{remark}
  We already observed in Corollary~\ref{cor:witt_z_burnside} that $W_{H \le G}(\mathbb{Z};
  \mathbb{Z})$ is (for $G$ finite) the Burnside ring of $H$. In fact $\underline{W}_G(\mathbb{Z};
  \mathbb{Z})$ is the Burnside Mackey functor. More generally, let $Y$ be a finite
  set, then we can describe $\underline{W}_G(\mathbb{Z}; \mathbb{Z}(Y))$. Recall an alternative characterisation of Mackey functors: let the Burnside category
  $\mathcal{A}^G$ be the category of finite $G$-sets with morphisms given by
  equivalence classes of spans of $G$-equivariant maps, then a Mackey functor is
  an additive functor $\mathcal{A}^G \to \text{Ab}$. %
  Proposition~\ref{prop:witt_free} shows $W^S_{H \le G}(\mathbb{Z}; \mathbb{Z}(Y)) \cong \prod_{V \in \can{S}} \mathbb{Z}((Y^{\times
    G/V})_{N_H(V)})$ and describes the ghost map. Using this we can check that
  $\underline{W}_G(\mathbb{Z}; \mathbb{Z}(Y))$ is the Mackey functor
  represented by the $G$-set $Y^{\times G} \in \mathcal{A}^G$.
\end{remark}

By  Proposition~\ref{prop:trunc_commute} the Witt vector truncation maps
assemble into maps of Mackey functors. This gives a clash of notation between
the Witt vector truncation map $R_S$ as defined in
Proposition~\ref{prop:truncation} and the
Mackey functor truncation map $R_S$ as described in
Remark~\ref{rem:mackey_truncation_conjugation}. Similarly if $S$ is an
$H$-truncation set and $g \notin H$ then we have a clash in
notation between the Witt vector conjugation maps $c_g$ as defined in
Proposition~\ref{prop:conjugation} and the conjugation maps $c_g$
defined between truncated Mackey functors in Remark~\ref{rem:mackey_truncation_conjugation}.
In fact we will show that the truncated Witt vectors agree with Mackey functor
truncation and the truncation and conjugation maps are the same. We start by proving that the truncated Witt vectors satisfy the
usual exact sequences (here $G$ is not necessarily finite).

\begin{lemma} \label{lem:witt_exact}
  Let $K \le_o H \le_o G$, and let $S$ be a truncation set for $H$. Then
  \[W_{K \le G}^{S\mid_K}(R; M)_{N_H(K)} \xrightarrow{V_K^H} W_{H \le G}^S(R; M) \xrightarrow{R_{S \setminus K}} W_{H \le G}^{S \setminus K}(R; M) \to 0\]
  is an exact sequence of Hausdorff topological abelian groups, by which we mean that $\text{im}(V^H_K)
  = \text{ker}(R_{S \setminus K})$ and $R_{S/K}$ is a quotient map. Recall we use $S \!\setminus\! K$ to denote $\{U \in S \mid \text{$U$ is
    not subconjugate to $K$}\}$. Note that $V^H_K$ factors through the group of
  Weyl group orbits as in Remark~\ref{rem:isotropy_orbits_form}.
\end{lemma}
\begin{proof}

  A reflexive coequaliser of right exact sequences in $\text{Ab}_\text{Haus}$ is right
  exact (certainly this is true on underlying abelian groups, and it also holds in
  $\text{Ab}_\text{Haus}$ since reflexive coequalisers preserve quotient maps in
  the sense of Lemma~\ref{lem:space_coeq_quotient} and Corollary~\ref{cor:abgroup_coeq_quotient}).
  So considering a free resolution for $(R; M)$ shows that it suffices to consider the case of free
  coefficients $(T; Q)$. Using the computation of the Witt vectors with free
  coefficients in Proposition~\ref{prop:witt_free}, the sequence becomes
  \[\left( \prod_{V \in \can{S\mid_K}} (Q^{\otimes_T G/V})_{N_K(V)} \right)_{N_H(K)}
    \to \prod_{V \in \can{S}} (Q^{\otimes_T G/U})_{N_H(V)} \to \prod_{V \in
      \can{S \setminus K}} (Q^{\otimes_T G/V})_{N_H(V)} \to 0 \text{.}\]
  Under this isomorphism the Verschiebung and truncation maps become maps
  we described in Remarks~\ref{rem:verschiebung_witt_free} and
  \ref{rem:truncation_witt_free}. The truncation is the obvious projection map,
  and so is a topological quotient. An element $n$ is in the image of the
  Verschiebung iff $n_V$ is only non-zero for $V$ subconjugate to $K$, which
  holds precisely when it is in the
  kernel of the truncation map. So the sequence is exact.
\end{proof}

The exact sequence is most useful in the following special case.

\begin{corollary} \label{cor:witt_exact_minimal}
  When $K$ is a minimal element of $S$ (i.e.\ $S$ does not
  contain any subgroup strictly subconjugate to $K$) then we have an exact sequence
  \[\left(M^{\otimes_R G/K}\right)_{N_H(K)} \xrightarrow{V_K^H \tau_{G/K}} W_{H \le G}^S(R; M) \xrightarrow{R_{S \setminus K}} W_{H \le G}^{S \setminus K}(R; M) \to 0\]
\end{corollary}
\begin{proof}
  Note $S\!\!\mid_K = \{K\}$, and we have an additive isomorphism $\tau_{G/K} :
  M^{\otimes_R G/K} \to W^{\{K\}}_{K \le G}(R; M)$.
\end{proof}

The exact sequence is enough to prove that truncation of Witt vectors agrees with truncation
of Mackey functors, as explained in the following lemma.

\begin{lemma} \label{lem:witt_mackey_truncation}
  Let $G$ finite, $H \le G$ and $S$ a truncation set for $H$. Then the
  $H$-Mackey functor $\underline{W}^{S}_{G}(R; M)$ is the $S$-truncation (as
  defined in Definition~\ref{def:mackey_truncation}) of the $G$-Mackey functor $\underline{W}_G(R; M)$.
  The Witt vector truncation maps $R_{S}$ as defined in Proposition~\ref{prop:truncation} agree with those
  corresponding to truncation of Mackey functors from
  Remark~\ref{rem:mackey_truncation_conjugation}, and the conjugation maps $c_g$
  as defined in Proposition~\ref{prop:conjugation} agree with those of Remark~\ref{rem:mackey_truncation_conjugation}.
\end{lemma}
\begin{proof}
We have $W_{K \le G}(R; M) \coloneqq W^{S_0 \mid_K}_{K \le G}(R; M)$ when $S_0$ is the set of all
subgroups of $H$. The exact sequence of Lemma~\ref{lem:mackey_truncation_exact} is
enough to compute the $S$-truncation of a Mackey functor as a quotient of the
untruncated version. Since the truncated Witt
vectors satisfy analogous exact sequences, we see inductively that the two
notions of truncation must agree.

More precisely, we use the exact sequences to prove by induction on the number of subgroups omitted from
a $K$-truncation set $S'$ (letting both $S'$ and $K$ vary) that the quotients $W_{K \le G}(R; M) \xrightarrow{R_{S'}} W^{S'}_{K \le G}(R; M)$ and $W_{K \le G}(R; M) = \underline{W}_G(R;
M)(K) \xrightarrow{R_{S'}} (\underline{W}_G(R; M))^{S'}(K)$ are isomorphic.
Note for the first quotient we use the Witt truncation map defined in
Section~\ref{sec:operators}, and for the second we use the Mackey truncation map
defined in Section~\ref{sec:mackey_isotropy}. For both notions of truncation
the untruncated operators descend to give the Mackey structure on the quotient, and so
all the Mackey structure also agrees.

By Proposition~\ref{prop:trunc_commute} the Witt vector conjugation operators
commute with truncation, and similarly for the Mackey functor conjugation operators.
The conjugation operators agree on untruncated Witt vectors by the definition of
the Mackey structure on the Witt vectors, so they must also agree on the
truncations (seen as quotients of the untruncated Witt vectors).
\end{proof}

\begin{remark}
The $S$-truncation of the ghost group Mackey functor $G/K \mapsto \text{gh}_{K
  \le G}(R; M)$ does not in general agree with the Mackey functor $H/K \mapsto \text{gh}^{S
  \mid_K}_{K \le G}(R; M)$. To see this, consider the analogue of the
sequence in Lemma~\ref{lem:witt_exact} for the ghost group with $\tilde{V}^H_K$
and $\tilde{R}_{S \setminus K}$; in the case where $K$ is minimal in $S$, the
sequence is exact iff $\text{tr}_{K}^{N_H(K)} : (M^{\otimes_R G/K})^K \to
(M^{\otimes_R G/K})^{N_H(K)}$ is surjective, which may not be true.
\end{remark}

\begin{remark}
  By Remark~\ref{rem:truncation_box_with_witt} we have
  \[\underline{W}^S_G(R; M) \cong \underline{W}^S_G(\mathbb{Z}; \mathbb{Z})
    \Osq \underline{W}_G(R; M) \text{.}\]
\end{remark}

Suppose $T$ is a ring whose additive group is free, and let $Q = T(Y)$ be a free
$T$-module. Let $S$ be a finite truncation set for $H$. Then $\text{gh}^S_{H \le
  G}(T; Q)$ is free abelian, so the subgroup $\text{im}(w) \cong W^S_{H \le G}(T;
Q)$ must also be free abelian. We can show that under these conditions
the exact sequence of Corollary~\ref{cor:witt_exact_minimal} is a split short exact sequence, and use this
to describe a basis (generalising Proposition~1.14 in \cite{dotto_witt_2025}).
Note when $(T; Q)$ is free we already have a computation of $W^S_{H \le G}(T;
Q)$ in Proposition~\ref{prop:witt_free}, but the following lemma gives a
slightly different isomorphism.

\begin{lemma} \label{lem:witt_ab_group_free_coeff}
  Let $T$ be a ring whose additive
  group is free, and let $Q = T(Y)$ be a free $T$-module.
  Let $S$ be a finite
  truncation set for $H$. Then we have an
  isomorphism of abelian groups
  \[W^S_{H \le G}(T; Q) \cong \bigoplus_{V \in \can{S}} \left( Q^{\otimes_T G/V} \right)_{N_H(V)} \text{.}\]
  Observe that $\left(Q^{\otimes_T G/V}\right)_{N_H(V)} \cong T\left((Y^{\times
      G/V})_{N_H(V)}\right)$ is free as an abelian group.

  Moreover we can describe a basis somewhat explicitly. Let $\{x_i\}$ be
  elements of $Q^{\otimes_T G/K}$ that represent a basis for
  $\left( Q^{\otimes_T G/K} \right)_{N_H(K)}$. Then under the isomorphism the corresponding component
  of $W^S_{H \le G}(T; Q)$ has basis $\{V^H_K \tau_{G/K}(x_i)\}$. %
\end{lemma}
\begin{proof}
  We proceed by induction on the size of $S$. If $S = \emptyset$ then the claim
  is immediate.  Otherwise let $K$ be a minimal element of
  $S$.
  By Corollary~\ref{cor:witt_exact_minimal} we have an exact sequence
  \[\left(Q^{\otimes_T G/K}\right)_{N_H(K)} \xrightarrow{V^H_K \tau_{G/K}} W^S_{H \le G}(T; Q)
    \xrightarrow{R_{S \setminus K}} W^{S \setminus K}_{H \le G}(T; Q) \to 0 \text{.}\]
  We can compute $w_K(V^H_K \tau_{G/K}(a)) =
  \text{tr}^{N_H(K)}_K a$. The map $\text{tr}^{N_H(K)}_K :
  \left( Q^{\otimes_T G/K}\right)_{N_H(K)} \to \left( Q^{\otimes_T G/K}
  \right)^{N_H(K)}$ is injective (as $Q$ is free over a torsion-free ring), so
  the sequence is exact on the left, giving a short exact sequence
  \[ 0 \to \left( Q^{\otimes_T G/K}\right)_{N_H(K)}
    \xrightarrow{V^H_K \tau_{G/K}} W^S_{H \le G}(T; Q)
    \xrightarrow{R_{S \setminus K}} W^{S \setminus K}_{H \le G}(T; Q) \to 0 \text{.}\]
  We know $W^{S
    \setminus K}_{H \le G}(T; Q)$ is free abelian, so the sequence splits
  and by induction we have an isomorphism of abelian groups
  \[W^S_{H \le G}(T; Q) \cong W^{S \setminus K}_{H \le G}(T; Q) \oplus \left(
      Q^{\otimes_T G/K}\right)_{N_H(K)} \cong \prod_{V \in \can{S}} \left( Q^{\otimes_T
    G/V} \right)_{N_H(V)} \text{.}\]

  To see the statement about the basis, note that $R_{S \setminus K}$
  commutes with $V^H_V \tau_{G/V}$ so we can choose our splittings of short
  exact sequences such that there is a basis as described.
\end{proof}

\begin{remark}
  Note the formula for the basis is the same as that of
  Proposition~\ref{prop:witt_free} and Remark~\ref{rem:free_V_tau} except with $\tau^f$ replaced with $\tau$.
\end{remark}

  Earlier we saw that $W^S_{H \le G} : \text{Mod} \to \text{Ab}$ is lax symmetric
  monoidal. We now have a Mackey functor-valued functor $\underline{W}^S_G$, and
  in the following proposition we prove that it is still lax symmetric monoidal.
  In fact we will subsequently show that the untruncated Witt vector functor
  $\underline{W}_G$ is strong symmetric monoidal.

\begin{proposition}
  The functors $\underline{W}^S_{G} : \text{Mod} \to \text{Mack}_H(\text{Ab})$ and $\underline{gh}^S_{G} : \text{Mod} \to \text{Mack}_H(\text{Ab})$ are lax
  symmetric monoidal with respect to the box product of Mackey functors.
\end{proposition}
\begin{proof}
  We will discuss the case of $\underline{W}^S_G$, and note at the end how
  $\underline{\text{gh}}^S_G$ differs. Maps $\underline{M} \Osq \underline{M}' \to \underline{N}$ out of a box
  product of $H$-Mackey functors correspond to pairings $(\underline{M}, \underline{M}') \to \underline{N}$ in the sense of
  \cite{lewis_theory_nodate} Proposition~1.4, %
  i.e.\ collections of maps $M(K) \otimes_{\mathbb{Z}} M'(K)
  \to N(K)$ for all $K \le H$ satisfying certain identities. The external product introduced in Proposition~\ref{prop:witt_monoidal} gives maps
  \[\star : W^{S \mid_K}_{K \le G}(R; M) \otimes_\mathbb{Z} W^{S \mid_K}_{K \le G}(R'; M') \to
    W^{S \mid_K}_{K \le G}(R
    \otimes_\mathbb{Z} R'; M \otimes_\mathbb{Z} M')\]
  for all $K \le H$. The external product satisfies a Frobenius reciprocity-type relation, and the Frobenius and conjugation operators are symmetric monoidal (identity (i) of Proposition~\ref{prop:witt_monoidal} and Proposition~\ref{prop:fc_sym_mon}).
  Together these precisely show that the collection of external product maps gives a pairing of Mackey
  functors $(\underline{W}^S_G(R; M),
  \underline{W}^S_G(R'; M')) \to \underline{W}^S_G(R \otimes_\mathbb{Z} R'; M \otimes_\mathbb{Z} M')$
  and so corresponds to a map of Mackey functors
  \[\mu : \underline{W}^S_G(R; M) \Osq \underline{W}^S_G(R'; M') \to
    \underline{W}^S_G(R \otimes_\mathbb{Z} R'; M \otimes_\mathbb{Z} M') \text{.}\]

  The monoidal unit of $\text{Mack}_H(\text{Ab})$ is the Burnside Mackey functor
  $\underline{\Omega} : H/K \mapsto \Omega(K)$. By Corollary~\ref{cor:witt_z_burnside} this is
  precisely the underlying $H$-Mackey functor of the $G$-Mackey functor $\underline{W}_G(\mathbb{Z}; \mathbb{Z})$. Using the above product
  and unit $R_S : \underline{W}_G(\mathbb{Z}; \mathbb{Z}) \to \underline{W}^S_G(\mathbb{Z};
  \mathbb{Z})$, we check that $\underline{W}^S_G$
  is a lax symmetric monoidal functor. Symmetry and associativity follow from
  symmetry and associativity of the external product $\star$. %
  Considering the Mackey functor definition of truncation $R_S$ we see that it is monoidal,
  so the map $R_S : W_{K \le G}(\mathbb{Z};
  \mathbb{Z}) \to W^{S \mid_K}_{K \le G}(\mathbb{Z}; \mathbb{Z})$ (for $K \le
  H$) is a ring homomorphism and in particular preserves the unit.
  Hence unitality also holds.

  The case of $\underline{\text{gh}}^S_G$ is essentially the same, using the
  external product $\tildestar$ on ghost groups. We must now take unit $w R_S : \underline{\Omega} \cong
  \underline{W}_G(\mathbb{Z}; \mathbb{Z}) \to \underline{\text{gh}}^S_{G}(\mathbb{Z}; \mathbb{Z})$.
\end{proof}

\begin{remark}
  Analogously to Remark~\ref{rem:multiplication} we find that for $R$ a
  commutative ring, $\underline{W}^S_G(R; R)$ is a commutative Green
  functor and $\underline{W}^S_G(R; M)$ is a $\underline{W}^S_G(R; R)$-module.
\end{remark}

\begin{proposition} \label{prop:untruncated_witt_monoidal}
  The untruncated Witt vector functor
  \[\underline{W}_G : \text{Mod} \to \text{Mack}_G(\text{Ab})\]
  is strong symmetric monoidal.
\end{proposition}
\begin{proof}
  Let $\otimes$ denote $\otimes_{\mathbb{Z}}$. We already saw that the unit map $\underline{\Omega} \cong
  \underline{W}_G(\mathbb{Z}; \mathbb{Z})$ is an isomorphism. It remains to
  check that the map
  \[\mu : \underline{W}_G(R; M) \Osq \underline{W}_G(R'; M') \to
    \underline{W}_G(R \otimes R'; M \otimes M')\]
  is an isomorphism.

  First we reduce to the case of free coefficients. Note that reflexive
  coequalisers of Mackey functors are evaluated pointwise (since
  a reflexive coequaliser of additive functors $\mathcal{A}_G \to \text{Ab}$
  remains additive). Hence $\underline{W}_G$ preserves reflexive
  coequalisers. The box product of Mackey functors and the monoidal product in
  $\text{Mod}$ both preserve reflexive coequalisers, so taking free resolutions
  for $(R; M)$ and $(R'; M')$ shows that it suffices to show that $\mu$ is an
  isomorphism for free coefficients $(T; Q)$ and $(T'; Q')$. This lets us make
  use of the computation of the Witt vectors with free coefficients from Proposition~\ref{prop:witt_free}.

  Fix a subgroup $H \le G$. We will exhibit a section $\sigma_H$ for the
  $H$-component $\mu_H$ of $\mu$, then show that $\sigma_H$ is surjective and so conclude that $\mu_H$ is an isomorphism.

  We previously used a universal property of the box product: maps $\underline{M} \Osq
  \underline{M}' \to \underline{N}$ out of
  a box product correspond to pairing maps $(\underline{M}, \underline{M}') \to \underline{N}$ of Mackey functors. This
  leads to a description of the values of the box product as an abelian group modulo relations (see \cite{luca_algebra_1996} Section~3.1). We have
  \[(\underline{W}_G(T; Q) \Osq \underline{W}_G(T'; Q'))(H) \cong \left( \bigoplus_{K
        \le H} W_{K \le G}(T; Q) \otimes W_{K \le G}(T'; Q') \right) \!\raisebox{-.5em}{$\Big/$} \raisebox{-.65em}{$\sim$}\]
  where the equivalence relation is generated by
  \begin{alignat*}{3}
    V^K_L(n) \otimes n' &\sim n \otimes F^K_L(n') &&\text{for } n \in W_{L \le
                                                     G}(T; Q),\  n' \in W_{K \le G}(T'; Q'),\\
    n \otimes V^K_L(n') &\sim F^K_L(n') \otimes n' &&\text{for } n \in W_{K \le
                                                      G}(T; Q),\  n' \in W_{L \le G}(T'; Q'),\\
    n \otimes n' &\sim c_h(n) \otimes c_h(n') \quad &&\text{for } n \in
                                                            W_{K \le G}(T; Q),\
                                                       n' \in W_{K \le G}(T';
                                                       Q'),\  h \in H.
  \end{alignat*}
  For $n \in W_{K \le G}(T; Q)$ and $n' \in W_{K \le G}(T'; Q')$ we denote the equivalence class of $n
  \otimes n'$ in $(\underline{W}_G(T; Q) \Osq \underline{W}_G(T'; Q'))(H)$ by $n
  \Osq_H n'$.

  We define the map $\sigma_H$ by
  \[\begin{tikzcd} W_{H \le G}(T \otimes T'; Q \otimes Q') \ar[r,
      "\sigma_H"] \ar[d, "\cong"] & (\underline{W}_G(T; Q) \Osq \underline{W}_G(T'; Q'))(H)\\
    \left( \displaystyle\bigoplus_{K \le H} (Q \otimes Q')^{\otimes_{T \otimes T'} G/K}
    \right)_H \ar[r, "s^{-1}", "\cong"'] & \left( \displaystyle\bigoplus_{K \le H} Q^{\otimes_T G/K} \otimes
      Q'^{\otimes_{T'} G/K} \right)_H \ar[u, "\tau^f \Osq_H \tau^f"']\end{tikzcd} \]
  where the lower map is induced by the inverse shuffle maps $s^{-1} : (Q
  \otimes Q')^{\otimes_{T \otimes T'} G/K} \cong Q^{\otimes_T G/K} \otimes
  Q'^{\otimes_{T'} G/K}$ and the right vertical map takes $q \otimes q' \in
  Q^{\otimes_T G/K} \otimes Q'^{\otimes_{T'} G/K}$ to $\tau^f_{G/K}(q) \Osq_H
  \tau^f_{G/K}(q')$. The last map is well-defined as a map out of the group of
  $H$-orbits since the action of $h \in H$ takes $q \otimes q'$ to $(h \cdot q)
  \otimes (h \cdot q') \in Q^{\otimes_T G/\prescript{h}{}{K}} \otimes
  Q'^{\otimes_{T'} G/\prescript{h}{}{K}}$, and $\tau^f_{G/\prescript{h}{}{K}}(h
  \cdot q) \Osq_H \tau^f_{G/\prescript{h}{}{K}}(h \cdot q') = c_h(\tau^f_{G/K}(q))
  \Osq_H c_h(\tau^f_{G/K}(q')) = \tau^f_{G/K}(q) \Osq_H \tau^f_{G/K}(q')$.

  Given $n \in W_{K \le G}(T; Q)$ and $n' \in W_{K \le G}(T'; Q')$, the map
  $\mu_H$ is defined by
  \[n \Osq_H n' \mapsto V_K^H(n \star n') \text{.}\]
  The group $W_{H \le G}(T \otimes T'; Q \otimes Q') \cong \left(
    \bigoplus_{K \le H} (Q \otimes Q')^{\otimes_{T \otimes T'} G/K} \right)_H$
  is generated by elements of the form $V^H_K \tau^f_{G/K} s(q \otimes q')$ for
  $q \in Q^{\otimes_T G/K}$, $q' \in Q'^{\otimes_{T'} G/K}$ (i.e. the image of
  $s(q \otimes q') \in (Q \otimes Q')^{\otimes_{T \otimes T'} G/K}$ under the inclusion map). We have
  \begin{align*}\mu_H(\sigma_H(V^H_K \tau^f_{G/K} s(q \otimes q'))) &= \mu_H(\tau^f_{G/K}(q) \Osq_H \tau^f_{G/K}(q'))\\
    &= V^H_K(\tau^f_{G/K}(q) \star \tau^f_{G/K}(q'))\\
                                                                    &= V^H_K \tau^f_{G/K} s(q \otimes q')
  \end{align*}
  so $\sigma_H$ is a section for $\mu_H$ as desired.

  Next we show that $\sigma_H$ is surjective. We know $(\underline{W}_G(T; Q)
  \Osq \underline{W}_G(T'; Q'))(H)$ is generated by elements of the form $n \Osq_H
  n'$ for $K \le H$, $n \in W_{K \le G}(T; Q)$ and $n' \in W_{K \le G}(T'; Q')$. Going
  further, it is generated by elements of the form $V^K_L \tau^f_{G/L}(q) \Osq_H
  V^K_{L'} \tau^f_{G/L'}(q')$ for $L, L' \le K$, $q \in Q^{\otimes_T G/L}$
  and $q' \in Q'^{\otimes_{T'} G/L'}$. But by a similar calculation to the key
  part of the
  proof of Proposition~\ref{prop:witt_monoidal} (in particular here using the
  $\tau^f$ version of Proposition~\ref{prop:witt_teichmuller_properties}~(v), see Remark~\ref{rem:alternative_teichmuller}) we have
  \begin{align*}
    V^K_L \tau^f_{G/L}(q) \Osq_H V^K_{L'} \tau^f_{G/L'}(q') &= \sum_{LkL' \in L\backslash K/L'} F^L_{L \cap \prescript{k}{}{L'}} \tau^f_{G/L}(q) \Osq_H c_k F^{L'}_{L^k \cap L'} \tau^f_{G/L'}(q')\\
    &= \sum_{LkL' \in L\backslash K/L'} \Big( \tau^f_{G/(L \cap
      \prescript{k}{}{L'})}\phi^L_{L \cap \prescript{k}{}{L'}}(q)\\
    &\hspace{6em} \Osq_H \tau^f_{G/
   (L \cap \prescript{k}{}{L'})}\phi^{\prescript{k}{}{L'}}_{L \cap
   \prescript{k}{}{L'}}(k \cdot q') \Big)\end{align*}
so in fact $(\underline{W}_G(T; Q) \Osq \underline{W}_G(T'; Q'))(H)$ is
generated just by box products of images of $\tau^f$; that is, by elements of the form
\[\tau^f_{G/K}(q) \Osq_H \tau^f_{G/K}(q') =
  \sigma_H(V^H_K \tau^f_{G/K} s(q \otimes q'))\]
for $K \le H$, $q \in Q^{\otimes_T G/K}$ and $q' \in Q'^{\otimes_{T'} G/K}$. So $\sigma_H$ is surjective.

We conclude that $\mu_H$ is an isomorphism with inverse $\sigma_H$ (hence the
maps $\sigma_H$ assemble into an inverse $\sigma$ for the map of Mackey
functors $\mu$). So $\underline{W}_G$ is strong monoidal.
\end{proof}

\begin{remark} \label{rem:untruncated_witt_monoidal}
  This proof was one of our main motivations in introducing the alternative
  Teichm\"uller map $\tau^f$ of Definition~\ref{def:alternative_teichmuller}. One
  could try to write a similar proof using
  Lemma~\ref{lem:witt_ab_group_free_coeff} and $\tau$ instead of
  Proposition~\ref{prop:witt_free} and $\tau^f$, but in order to apply
  Proposition~\ref{prop:witt_teichmuller_properties}~(v) one needs to make
  compatible choices of coset representatives for $G/H$ and $G/K$, and we do not
  believe it is always possible to make choices of coset representatives when defining $\sigma_H$ such that they are compatible for all pairs of subgroups required in the proof.
\end{remark}
\begin{remark}
  In the truncated case the map
  \[\mu : \underline{W}^S_G(R; M) \Osq \underline{W}^S_G(R'; M') \to
    \underline{W}^S_G(R \otimes R'; M \otimes M')\]
  is still an isomorphism. However the unit map
  \[R_S : \underline{\Omega} \cong \underline{W}_G(\mathbb{Z}; \mathbb{Z}) \to
    \underline{W}^S_G(\mathbb{Z}; \mathbb{Z})\]
  is not an isomorphism, so $\underline{W}^S_G$ is not strong monoidal as a
  functor to $\text{Mack}_H(\text{Ab})$.
  The fact that $\mu$ is an isomorphism can be proved by essentially the same method as
  above. Alternatively, once we have
  Lemma~\ref{lem:witt_mackey_truncation} proving that
  $\underline{W}^S_G(R; M)$ is the $S$-truncation of the Mackey functor
  $\underline{W}_G(R; M)$ then this follows from the untruncated case.
\end{remark}

\begin{corollary} \label{cor:witt_r_m_z_m}
  We have an isomorphism of $G$-Mackey functors
  \[\underline{W}_G(R; M) \cong \underline{W}_G(R; R)
    \Osq_{\underline{W}_G(\mathbb{Z}; R)} \underline{W}_G(\mathbb{Z}; M) \text{.}\]
  Note $(\mathbb{Z}; R)$ is a commutative monoid in $\text{Mod}$ and $(R; R)$ and
  $(\mathbb{Z}; M)$ are $(\mathbb{Z}; R)$-modules; so
  $\underline{W}_G(\mathbb{Z}; R)$ is a commutative Green functor and $\underline{W}_G(R; R)$ and $\underline{W}_G(\mathbb{Z};
  M)$ are $\underline{W}_G(\mathbb{Z}; R)$-modules.
\end{corollary}
\begin{proof}
  There is a reflexive coequaliser
  \[(R \otimes \mathbb{Z} \otimes \mathbb{Z}; R \otimes R
    \otimes M) \, \substack{\xrightarrow{}\\[-0.2em]
      \xleftarrow{\makebox[0ex]{}}\\[-0.2em] \xrightarrow[]{}} \, (R \otimes
    \mathbb{Z}; R \otimes R) \to (R; M)\]
  witnessing the relative tensor product $(R; R) \otimes_{(\mathbb{Z}; R)}
  (\mathbb{Z}; M) \cong (R; M)$. Applying the functor $\underline{W}_G$ gives
  a reflexive coequaliser
  \[\underline{W}_G(R \otimes \mathbb{Z} \otimes \mathbb{Z}; R \otimes R
    \otimes \mathbb{M}) \, \substack{\xrightarrow{}\\[-0.2em]
      \xleftarrow{\makebox[0ex]{}}\\[-0.2em] \xrightarrow[]{}} \, \underline{W}_G(R \otimes
    \mathbb{Z}; R \otimes R) \to \underline{W}_G(R; M) \text{.}\]
  Since $\underline{W}_G$ sends tensor products in $\text{Mod}$ to box
  products in $\text{Mack}_G(\text{Ab})$, we get a reflexive
  coequaliser witnessing the relative box product
  \[\underline{W}_G(R; R)
  \Osq_{\underline{W}_G(\mathbb{Z}; R)} \underline{W}_G(\mathbb{Z}; M) \cong \underline{W}_G(R;
  M)\text{.}\qedhere\]
\end{proof}

\section{The norm}  \label{sec:norm}

From now on, $G$ will always be a finite group.
In this section we will analyse the group $\pi^H_0(N_{\{e\}}^G(X) \wedge
\tilde{E} \mathcal{F}(S))$ for $X$ a connective spectrum, $H \le G$ and $S$ a
truncation set for $H$, in particular constructing a
Teichm\"uller map that behaves analogously to that for the Witt vectors.
We will use this to prove the main result,
\[\pi^H_0(N_{\{e\}}^G(X) \wedge {\tilde{E} \mathcal{F}(S)}) \cong W_{H \le G}^{S}(\mathbb{Z}; \pi_0(X)) \text{.}\]

\subsection{Technical details of the construction} \label{sec:norm_technical}

We begin by recalling the definition of the norm and commenting on some
technical matters. The Hill-Hopkins-Ravenel norm
$N^G_H({-})$, as described in \cite{hill_nonexistence_2016}, is a functor
from orthogonal $H$-spectra to orthogonal $G$-spectra. The authors do not distinguish
notationally between the underived version of the functor (which can be more easily described at the
point-set level) and the left derived version (which is homotopically meaningful). We will largely work with the derived version, which we
denote $N^G_H$, but at some points it will be
helpful to distinguish the underived version, which we will denote $\tilde{N}^G_H$.

We are only interested in the norm when
$H = \{e\}$ is the trivial subgroup, and this case is particularly
easy to describe.
Given an orthogonal spectrum $X$, the underlying spectrum of
$\tilde{N}^G_{\{e\}}X$
is the $\abs{G}$-fold smash product $X^{\wedge \abs{G}}$. The group $G$ acts on this
by permuting the factors: the action of $g \in G$ sends the factor
indexed by $g'$ to that indexed by $g' g^{-1}$. This gives an orthogonal
$G$-spectrum. We can then define the derived functor
\[N_{\{e\}}^G X \coloneqq \tilde{N}^G_{\{e\}} QX\]
where $Q$ is a cofibrant replacement functor on the category $\text{Sp}$ of
orthogonal spectra with the stable model structure.

For a space $X$, the $H$-fixed points of the $G$-space $X^{\wedge \abs{G}}$ are homeomorphic
to $X^{\wedge \abs{G/H}}$. The geometric fixed points of the norm behave
similarly. By Proposition~2.57 of \cite{hill_nonexistence_2016} and the remarks following that,
the spectrum $(\tilde{N}_{\{e\}}^G X)^{\Phi G}$ is weakly equivalent to $X$ when
$X$ is cofibrant. More generally we have the following.

\begin{lemma} \label{lem:H_geo_fixed_of_G_norm}
  Let $H$ be a subgroup of $G$ and $X$ a cofibrant spectrum. Then we have a canonical weak equivalence
  \[(\tilde{N}^G_{\{e\}}X)^{\Phi H} \simeq X^{\wedge G/H} \text{.}\]
\end{lemma}
\begin{proof}
  Make a choice of coset
  representatives for $G/H$. Using this choice gives us an isomorphism between $G$
  considered as an $H$-set and a disjoint union of $G/H$ copies of the $H$-set $H$;
  accordingly from the definition we see that the underlying $H$-spectrum of $\tilde{N}^G_{\{e\}}X$
  is isomorphic to $(\tilde{N}^H_{\{e\}} X)^{\wedge G/H}$. Then we have
  \[(\tilde{N}_{\{e\}}^GX)^{\Phi H} \simeq ((\tilde{N}_{\{e\}}^H X)^{\wedge G/H})^{\Phi
      H} \simeq ((\tilde{N}_{\{e\}}^HX)^{\Phi H})^{\wedge G/H} \simeq X^{\wedge
      G/H} \text{.}\]

  It remains to show that this weak equivalence is
  canonical; that is, that it does not depend on the choice of coset
  representatives that we made at the start of the proof. %
To do so we give an alternative construction of the weak equivalence following Proposition~B.209 of
    \cite{hill_nonexistence_2016}. Let us sketch the proof here, though to be
    fully rigorous requires a more detailed analysis of point-set models. First
    observe that if $X$ is a space, we have a canonical weak equivalence
    \[(N_{\{e\}}^G(\Sigma^\infty X))^{\phi H} \simeq \Sigma^\infty ((X^{\wedge
      G})^H) \simeq \Sigma^\infty (X^{\wedge G/H}) \simeq (\Sigma^\infty
    X)^{\wedge G/H} \text{,}\]
  For an inner
  product space $V$, we have a canonical weak equivalence
  \[(N_{\{e\}}^G S^{-V})^{\Phi H} =
    (S^{-\text{ind}_{\{e\}}^G V})^{\Phi H} \simeq (S^{-V})^{\wedge
      G/H}\text{.}\]%
    Since the norm and the geometric fixed point construction have good monoidal
    properties, we can use the tautological presentation of a spectrum to extend
    to a general canonical weak equivalence $({N}_{\{e\}}^GX)^{\Phi H} \simeq
    X^{\wedge G/H}$. This must be the same weak equivalence that we constructed
    earlier, since they agree for suspension spectra and both have good monoidal properties. But this new description didn't depend on any arbitrary choice of coset representatives.
\end{proof}

For $X$ any connective spectrum we deduce
  \[\pi_0(N_{\{e\}}^GX)^{\Phi H} \cong \pi_0((QX)^{\wedge G/H}) \cong (\pi_0 X)^{\otimes_{\mathbb{Z}} G/H} \text{.}\]

Since it is built from the smash product, the underived norm $\tilde{N}^G_{\{e\}}$ is a
strong symmetric monoidal functor. We get a weak equivalence
\[N^G_{\{e\}}(X \wedge X') \simeq N^G_{\{e\}}(X) \wedge N^G_{\{e\}}(X') \text{,}\]
and so (since $\tilde{E}\mathcal{F}(S)$ is idempotent) $X \mapsto \underline{\pi}_0(N^G_{\{e\}}(X) \wedge
\tilde{E}\mathcal{F}(S))$ is a strong symmetric monoidal functor $\text{Sp} \to \text{Mack}_G(\text{Ab})$.

\subsection{Reduction to Eilenberg-MacLane spectra}

We first show that $\underline{\pi}_0(N_{\{e\}}^G(X) \wedge {\tilde{E}
  \mathcal{F}(S)})$ only depends on $X$ through $\pi_0 X$. This means we can
reduce to the case where $X = HM$ is an Eilenberg-MacLane spectrum, for $M$ some
abelian group.
Similar results appear in \cite{ullman_symmetric_2013} (for example
Lemma~3.1), but we include a different proof here.

We will need to be quite careful about cofibrant replacements in the following lemmas, since we will later use a cofibrant replacement in the category of simplicial orthogonal spectra and this is not the same as levelwise applying the cofibrant replacement $Q$.

\begin{lemma}
  Let $X$ and $X'$ be connective cofibrant spectra and $f : X \to X'$ a $1$-connected
  map (that is, $f$ induces an isomorphism on $\pi_0$ and an epimorphism on
  $\pi_1$). Then the map
  \[\tilde{N}_{\{e\}}^G(f) : \tilde{N}_{\{e\}}^G(X) \to \tilde{N}_{\{e\}}^G(X')\]
  is $1$-connected (that
  is, it induces an
  isomorphism on all zeroth equivariant stable homotopy groups and an
  epimorphism on all first equivariant stable homotopy groups).
\end{lemma}
\begin{proof}
  Let $Y = \text{fib}(\tilde{N}_{\{e\}}^G(X) \to \tilde{N}_{\{e\}}^G(X'))$. Since taking genuine fixed points
  commutes with fibres, it suffices to show that $Y^H$ is $0$-connected for all $H \le
  G$, i.e.\ $\pi_i \, Y^H = 0$ for $i \le 0$.

  We have
  \[Y^{\Phi H} \simeq \text{fib}(\tilde{N}_{\{e\}}^G(X)^{\Phi H} \to \tilde{N}_{\{e\}}^G(X')^{\Phi H}) \simeq
    \text{fib}(X^{\wedge G/H} \to X'^{\wedge G/H}) \text{.}\] %
  Since a smash product of $1$-connected maps between connective spectra remains
  $1$-connected, we deduce that $Y^{\Phi H}$ is $0$-connected. By an isotropy
  separation argument we conclude that
  $Y$ is connective, i.e.\ all negative homotopy groups are zero. Moreover the map of Mackey functors $0 \to \underline{\pi}_0(Y)$ induces
  surjections on geometric fixed points, so by Lemma~\ref{lem:mackey_isotropy_surjective} the map of Mackey functors is
  itself surjective, hence we also have $\underline{\pi}_0(Y) = 0$.
\end{proof}

\begin{corollary} \label{cor:replace_with_em}
  For any connective cofibrant spectrum $X$ we have
  \[\pi^H_0(\tilde{N}^G_{\{e\}}(X) \wedge {\tilde{E}\mathcal{F}(S)}) \cong \pi^H_0(N_{\{e\}}^G(H\pi_0
    (X)) \wedge {\tilde{E}\mathcal{F}(S)}) \text{,}\]
  where note we use the underived norm functor on the left but the derived norm on the right.
\end{corollary}
\begin{proof}
  The canonical map $X \to H \pi_0(X)$ is $1$-connected. Using the lemma we
  get an isomorphism of Mackey functors
  \[\underline{\pi}_0(\tilde{N}_{\{e\}}^G(X)) \cong \underline{\pi}_0(\tilde{N}_{\{e\}}^G(QX)) \cong \underline{\pi}_0(\tilde{N}_{\{e\}}^G(QH \pi_0(X))) =
    \underline{\pi}_0(N_{\{e\}}^G(H \pi_0(X))) \text{,}\]
  and hence an isomorphism of all their truncations. By
  Remark~\ref{rem:truncations_agree_spectra} the statement follows.
\end{proof}

So it suffices to consider the case of Eilenberg-MacLane spectra. That is, we want to
analyse the functor $M \mapsto \pi^H_0(N_{\{e\}}^G(HM)
\wedge \tilde{E}\mathcal{F}(S))$. We can now use the tools that we developed
in Section~\ref{sec:witt_final_definition}: we will show that this functor
preserves reflexive coequalisers of abelian groups, and then it is enough to only analyse what
happens for free abelian groups.

The proof of the following lemma is analogous to that of \cite{dotto_witt_2025}
Proposition~2.11~(vii), %
and the same result appears as Lemma~3.4 of \cite{ullman_symmetric_2013} with a
different proof.

\begin{lemma} \label{lem:norm_refl_coeq}
  The functor $\text{Ab} \to \text{Ab}$ given by
  \[M \mapsto \pi^H_0(N_{\{e\}}^G(HM) \wedge {\tilde{E} \mathcal{F}(S)})\]
  preserves reflexive coequalisers.
\end{lemma}
\begin{proof}
  Suppose we have a reflexive coequaliser in $\text{Ab}$,
  \[M_1 \, \substack{\xrightarrow{}\\[-0.2em]
      \xleftarrow{}\\[-0.2em] \xrightarrow[]{}} \, M_0 \xrightarrowdbl{} M
    \text{.}\]
  The diagram $M_1 \, \substack{\xrightarrow{}\\[-0.2em] \xleftarrow{}\\[-0.2em]
    \xrightarrow[]{}} \, M_0$ describes a presheaf of abelian groups on the truncated simplex
  category $\Delta_{\le 1}$. Right Kan extension along the subcategory
  inclusion $\Delta_{\le 1}^{\text{op}} \hookrightarrow \Delta^{\text{op}}$ gives
  an extension to a simplicial abelian group $M_{\bullet}$, and so taking
  Eilenberg-MacLane spectra levelwise gives a simplicial %
  spectrum $HM_{\bullet}$. The stable model structure on orthogonal spectra is
  cofibrantly generated, so there is a projective model structure on simplicial
  orthogonal spectra; let $\widehat{HM}_\bullet$ be the cofibrant replacement of
  $HM_\bullet$ in this model structure. Geometric realisation
  is left Quillen so $\abs*{\widehat{HM}_\bullet}$ is cofibrant. Also note that
  $\widehat{HM}_\bullet$ is levelwise cofibrant, and is levelwise weakly
  equivalent to $HM_\bullet$.

  The zeroth homotopy group of a geometric realisation %
  is the reflexive
  coequaliser of the zeroth homotopy groups of the last two terms. %
  Applying this
  to the simplicial spectrum $\tilde{N}_{\{e\}}^G(\widehat{HM}_\bullet)
  \wedge {\tilde{E} \mathcal{F}(S)}$ gives a reflexive coequaliser diagram
  \[\pi^H_0(\tilde{N}_{\{e\}}^G(\widehat{HM}_1) \wedge {\tilde{E} \mathcal{F}(S)}) \, \substack{\xrightarrow{}\\[-0.2em]
      \xleftarrow{}\\[-0.2em] \xrightarrow[]{}} \, \pi^H_0(\tilde{N}_{\{e\}}^G(\widehat{HM}_0) \wedge
    {\tilde{E} \mathcal{F}(S)}) \xrightarrowdbl{} \pi^H_0(\abs*{\tilde{N}_{\{e\}}^G(\widehat{HM}_\bullet)
      \wedge {\tilde{E} \mathcal{F}(S)}})
    \text{.}\]

  Geometric realisation commutes with smash products and smash powers up to
  isomorphism of orthogonal spectra, so it commutes with the non-derived Hill-Hopkins-Ravenel
  norm $\tilde{N}^G_{\{e\}}$ up to isomorphism of orthogonal $G$-spectra. This
  gives
  \[\pi^H_0(\abs*{\tilde{N}_{\{e\}}^G(\widehat{HM}_\bullet)
      \wedge {\tilde{E} \mathcal{F}(S)}}) \cong \pi^H_0 (\tilde{N}_{\{e\}}^G
    (\abs*{\widehat{HM}_{\bullet}}) \wedge {\tilde{E} \mathcal{F}(S)}) \text{.}\]
  Observe that $\pi_0 \abs*{\widehat{HM}_{\bullet}} \cong \text{coeq}(\pi_0(\widehat{HM}_1) \, \substack{\xrightarrow{}\\[-0.2em] \xleftarrow{}\\[-0.2em]
    \xrightarrow[]{}} \, \pi_0(\widehat{HM}_0)) \cong \text{coeq}(M_1 \, \substack{\xrightarrow{}\\[-0.2em] \xleftarrow{}\\[-0.2em]
    \xrightarrow[]{}} \, M_0) \cong M$. So using
  Corollary~\ref{cor:replace_with_em} we have
  \[\pi^H_0(\tilde{N}_{\{e\}}^G \abs*{\widehat{HM}_{\bullet}} \wedge {\tilde{E} \mathcal{F}(S)})
    \cong \pi^H_0(N_{\{e\}}^G(HM) \wedge {\tilde{E} \mathcal{F}(S)}) \text{.}\]
  By Corollary~\ref{cor:replace_with_em} we also have isomorphisms
  \[\pi^H_0 (\tilde{N}^G_{\{e\}}(\widehat{HM}_i) \wedge \tilde{E}\mathcal{F}(S))
  \cong \pi^H_0(N_{\{e\}}^G(HM_i) \wedge \tilde{E}\mathcal{F}(S)) \text{.}\]
  We conclude that we have a reflexive coequaliser
  \[\pi^H_0(N_{\{e\}}^G(HM_1) \wedge {\tilde{E} \mathcal{F}(S)}) \, \substack{\xrightarrow{}\\[-0.2em]
      \xleftarrow{}\\[-0.2em] \xrightarrow[]{}} \, \pi^H_0(N_{\{e\}}^G(HM_0) \wedge
    {\tilde{E} \mathcal{F}(S)}) \xrightarrowdbl{} \pi^H_0(N_{\{e\}}^G(HM) \wedge \tilde{E} \mathcal{F}(S))
  \]
  as desired. %
\end{proof}

\subsection{The Teichm\"uller map}\label{sec:norm_teichmuller}

We now construct the topological analogue to the Teichm\"uller map $\tau_{G/H} :
M^{\otimes_R G/H} \to W^S_{H \le G}(R; M)$.

\begin{proposition}\label{prop:norm_teichmuller_properties}
  Given a choice $\{g_i H\}$ of coset representatives for $G/H$, we have
  a (not necessarily additive) Teichm\"uller map
  \[\tau_{G/H} : M^{\otimes_{\mathbb{Z}} G/H} \to \pi^H_0(N_{\{e\}}^G(HM) \wedge
    {\tilde{E} \mathcal{F}(S)}) \text{,}\]
  natural in the choice of abelian group $M$.
  The map $\tau$ has the following properties, analogous to those we proved for
  the Witt vectors in Proposition~\ref{prop:witt_teichmuller_properties}:
  \begin{enumerate}[(i)]
  \item We have $\tau_{G/H}(0) = 0$.
  \item
    The map $R$ interacts well with $\tau_{G/H}$, in the sense that the diagram
    \[\begin{tikzcd}
        M^{\otimes_\mathbb{Z} G/H} \ar[r, "\tau_{G/H}"]
        \ar[rd, "\tau_{G/H}" swap] &
        \pi^H_0(N_{\{e\}}^G(HM) \wedge {\tilde{E} \mathcal{F}(S)})
        \ar[d, "R_{S'}"]\\
        & \pi^H_0(N_{\{e\}}^G(HM) \wedge {\tilde{E} \mathcal{F}(S')})
      \end{tikzcd}\]
    commutes.
  \item
    The map $\tau_{G/H}$ is equivariant, in the sense that
    \[\tau_{G/\prescript{g}{}{H}}(g \cdot n) = c_g \tau_{G/H}(n)\]
    for any $g \in G$ (where we use the map $\tau_{G/\prescript{g}{}{H}}$
    corresponding to the coset representatives $G/\prescript{g}{}{H} = \{g_i g^{-1} (\prescript{g}{}{H})\}$).
  \item The map
    \[M^{\otimes_\mathbb{Z} G/H} \xrightarrow{\tau_{G/H}} \pi^H_0(N_{\{e\}}^G(HM) \wedge
      \tilde{E}\mathcal{F}(\{H\})) = \pi_0(N_{\{e\}}^G(HM)^{\Phi H})\]
    is a monoidal additive isomorphism, and is independent of the choice of coset representatives.
  \item
    Suppose we have coset representatives $\{g_i\}$ for $G/H$ and $\{h_j\}$ for
    $H/K$. Observe that $\{g_i h_j\}$ is a set of coset representatives for
    $G/K$. Then the diagram
    \[\begin{tikzcd}
        M^{\otimes_\mathbb{Z} G/H} \ar[r, "\tau_{G/H}"] \ar[d,
        "({-})^{\otimes_\mathbb{Z} H/K}"] &
        \pi^H_0(N_{\{e\}}^G(HM) \wedge {\tilde{E} \mathcal{F}(S)})
        \ar[dd, "\text{res}_K^H"]\\
        M^{\otimes_\mathbb{Z} G/H \times H/K} \ar[d, "f_{G/H}"]&\\
        M^{\otimes_\mathbb{Z} G/K} \ar[r, "\tau_{G/K}"] & \pi^K_0(N_{\{e\}}^G(HM) \wedge
        {\tilde{E} \mathcal{F}(S\!\! \mid_K)})
      \end{tikzcd}\]
    commutes (where $f_{G/H}$, $\tau_{G/H}$ and $\tau_{G/K}$ are defined using
    the above coset representatives).
  \end{enumerate}
\end{proposition}
\begin{proof}
  We start by constructing the map. It suffices to define
  \[\tau_{G/H} : M^{\otimes_\mathbb{Z} G/H} \to \pi^H_0(N_{\{e\}}^G(HM)) \text{,}\]
  and then we can more generally define $\tau_{G/H} : M^{\otimes_\mathbb{Z} G/H} \to \pi^H_0(N_{\{e\}}^G(HM)
  \wedge {\tilde{E} \mathcal{F}(S)})$ by postcomposing with the truncation map
  $R_{S}$.

  Note that
  \[\pi_0((QHM)^{\wedge G/H}) \cong M^{\otimes_\mathbb{Z} G/H} \text{.}\]
  So given an element $m \in M^{\otimes_\mathbb{Z} G/H}$ we get an element of
  \[\pi_0((QHM)^{\wedge G/H}) \coloneqq \text{colim}_n [S^n, ((QHM)^{\wedge G/H})_n] \text{,}\]
  which can be represented by some map $\alpha : S^n \to ((QHM)^{\wedge G/H})_n$.
  We want to produce an element of $\pi^H_0(N_{\{e\}}^G(QHM))$; it suffices to obtain an
  $H$-equivariant map $S^V \to (N_{\{e\}}^G(QHM))_V$ where $V$ is some finite-dimensional
  $H$-representation. We will exhibit such a map for $V = n \rho_H$, where
  $\rho_H$ is the regular representation constructed as $\mathbb{R}^H$ where $h
  \in H$ acts by sending the factor indexed by $h'$ to the factor indexed by $h'
  h^{-1}$. The desired map is
  \[\begin{tikzcd}S^{n \rho_H} \ar[d, "\alpha^{\wedge H}"] & & \\
     ((QHM)^{\wedge G/H})_n^{\wedge H} \ar[r] & ((QHM)^{\wedge G/H \times
       H})_{n\abs{H}} \ar[r, "\cong"] & ((QHM)^{\wedge
       G})_{n \abs{H}} \ar[d, "\cong"]\\
     & & (N_{\{e\}}^G(HM))_{n \rho_H}
    \end{tikzcd}\]
  where the second map is the canonical inclusion and the third map is induced by
  the isomorphism $G/H \times H \cong G$ corresponding to the choice of coset
  representatives (i.e.\ $(g_i H, s) \mapsto g_i s$). By abuse of notation we
  will sometimes refer to this whole composition as $\alpha^{\wedge H}$.

  It is easy to check this map is $H$-equivariant. Moreover alternative
  representatives $\alpha$ give the same element of $\pi^H_0(N_{\{e\}}^G(HM))$, so we get a
  well-defined map $\tau_{G/H} : M^{\otimes_\mathbb{Z} G/H} \to \pi^H_0(N_{\{e\}}^G(HM))$. %

  We now check this map satisfies the desired properties.

  \begin{enumerate}[(i)]
  \item This is immediate from the definition of $\tau_{G/H}$, taking $\alpha$ to be a constant map.
  \item
    This is immediate, using the fact
    that for $S'' \subseteq S' \subseteq S$ we have $R_{S''}R_{S'} = R_{S''}$.
    \item{
        It is straightforward to check that
        $\tau_{G/H}$ is $G$-equivariant in the appropriate sense---the key
        point is that we have a commutative diagram of isomorphisms of sets
        \[
          \begin{tikzcd}
            G/H \times H \ar[r] \ar[d] & G \ar[d]\\
            G/\prescript{g}{}{H} \times \prescript{g}{}{H} \ar[r] & G
          \end{tikzcd}
          \quad
          \begin{tikzcd}
            (g_i H, s) \ar[r, mapsto] \ar[d, mapsto] & g_i s \ar[d, mapsto] \\
            (g_i g^{-1} \prescript{g}{}{H}, gsg^{-1}) \ar[r, mapsto] & g_i s g^{-1} \text{.}
          \end{tikzcd}
        \]
      }
  \item
    For $X$ a cofibrant spectrum, Lemma~\ref{lem:H_geo_fixed_of_G_norm} gives us
    a canonical weak equivalence
    \[(N_{\{e\}}^GX \wedge {\tilde{E} \mathcal{F}(\{H\})})^H = (N_{\{e\}}^GX)^{\Phi H}
      \simeq ((N_{\{e\}}^HX)^{\wedge G/H})^{\Phi H} \simeq X^{\wedge G/H} \text{.}\]

    In particular, the case $X = QHM$ gives us a canonical monoidal additive isomorphism
    \[\pi^H_0(N_{\{e\}}^G(HM) \wedge {\tilde{E} \mathcal{F}(\{H\})}) \cong M^{\otimes_\mathbb{Z}
        G/H} \text{.}\]

    When $S = \{H\}$ we defined $\tau_{G/H}(m)$ by taking
    $\alpha : S^n \to ((QHM)^{G/H})_n$ corresponding to $m$, producing $\alpha^{\wedge H} : S^{n
      \rho_H} \to (N_{\{e\}}^G(HM))_{n \rho_H}$, then applying the truncation map
    $\pi^H_0(N_{\{e\}}^G(HM)) \to \pi_0(N_{\{e\}}^G(HM)^{\Phi H})$, which corresponds to
    applying the naive $H$-fixed point functor to $\alpha^{\wedge H}$ giving us
    a %
    map $S^n \to ((N_{\{e\}}^G(HM))_{n \rho_H})^H$. Using this description we can check
    that $\tau_{G/H}$ is the inverse of the isomorphism described above, so %
    $\tau_{G/H}$ is a monoidal additive isomorphism and independent of the choice of
    coset representatives.
  \item
      We need to check the diagram
      \[\begin{tikzcd}
        M^{\otimes_\mathbb{Z} G/H} \ar[r, "\tau_{G/H}"] \ar[d,
        "({-})^{\otimes_\mathbb{Z} H/K}"] &
        \pi^H_0(N_{\{e\}}^G(HM) \wedge {\tilde{E} \mathcal{F}(S)})
        \ar[dd, "\text{res}_K^H"]\\
        M^{\otimes_\mathbb{Z} G/H \times H/K} \ar[d, "f_{G/H}"]&\\
        M^{\otimes_\mathbb{Z} G/K} \ar[r, "\tau_{G/K}"] &
        \pi^K_0(N_{\{e\}}^G(HM) \wedge {\tilde{E} \mathcal{F}(S \!\!\mid_K)})
      \end{tikzcd}\]
      where we use coset representatives $G/H = \{g_i H\}$ and $G/K = \{g_i h_j
      K\}$ arising from $H/K = \{h_j K\}$.

      By (ii), it suffices to check this when $S$ is the family
      of all subgroups, i.e.\ we can ignore the ${\tilde{E} \mathcal{F}(S)}$ term.
      Let $m \in M^{\otimes_\mathbb{Z} G/H}$, and let $\alpha : S^n \to
      ((QHM)^{\wedge G/H})_n$ represent the corresponding element of
      $\pi_0((QHM)^{\wedge G/H})
      \cong M^{\otimes_\mathbb{Z} G/H}$. The isomorphism $H/K \times K \cong H$
      given by $(h_j K, s) \mapsto h_j s$ induces an isomorphism of
      $K$-representations $n \rho_H \cong n \abs{H/K} \rho_K$. Then $\text{res}^H_K \tau_{G/H}(m)$ is obtained by
      taking the map $\alpha^{\wedge H} : S^{n \rho_H} \to (N_{\{e\}}^G(HM))_{n \rho_H}$ from the definition
      of $\tau$, then applying this isomorphism to get a map $S^{n \abs{H/K}
        \rho_K} \to (N_{\{e\}}^G(HM))_{n \abs{H/K} \rho_K}$.
      The element $f_{G/H}(m^{\otimes_\mathbb{Z} H/K})
      \in M^{\otimes_\mathbb{Z} G/K}$
      corresponds to the element of $\pi_0((QHM)^{\wedge G/K})$ represented by
      \[\begin{tikzcd}[column sep=small] S^{n \abs{H/K}} \ar[d, "\alpha^{\wedge H/K}"] & &\\
          ((QHM)^{\wedge G/H})_n^{\wedge H/K} \ar[r] & ((QHM)^{\wedge G/H \times
            H/K})_{n \abs{H/K}} \ar[r, "\cong"] & ((QHM)^{\wedge G/K})_{n
            \abs{H/K}} \text{,}
          \end{tikzcd}\]
      (using the usual isomorphism $G/H \times H/K \cong G/K$ given by $(g_i H,
      sK) \mapsto g_i s K$) and unwinding the definition of $\tau_{G/K}$ in
      $\tau_{G/K}(f_{G/H}(m^{\otimes_\mathbb{Z} H/K}))$ gives an element of
      $\pi_0(N_{\{e\}}^G(HM))$ represented by a map $S^{n \abs{H/K} \rho_K} \to
      (N_{\{e\}}^G(HM))_{n \abs{H/K} \rho_K}$. Careful checking shows that our
      representatives for the elements $\text{res}^H_K \tau_{G/H}(m)$ and
      $\tau_{G/K}(f_{G/H}(m^{\otimes_\mathbb{Z} H/K}))$ are in fact equal. The
      key is the commutative diagram
      \[
        \begin{tikzcd}
          G/H \times H/K \times K \ar[r] \ar[d] & G/K \times K \ar[d]\\
          G/H \times H \ar[r] & G
        \end{tikzcd}
        \quad
        \begin{tikzcd}
          (g_i H, h_j K, s) \ar[d, mapsto] \ar[r, mapsto] & (g_i h_j K, s)
          \ar[d, mapsto]\\
          (g_i H, h_j s) \ar[r, mapsto] & g_i h_j s \text{.}
        \end{tikzcd}
      \]
  \end{enumerate}
\end{proof}

\subsection{Computation of the zeroth homotopy}

We are now ready to prove the main result of the paper.
Our approach is analogous to that taken in
\cite{dotto_witt_2025}.
We will ultimately construct a
map $I : \prod_{V \in \can{S}} M^{\otimes_\mathbb{Z} G/V} \to
\pi_0^H(N^G_{\{e\}}(HM) \wedge \tilde{E}\mathcal{F}(S))$ and show that it
descends to an isomorphism out of the Witt vectors. To do so we first need to
define a map from the zeroth homotopy of the norm to
the ghost group, which will factor the usual ghost map.

\begin{lemma} \label{lem:norm_ghost}
  There is an analogue of the ghost map,
  \[\overline{w} : \pi^H_0(N_{\{e\}}^G(HM) \wedge {\tilde{E} \mathcal{F}(S)}) \to \text{gh}_{H \le G}^S(\mathbb{Z}; M) \text{.}\]
  This is a monoidal map of Mackey functors. Moreover we have
  $\overline{w} c_g = \tilde{c}_g \overline{w}$ for any $g \in G$, $\overline{w}R_{S'} = \tilde{R}_{S'} \overline{w}$ for $S' \subseteq S$
  and $\overline{w} \tau_{G/H}(m) = \tilde{\tau}_{G/H}$ given any choice of coset representatives for $G/H$.
\end{lemma}
\begin{proof}
  To make the notation less cumbersome, define
  \[\mathcal{N}_{H \le G}^S(M) \coloneqq \pi^H_0(N_{\{e\}}^G(HM) \wedge {\tilde{E}
    \mathcal{F}(S)}) \text{.}\]

  Given $U \in S$, define $\overline{w}_U$ to be the composition
  \[\overline{w}_U : \mathcal{N}_{H \le G}^S(M) \xrightarrow{\text{res}^H_U} \mathcal{N}_{U \le G}^{S \mid_U}(M) \xrightarrow{R_{\{U\}}} \mathcal{N}_{U \le G}^{\{U\}}(M) \xrightarrow{\tau_{G/U}^{-1}} M^{\otimes_\mathbb{Z} G/U}\]
  where recall that we just proved $\tau_{G/U} : M^{\otimes_\mathbb{Z} G/U} \to \mathcal{N}_{U \le
    G}^{\{U\}}(M)$ is an additive isomorphism (and does not depend on a choice
  of coset representatives). Note $\overline{w}_U$ is monoidal, since $\text{res}^H_U$ and $R_{\{U\}}$ are
  monoidal by standard equivariant stable homotopy theory and $\tau^{-1}_{G/U}$
  is monoidal by Proposition~\ref{prop:norm_teichmuller_properties}~(iv).

  Now define
  \[\overline{w} : \mathcal{N}^S_{H \le G}(M) \to \prod_{U \in S}
    M^{\otimes_\mathbb{Z} G/U}\]
  to be the product of the $\overline{w}_U$. We claim
  the image of this map lies
  in the $H$-fixed points of $\prod_{U \in S} M^{\otimes_\mathbb{Z} G/U}$.
  Indeed, for $n \in \mathcal{N}^S_{H \le G}(M)$ and $a \in H$ we have
  \begin{align*}
    a \cdot \overline{w}_U(n) = a \cdot \tau^{-1}_{G/U}(R_{\{U\}} \text{res}^H_U(n)) &=
    \tau^{-1}_{G/\prescript{a}{}{U}}(R_{\{\prescript{a}{}{U}\}}\text{res}^H_{\prescript{a}{}{U}}(c_a n))\\
    &= \tau^{-1}_{G/\prescript{a}{}{U}}(R_{\{\prescript{a}{}{U}\}}\text{res}^H_{\prescript{a}{}{U}}(n))\\
    &= \overline{w}_{\prescript{a}{}{U}}(n) \text{.}
  \end{align*}
  So we get an additive map $\overline{w} : \mathcal{N}^S_{H \le G}(M) \to
  \left( \prod_{U \in S} M^{\otimes_\mathbb{Z} G/U} \right)^H = \text{gh}^S_{H
    \le G}(\mathbb{Z}; M)$. Recalling the monoidal structure of $\text{gh}^S_{H
    \le G}(\mathbb{Z}; M)$ and using monoidality of $\overline{w}_U$ we see that $\overline{w}$ is monoidal.

  Next we need to check that this map commutes with the operators.
  For the restriction map, we need to show that we have $\overline{w} \text{res}^H_K = \tilde{F}^H_K \overline{w} :
  \mathcal{N}^S_{H \le G}(M) \to \text{gh}^{S \mid_K}_{K \le G}(\mathbb{Z}; M)$
  (recall that $\tilde{F}^H_K$ is the restriction map in the Mackey structure
  on the ghost group).
  Checking the
  $U$-component for $U \in S\!\!\mid_K$ gives
  \[\overline{w}_U\text{res}^H_K = \tau^{-1}_{G/U} R \text{res}^K_U \text{res}^H_K = \tau^{-1}_{G/U}R\text{res}^H_U
    = \overline{w}_U \text{,}\]
  which is the $U$-component of $\tilde{F}^H_K \overline{w}$ as desired.

  Similarly for the transfer map we need to check that $\overline{w}\text{tr}^H_K =
  \tilde{V}^H_K\overline{w} : \mathcal{N}^{S \mid_K}_{K \le G}(M) \to
  \text{gh}^S_{H \le G}(\mathbb{Z}; M)$, or in components
  \[\overline{w}_U \text{tr}^H_K = \sum_{hK \in (H/K)^U} h \cdot \overline{w}_{U^h}\]
  for $U \in S$.
But indeed
\begin{align*}
  \overline{w}_U \text{tr}^H_K &= \tau^{-1}_{G/U} R_{\{U\}} \text{res}^H_U \text{tr}^H_K\\
                       &= \sum_{UhK \in U \backslash H / K} \tau^{-1}_{G/U} R_{\{U\}} \text{tr}^U_{U \cap \prescript{h}{}{K}} c_h \text{res}^K_{U^h \cap K}\\
                       &= \sum_{hK \in (H/K)^U} \tau^{-1}_{G/U} R_{\{U\}} c_h \text{res}^K_{U^h}\\
                               &= \sum_{hK \in (H/K)^U} h \cdot \tau^{-1}_{G/U^h} R_{\{U^h\}} \text{res}^K_{U^h}\\
                       &= \sum_{hK \in (H/K)^U} h \cdot \overline{w}_{U^h} \text{.}
\end{align*}
The third equality holds because
if
$U \le \prescript{h}{}{K}$ then $\text{tr}^U_{U \cap \prescript{h}{}{K}}$ is the
identity and $UhK = hK$; or otherwise 
$R_{\{U\}}\text{tr}^U_{U \cap \prescript{h}{}{K}} =
\text{tr}^U_{U \cap \prescript{h}{}{K}}R_{\{U\}\mid_{U \cap \prescript{h}{}{K}}}$ 
and $\{U\}\!\!\mid_{U \cap \prescript{h}{}{K}} = \emptyset$ so
$R_{\{U\}}\text{tr}^U_{U \cap \prescript{h}{}{K}} = 0$.

Next consider the conjugation map. Given $g \in G$, we need
\[\overline{w}_{\prescript{g}{}{U}} c_g
= g \cdot \overline{w}_{U} : \mathcal{N}^S_{H \le G}(M) \to
\text{gh}^{\prescript{g}{}{S}}_{\prescript{g}{}{H} \le G}(\mathbb{Z}; M)\]
for $U \in S$. Indeed
we have $\overline{w}_{\prescript{g}{}{U}} c_g = \tau^{-1}_{G/\prescript{g}{}{U}} R_{\{\prescript{g}{}{U}\}}
\text{res}^{\prescript{g}{}{H}}_{\prescript{g}{}{U}} c_g = g \cdot \tau^{-1}_{G/U} R_{\{U\}}
\text{res}^H_U = g \cdot \overline{w}_U$.

Let $S' \subseteq S$ be truncation sets for $H$. For $U \in S'$, we have
\[\overline{w}_U R_{S'} = \tau^{-1}_{G/U} R_{\{U\}} \text{res}^H_U R_{S'} =
\tau^{-1}_{G/U} R_{\{U\}}R_{S'\mid_U} \text{res}^H_U = \tau^{-1}_{G/U} R_{\{U\}} \text{res}^H_U =
\overline{w}_U\]
as desired.

Finally, we consider the Teichm\"uller map. Choose coset representatives for
$G/H$. Given $m \in M^{\otimes_R G/H}$ and
$U \in S$, we want $\overline{w}_U \tau_{G/H}(m) = f_{G/H}(m^{\otimes_\mathbb{Z} H/U})$.
Recall from Proposition~\ref{prop:norm_teichmuller_properties} (ii) that $R_{\{U\}}\tau_{G/U} = \tau_{G/U}$.
Then we see
\begin{align*}\overline{w}_U \tau_{G/H}(m) = \tau^{-1}_{G/U} R_{\{U\}} \text{res}^H_U \tau_{G/H}(m) &=
                                                                                                      \tau^{-1}_{G/U} R_{\{U\}} \tau_{G/U} f_{G/H}(m^{\otimes_\mathbb{Z} H/U})\\
  &= f_{G/H}(m^{\otimes_\mathbb{Z}
  H/U})
  \end{align*}
as desired.

So $\overline{w}$ is a map of Mackey functors, and since it is pointwise
monoidal it is also monoidal as a map of Mackey functors.
\end{proof}

As with the Witt vectors, the ghost is injective at free objects.

\begin{lemma} \label{lem:top_ghost_injective}
  For $Q$ a free abelian group, the ghost map
  \[\overline{w} : \pi^H_0(N_{\{e\}}^G(HQ) \wedge {\tilde{E} \mathcal{F}(S)}) \to
    \text{gh}^S_{H \le G}(\mathbb{Z}; Q)\]
  is injective.
\end{lemma}
\begin{proof}
  We continue to write $\mathcal{N}_{H \le G}^S(Q)$ for $\pi^H_0(N_{\{e\}}^G(HQ)
  \wedge {\tilde{E} \mathcal{F}(S)})$.

 We proceed by induction on the size
  of the truncation set $S$, somewhat analogously to Lemma~\ref{lem:mackey_isotropy_surjective}. When $S$ is empty then $\mathcal{N}^\emptyset_{H
    \le G}(Q) = 0$ so $\overline{w}$ is injective.
  Now suppose $S$ is a non-empty truncation set and $\overline{w}$ is injective
  for every smaller truncation set. Pick $K$ a minimal element of
  $S$ (i.e.\ $S$ does not contain any subgroup strictly subconjugate to $K$). By
  Lemma~\ref{lem:isotropy} and Remark~\ref{rem:isotropy_orbits_form} we
  have an exact sequence
  \[(\mathcal{N}^{\{K\}}_{K \le G}(Q))_{N_H(K)} \xrightarrow{\text{tr}_K^H} \mathcal{N}^S_{H \le
    G}(Q) \xrightarrow{R_{S \setminus K}} \mathcal{N}^{S \setminus K}_{H \le G}(Q) \to 0 \text{.} \]
Note $\tau_{G/K} :  Q^{\otimes_\mathbb{Z} G/K} \cong \mathcal{N}^{\{K\}}_{K
  \le G}(Q)$ is an isomorphism. We get a commutative diagram
  \[\begin{tikzcd}
    & (Q^{\otimes_\mathbb{Z} G/K})_{N_H(K)} \ar[r, "\text{tr}_K^H
    \tau_{G/K}"] \ar[d, "\text{tr}_K^{N_H(K)}"] & \mathcal{N}^S_{H \le G}(Q)
    \ar[r, "R_{S \setminus K}"]
    \ar[d, "\overline{w}"]& \mathcal{N}^{S \setminus
    K}_{H \le G}(Q) \ar[d, "\overline{w}"]\\
  0 \ar[r] & (Q^{\otimes_\mathbb{Z} G/K})^{N_H(K)} \ar[r] & \left( \prod_{U \in S} Q^{\otimes_\mathbb{Z}
    G/U} \right)^{H} \ar[r, "\tilde{R}_{S \setminus K}"] & \left( \prod_{U \in S \setminus K} Q^{\otimes_\mathbb{Z}
    G/U} \right)^H
  \end{tikzcd}\]
with exact rows.

  The left square commutes since
  $\overline{w}\text{tr}_K^H \tau_{G/K} = \tilde{V}^H_K \tilde{\tau}_{G/K}$,
  where (as $K$ is minimal in $S$) we have that $\tilde{\tau}_{G/K}$ is the
  identity map and $\tilde{V}^H_K$ is the transfer $\text{tr}^{N_H(K)}_K :
  Q^{\otimes_{\mathbb{Z}} G/K} \to Q^{\otimes_{\mathbb{Z}} G/K}$ followed by the
  inclusion of the $K$ component of $\text{gh}^S_{H \le G}(\mathbb{Z}; Q)$.
  The right square commutes by the previous lemma.

  The left vertical map $\text{tr}_K^{N_H(K)}$ is %
  injective since $Q$ is free. We deduce that
  $\text{tr}_K^H \tau_{G/K}$ must also be injective, so the top row is in fact left
  exact. The rightmost vertical map is
  injective by the inductive hypothesis. So by the four lemma we deduce that
  $\overline{w} : \mathcal{N}^S_{H \le G}(Q) \to \left( \prod_{U \in S}
    Q^{\otimes_\mathbb{Z} G/U} \right)^H  = \text{gh}^S_{H \le G}(\mathbb{Z}; Q)$
  is injective, and the induction holds.
\end{proof}

Define the map
$I : \prod_{V \in \can{S}} M^{\otimes_{\mathbb{Z}} G/V} \to \pi^H_0(N_{\{e\}}^G(HM) \wedge {\tilde{E} \mathcal{F}(S)})$
by the formula
\[I(n) = \sum_{V \in \can{S}} \text{tr}^H_V \tau_{G/V} (n_V) \text{,}\]
where for each $V \in \can{S}$ we need some choice of coset representatives for
$G/V$ to define
$\tau_{G/V} : M^{\otimes_{\mathbb{Z}} G/V} \to \pi^V_0(N_{\{e\}}^G(HM) \wedge
{\tilde{E} \mathcal{F}(S \!\!\mid_V)})$.

\begin{lemma}
  The ghost map
  \[w : \prod_{V \in \can{S}} M^{\otimes_\mathbb{Z} G/V} \to \text{gh}^S_{H \le
      G}(\mathbb{Z}; M)\]
  factors as $w = \overline{w}I$, where we use the same choices of coset
  representatives to define $w$ and $I$.
\end{lemma}
\begin{proof}
  We have just seen how $\overline{w}$ interacts with all the operators, so we
  can compute
  \[\overline{w}(I(n)) = \overline{w}\left(\sum_{V \in \can{S}} \text{tr}^H_V \tau_{G/V}(n_V) \right) =
    \sum_{V \in \can{S}} \tilde{V}^H_V \tilde{\tau}_{G/V}(n_V) = w(n)\]
  where the final equality was observed in Lemma~\ref{lem:ghost_via_operators}.
\end{proof}

We now prove that $I$ descends to an isomorphism out of the Witt vectors.

\begin{theorem}
  \label{thm:witt_norm_iso}
  The map $I$ descends to the quotient and gives an isomorphism
  \[I : W_{H \le G}^S({\mathbb{Z}}; M) \cong \pi^H_0(N_{\{e\}}^G(HM) \wedge \tilde{E} \mathcal{F}(S)) \text{,}\]
  natural in $M$. This isomorphism respects the Mackey structure, $c_g$, $R$ and $\tau$
  operators and monoidal structure.
\end{theorem}
\begin{proof}
  We continue to write $\mathcal{N}_{H \le G}^S(M)$ for $\pi^H_0(N_{\{e\}}^G(HM)
  \wedge {\tilde{E} \mathcal{F}(S)})$, and additionally write
  $\underline{\mathcal{N}}^S_G(M)$ for the $H$-Mackey functor
  $\underline{\pi}_0(N_{\{e\}}^G(HM) \wedge \tilde{E}\mathcal{F}(S))$.
  For $Q$ a free abelian group, Lemmas~\ref{lem:norm_ghost} and \ref{lem:top_ghost_injective} show that $\overline{w} : \mathcal{N}^S_{H \le G}(Q) \to \text{gh}^S_{H \le
    G}(\mathbb{Z}; Q)$ is an additive injection, respecting the Mackey structure.
  The ghost map $w : \prod_{V \in \underline{S}} Q^{\otimes_\mathbb{Z} G/V} \to
  \text{gh}^S_{H \le G}(\mathbb{Z}; Q)$ descends to an additive injection
  $W^S_{H \le G}(\mathbb{Z}; Q) \to \text{gh}^S_{H \le G}(\mathbb{Z}; Q)$, also
  respecting the Mackey structure. So
  since $w = \overline{w} I$, we deduce that $I$ descends to the
  quotient giving an additive injection
  \[I : W^S_{H \le G}(\mathbb{Z}; Q) \to \mathcal{N}^S_{H \le G}(Q) \text{,}\]
  and these maps assemble to give an injective map of $H$-Mackey functors
  $\underline{W}^S_G(\mathbb{Z}; Q) \to \underline{\mathcal{N}}^S_G(Q)$.

  We will show that $I$ is also surjective. By
  Lemma~\ref{lem:mackey_isotropy_surjective} it suffices to check that for $K
  \in S$, the map of Mackey functors induces a surjection between the $K$-geometric fixed
  points.
  But indeed we know that $W^{\{K\}}_{K \le G}(\mathbb{Z}; Q)$ and
  $\mathcal{N}^{\{K\}}_{K \le G}(Q)$ are both isomorphic to
  $Q^{\otimes_{\mathbb{Z}} G/K}$, and it's easy to check that the induced map
  between them is an isomorphism. Hence $I$ is surjective.

  So when restricting to $Q$ free abelian, we've shown that $I$ descends
  along the quotient map $q$ to give an isomorphism of abelian groups.
  How about at an arbitrary
  abelian group $M$? We know that the Witt vectors
  preserve reflexive coequalisers in $\text{Mod}$ (essentially by definition). Reflexive coequalisers in
  $\text{Mod}$ are computed by taking the reflexive coequalisers of the
  underlying rings and modules, so the inclusion $\text{Ab} \to \text{Mod}$
  preserves reflexive coequalisers. So $W^S_{H \le G}(\mathbb{Z}; {-})$
  preserves reflexive coequalisers of abelian groups; and similarly for
  $M \mapsto \prod_{V \in \can{S}} M^{\otimes_\mathbb{Z} G/V}$. The functor
  $\mathcal{N}^S_{H \le G}({-})$ preserves reflexive coequalisers by Lemma~\ref{lem:norm_refl_coeq}.
  So by Lemma~\ref{lem:finite_module_kan} and Remark~\ref{rem:finite_module_kan}
  we deduce that in general $I : \prod_{V \in \can{S}} M^{\otimes_{\mathbb{Z}}
    G/V} \to \mathcal{N}^S_{H \le G}(M)$ factors as the quotient map $q$ followed by a natural isomorphism of
  abelian groups
  $W^S_{H \le G}(\mathbb{Z}; M) \cong \mathcal{N}^S_{H \le G}(M)$.
  Since the operators on the Witt vectors are the unique lifts of
  maps between ghost groups, and by Lemma~\ref{lem:norm_ghost} the
  operators on $\mathcal{N}^S_{H \le G}(\mathbb{Z}; Q)$ lift the same maps
  between ghost groups, we deduce that $I$ respects all the operators. Similarly
  since the monoidal structure on Witt vectors is the unique lift of the
  monoidal structure on the ghost groups, and by Lemma~\ref{lem:norm_ghost}
  $\overline{w}$ is monoidal, we deduce that $I$ is monoidal.
\end{proof}

\begin{remark}
  Recall that when \cite{dotto_witt_2025} prove that their Witt vectors compute the
  components of $\text{TR}$ with coefficients, they only use a limited set of
  properties of $\text{TR}$ axiomatised in their Proposition~2.11. Similarly we haven't used very many properties of the norm in order to prove Theorem~\ref{thm:witt_norm_iso}. We just needed that the functor $M
  \mapsto \pi_0^H(N^G_{\{e\}}(HM) \wedge \tilde{E}\mathcal{F}(S))$ preserves
  reflexive coequalisers, and comes with associated Mackey structure and a
  Teichm\"uller map satisfying
  Proposition~\ref{prop:norm_teichmuller_properties}---most crucially,
  when $S = \{H\}$ we get the isomorphism $\pi^H_0(N^G_{\{e\}}(HM) \wedge
  \tilde{E}\mathcal{F}(\{H\})) \cong \pi_0(N^G_{\{e\}}(HM)^{\Phi H}) \cong M^{\otimes_\mathbb{Z} G/H}$.
\end{remark}

\begin{theorem} \label{thm:main_theorem}
  For $X$ a connective spectrum and $G$ a finite group, we have an
  isomorphism of $G$-Mackey functors
  \[\underline{\pi}_0(N_{\{e\}}^G X) \cong \underline{W}_G(\mathbb{Z}; \pi_0 X)\]
  natural in $X$. More generally if $S$ is a truncation set for a subgroup $H
  \le G$ we have a natural
  isomorphism of $H$-Mackey functors
  \[\underline{\pi}_0(N_{\{e\}}^G X \wedge {\tilde{E} \mathcal{F}(S)}) \cong \underline{W}^S_G(\mathbb{Z}; \pi_0 X) \text{.}\]
  These isomorphisms also respect the monoidal structure.
\end{theorem}
\begin{proof}
  We have monoidal natural isomorphisms of Mackey functors
  \[\underline{\pi}_0(N_{\{e\}}^G X \wedge {\tilde{E} \mathcal{F}(S)}) \cong \underline{\pi}_0(N_{\{e\}}^G(H \pi_0 X) \wedge \tilde{E} \mathcal{F}(S)) \cong \underline{W}^S_G(\mathbb{Z}; \pi_0 X)\]
  where the first isomorphism is Corollary~\ref{cor:replace_with_em} (applied to
  $QX$) and the
  second isomorphism is given by Theorem~\ref{thm:witt_norm_iso}. The untruncated
  version is the special case where $S$ is the set of all subgroups of $G$.
\end{proof}

\appendix

\section{Example computations} \label{app:computation}

Our Witt vectors are fairly amenable to explicit calculation, either directly
from the definition, or by using Proposition~\ref{prop:witt_free} or
Lemma~\ref{lem:witt_ab_group_free_coeff} to compute Witt vectors with
free coefficients and then taking reflexive coequalisers.
Very small cases are reasonable to do by hand, and it would be straightforward
for a computer algebra system to compute somewhat larger cases. To illustrate
this we include some calculations for $G = D_6$, the dihedral
group of order $6$.

Define
\[D_6 = \langle r, s \mid r^3 = s^2 = e, srs^{-1} = r^{-1} \rangle \text{.}\]
Then $D_6$ has elements
\[\{e, r, r^2, s, sr, sr^2\}\]
with subgroups
\[\{e\}, \{e, r, r^2\}, \{e, s\}, \{e, sr\}, \{e, sr^2\}, D_6\]
where the three order $2$ subgroups are all conjugate.

We choose $\langle s
\rangle = \{e, s\}$ to represent the conjugacy class of order 2 subgroups. We take
$\{e, r, r^2, s, sr, sr^2\}$ as coset representatives for $D_6 / \{e\}$, $\{e, r, r^2\}$ as
representatives for $D_6/\langle s \rangle$, $\{e, s\}$ as
representatives for $D_6/\langle r \rangle$, and $\{e\}$ as
the representative for $D_6/D_6$. It will be useful to fix an
ordering on each set of coset representatives, and we use the orderings just given.

\subsection{Direct from definition}

We know that $W_{D_6}(R; M)$ has underlying set a quotient of
\[M^{\otimes_R D_6/D_6} \times
M^{\otimes_R D_6/\langle r \rangle} \times M^{\otimes_R D_6/\langle s \rangle}
\times M^{\otimes_R D_6/\{e\}} \cong M \times M^{\otimes_R 2} \times
M^{\otimes_R 3} \times M^{\otimes_R 6} \text{.}\]
We will switch between the left and right sides of the isomorphism
wherever convenient, using the orderings fixed above.

We can use the ghost maps to gain some understanding of this quotient and
obtain formulae describing the addition operation at the level of
representatives in the above product of tensor powers. Recall the ghost
components are maps $w_U : \prod_{V \in \can{S}} M^{\otimes_R D_6/V} \to
M^{\otimes_R D_6/U}$. We have
\begin{align*}
  w_{D_6}(n) &= n_{D_6}\\
  w_{\langle r \rangle}(n) &= n_{D_6}^{\otimes_R D_6 / \langle r \rangle} + \text{tr}^{D_6}_{\langle r \rangle} n_{\langle r \rangle}\\
  w_{\langle s \rangle}(n) &= n_{D_6}^{\otimes_R D_6 / \langle s \rangle} + n_{\langle s \rangle}\\
  w_{\{e\}}(n) &= n_{D_6}^{\otimes_R D_6/\{e\}} + \text{tr}^{D_6}_{\langle r \rangle} f_{D_6/\langle r \rangle}(n_{\langle r \rangle}^{\otimes_R \langle r \rangle/\{e\}}) + \text{tr}^{D_6}_{\langle s \rangle} f_{D_6/\langle s \rangle}(n_{\langle s \rangle}^{\otimes_R \langle s \rangle/\{e\}}) + \text{tr}^{D_6}_{\{e\}} n_{\{e\}} \text{,}
\end{align*}
where $f_{D_6/\langle r \rangle} : M^{\otimes_R D_6/\langle r \rangle \times
  \langle r \rangle/\{e\}} \to M^{\otimes_R D_6/\{e\}}$ is defined using our choice of coset
representatives for $D_6/\langle r \rangle$, and similarly for $f_{D_6/\langle s \rangle}$.
Note in this case we could write all the formulae neatly using transfers, but this is not
possible in general; for many larger choices of $G$ we would need to write out sums
over $gV \in (G/V)^U$ as in the definition of the ghost map.

For $(T; Q)$ free then $W_{D_6}(T; Q)$ has underlying set $Q \times Q^{\otimes_T 2} \times
Q^{\otimes_T 3} \times Q^{\otimes_T 6}$ modulo identifying any two elements
which have the same images under all the ghost components. Suppose we have $n$ and
$n'$ such that $w_V(n) = w_V(n')$ for all distinguished subgroups $V$. Since
$w_{D_6}(n) = w_{D_6}(n')$ we see that $n_{D_6} = n'_{D_6}$.
From $V = \langle s \rangle$ we get $n_{\langle s \rangle} = n'_{\langle s
  \rangle}$. From $V = \langle r \rangle$ we get
$\text{tr}^{D_6}_{\langle r \rangle} n_{\langle r \rangle} =
\text{tr}^{D_6}_{\langle r \rangle} n'_{\langle r \rangle}$. Since $(T; Q)$ is
free, this is equivalent to saying that $n_{\langle r \rangle}$ and $n'_{\langle r \rangle}$ represent the same element
in the group of orbits $(Q^{\otimes_T 2})_{C_2}$. The condition that
$w_{\{e\}}(n) = w_{\{e\}}(n')$ is harder to interpret in general, since
$\text{tr}^{D_6}_{\langle r \rangle} f_{D_6/\langle r \rangle}(n_{\langle r
  \rangle}^{\otimes_T \langle r \rangle/\{e\}})$ and $\text{tr}^{D_6}_{\langle r
  \rangle} f_{D_6/\langle r \rangle}({n'}_{\langle r \rangle}^{\otimes_T \langle
  r \rangle/\{e\}})$ are not necessarily the same and so we have a more
complicated condition involving both the $\{e\}$ and $\langle r \rangle$ components. We can still make conclusions
in particular cases: for example when $T = Q = \mathbb{Z}$ it is easy to check
that $w$ is already injective, so it induces an isomorphism of sets $\mathbb{Z}^4 \cong \prod_{V \in
  \can{S}} \mathbb{Z}^{\otimes_\mathbb{Z} D_6/V} \cong W_{D_6}(\mathbb{Z}; \mathbb{Z})$.

Let us determine formulae for the addition operation, in terms of
representatives for elements of the quotient. Suppose we have $n, n' \in Q
\times Q^{\otimes_T 2} \times
Q^{\otimes_T 3} \times Q^{\otimes_T 6}$. We want to find an $m$ such that $[n] +
[n'] = [m] \in W_{D_6}(T; Q)$. It suffices to ensure that $w_V(m) = w_V(n) +
w_V(n')$ for each distinguished subgroup $V$. We can consider each $V$ in turn going from larger to smaller, where at
each stage we use the equation for the $V$ component of the ghost map to
determine a value for $m_V$. By the proof of the Dwork lemma this is guaranteed
to work: we will never ``get stuck'' and need to change one of our earlier choices.

Considering $V = D_6$ shows
\[m_{D_6} = n_{D_6} + n'_{D_6} \text{.}\]
From $V = \langle s \rangle$ we get
\begin{align*}
  m_{\langle s \rangle} &= n_{\langle s \rangle} + n'_{\langle s \rangle} +
                          n_{D_6}^{\otimes_T D_6/\langle s \rangle} + n_{D_6}^{\otimes_T D_6/\langle s
                          \rangle} - (n_{D_6} + n'_{D_6})^{\otimes_T D_6/\langle s \rangle}\\
                        &= n_{\langle s \rangle} +  n'_{\langle s \rangle} - \text{tr}_{\langle s
                          \rangle}^{D_6}(n_{D_6} \otimes_T n'_{D_6} \otimes_T n'_{D_6} + n_{D_6} \otimes_T n_{D_6} \otimes_T n'_{D_6})\text{.}
\end{align*}
Looking at $V = \langle r \rangle$ we see that
\begin{align*}
  \text{tr}_{\langle r \rangle}^{D_6} m_{\langle r \rangle} &= \text{tr}_{\langle
    r \rangle}^{D_6} n_{\langle r \rangle} + \text{tr}_{\langle r
    \rangle}^{D_6} n'_{\langle r \rangle} + n_{D_6}^{\otimes_T D_6/\langle r
                                                              \rangle} + {n'}_{D_6}^{\otimes_T D_6/\langle r \rangle} - (n_{D_6}+n'_{D_6})^{\otimes_T D_6/\langle r \rangle}\\
                                                            &= \text{tr}^{D_6}_{\langle r \rangle} (n_{\langle r \rangle} + n'_{\langle r \rangle}) - n_{D_6} \otimes_T n'_{D_6} - n'_{D_6} \otimes_T n_{D_6}\\
  &= \text{tr}^{D_6}_{\langle r \rangle}(n_{\langle r \rangle} + n'_{\langle r \rangle} - n_{D_6} \otimes_T n'_{D_6}) \text{.}
\end{align*}
So we can take
\[m_{\langle r \rangle} = n_{\langle r \rangle} + n'_{\langle r \rangle} -
  n_{D_6} \otimes_T n'_{D_6} \text{.}\]
We can do much the same to determine a formula for $m_{\{e\}}$, but expanding out
sixth tensor powers rapidly becomes tedious and so we will not reproduce it here.

We derived these formulae under the condition that $(T; Q)$ is free; but they
are natural in the choice of coefficients and so by functoriality of $W_{D_6}$
we conclude that they hold for general coefficients $(R; M)$. We can consider these formulae to be a
generalisation of the Witt polynomials describing the ring structure of the
usual Witt vectors of a ring.

\subsection{Explicit computations with free abelian groups}

For $(T; Q)$ free then Proposition~\ref{prop:witt_free} says that we have an
isomorphism of (free) abelian groups
\begin{equation}\label{eq:explicit_free_iso}W_{D_6}(T; Q) \cong \bigoplus_{V \in \can{S}} (Q^{\otimes_T G/V})_{N_{D_6}(V)}
  \text{.}\end{equation}
The ghost group is
$\text{gh}_{D_6}(\mathbb{Z}; \mathbb{Z}) = \left( \bigoplus_{U \in S} Q^{\otimes_T
    G/U} \right)^H \cong \bigoplus_{V \in \can{S}} (Q^{\otimes_T G/V})^{N_{D_6}(V)}$, and
under the isomorphism (\ref{eq:explicit_free_iso}) the ghost map has formula
\[w_U(n) = \sum_{V \in \can{S}} \sum_{gV \in (D_6/V)^U} g \cdot \phi^V_{U^g}(n_V) \text{.} \]
This lets us find generators of $W_{D_6}(T; Q)$ as a free abelian subgroup of
$\text{gh}_{D_6}(T; Q)$. We will apply this to compute $W_{D_6}(\mathbb{Z}; \mathbb{Z}/3\mathbb{Z})$.
Larger choices of coefficients are painful to compute with by hand, but could certainly be done with computer assistance.

\begin{lemma}
  We have
  \[W_{D_6}(\mathbb{Z}; \mathbb{Z}/3\mathbb{Z}) \cong (\mathbb{Z}/3\mathbb{Z})^2
    \times \mathbb{Z}/9\mathbb{Z} \text{.}\]
\end{lemma}
\begin{proof}
  First we consider $T = Q = \mathbb{Z}$. By Proposition~\ref{prop:witt_free}, $W_{D_6}(\mathbb{Z};
  \mathbb{Z}) \cong \bigoplus_{V \in \can{S}} \mathbb{Z}$. The ghost group is
  $\bigoplus_{V \in \can{S}} \mathbb{Z}$, and the ghost map has formula
  \[w_U(n) = \sum_{V \in \can{S}} \abs{(D_6/V)^U} n_V \text{.} \]
  Ordering the subgroups from largest to smallest, we find that
  $W_{D_6}(\mathbb{Z}; \mathbb{Z})$
  considered as a subgroup of $\text{gh}_{D_6}(\mathbb{Z}; \mathbb{Z}) \cong \mathbb{Z}^4$ has basis
  \begin{equation} \label{eq:free_basis_list}\{(1, 1, 1, 1), (0, 2, 0, 2), (0, 0, 1, 3), (0, 0, 0, 6)\} \subset \mathbb{Z}^4 \text{.}\end{equation}

  The point of this is to give us a concrete setting in which to
  compute $W_{D_6}(\mathbb{Z}; \mathbb{Z}/3)$ as a quotient of
  $W_{D_6}(\mathbb{Z}; \mathbb{Z})$. We can write
  $\mathbb{Z}/3\mathbb{Z}$ as a reflexive coequaliser of free abelian groups:
  \[\mathbb{Z}^2 \, \substack{\xrightarrow{p}\\[-0.2em]
      \xleftarrow{}\\[-0.2em] \xrightarrow[q]{}} \, \mathbb{Z} \to
    \mathbb{Z}/3\mathbb{Z}\]
  where $p(a, b) = a + 3b$ and $q(a, b) = a$, with common section $s(a) = (a, 0)$. So we have a reflexive coequaliser diagram
  \[W_{D_6}(\mathbb{Z}; \mathbb{Z}^2) \, \substack{\xrightarrow{p_\ast}\\[-0.2em]
      \xleftarrow{\makebox[1.5ex]{}}\\[-0.2em] \xrightarrow[q_\ast]{}} \, W_{D_6}(\mathbb{Z}; \mathbb{Z}) \to
    W_{D_6}(\mathbb{Z}; \mathbb{Z}/3\mathbb{Z}) \text{.}\]
  Our next step is to take generators for $W_{D_6}(\mathbb{Z}; \mathbb{Z}^2)$ and
  compute their images under the ghost map followed by $p_\ast$ and $q_\ast$. This
  will give us the relations defining $W_{D_6}(\mathbb{Z};
  \mathbb{Z}/3\mathbb{Z})$ as a quotient of $W_{D_6}(\mathbb{Z}; \mathbb{Z})$.

  Write $\mathbb{Z}^2$ as $\mathbb{Z}(\alpha, \beta)$, the free abelian group on
  two elements. Then $W_{D_6}(\mathbb{Z}; \mathbb{Z}(\alpha, \beta))$ is generated by elements
  of the form
  \begin{align*}\bigotimes_{gW \in D_6/W} y_{gW} \in (\mathbb{Z}(\alpha,
    \beta)^{\otimes_\mathbb{Z} D_6/W})_{N_{D_6}(W)} &\le \bigoplus_{V \in \can{S}} (\mathbb{Z}(\alpha,
                                                      \beta)^{\otimes_\mathbb{Z} D_6/V})_{N_{D_6}(V)}\\
    &\cong W_{D_6}(\mathbb{Z};
    \mathbb{Z}(\alpha, \beta))\end{align*}
  where $W \in \can{S}$ and $y_{gW} \in \{\alpha, \beta\}$ for $gW \in G/W$.
  Denote this element by $y$. The image of such an element under the $V$-component of the ghost map is given by
  \[w_V(y) = \sum_{gW \in (G/W)^V} g \cdot \left( \bigotimes_{hV^g \in G/V^g} y_{hW}\right) \in
    \mathbb{Z}(\alpha, \beta)^{\otimes_\mathbb{Z} D_6/V} \text{.} \]
  If $V$ is not subconjugate to $W$ then this is zero. Otherwise suppose $y_{gW}$ is $\alpha$ for $u$ values of $gW$ and $\beta$ for $v =
  \abs{G/W} - u$ values, and let $r = \abs{G:V}/\abs{G:W}$. Then after applying $p_\ast$ and $q_\ast$ to the above
  we get
  $p_\ast(w_V(y)) = \abs{(G/W)^V} 1^{ur}
  3^{vr}$ and $q_\ast(w_V(y)) = \abs{(G/W)^V} 1^{ur} 0^{vr}$ respectively.
  
  For example, consider the element
  \[y = \alpha \otimes \beta \in
  (\mathbb{Z}^{\otimes_\mathbb{Z} D_6/\langle r \rangle})_{N_{D_6}(\langle r \rangle)} \le W_{D_6}(\mathbb{Z};
  \mathbb{Z}(\alpha, \beta))\text{.}\]
We see that $w_{D_6}(y)$ and
  $w_{\langle s \rangle}(y)$ are zero. Meanwhile $w_{\langle r
    \rangle}(y) = \alpha \otimes \beta + \beta \otimes \alpha$
  and $w_{\{e\}}(y) = \text{tr}_{\{e\}}^{D_6}(\alpha^{\otimes
    3} \otimes \beta^{\otimes 3})$. And we have
  \[p_\ast (w(y)) = (0, 2 \cdot 3, 0, 2 \cdot 3^3) \text{,} \quad q_\ast(w(y)) = (0, 0, 0, 0) \text{.}\]

  We do similar calculations for all the other generators of
  $W_{D_6}(\mathbb{Z}; \mathbb{Z}(\alpha, \beta))$.
  Observe that
  $p_\ast(w(y))$ and $q_\ast(w(y))$ only depend on the numbers $u$ and $v$ of occurrences of $\alpha$ and
  $\beta$ in $y$, so many different generators give the same results.
  We get
  the following table:

  \[\begin{array}{c c c c c}
      W & u & v & p_\ast(w(y)) & q_\ast(w(y))\\
      \hline
      D_6&1&0&\begin{pmatrix}1&1&1&1\end{pmatrix}&\begin{pmatrix}1&1&1&1\end{pmatrix}\\
      D_6&0&1&\begin{pmatrix}3&3^2&3^3&3^6\end{pmatrix}&\begin{pmatrix}0&0&0&0\end{pmatrix}\\
      \langle r \rangle&2&0&\begin{pmatrix}0&2&0&2\end{pmatrix}&\begin{pmatrix}0&2&0&2\end{pmatrix}\\
      \langle r \rangle&1&1&\begin{pmatrix}0&2\cdot3&0&2\cdot3^3\end{pmatrix}&\begin{pmatrix}0&0&0&0\end{pmatrix}\\
      \langle r \rangle&0&2&\begin{pmatrix}0&2\cdot3^2&0&2\cdot3^6\end{pmatrix}&\begin{pmatrix}0&0&0&0\end{pmatrix}\\
      \langle s \rangle&3&0&\begin{pmatrix}0&0&1&3\end{pmatrix}&\begin{pmatrix}0&0&1&3\end{pmatrix}\\
      \langle s \rangle&2&1&\begin{pmatrix}0&0&3&3 \cdot 3^2\end{pmatrix}&\begin{pmatrix}0&0&0&0\end{pmatrix}\\
      \langle s \rangle&1&2&\begin{pmatrix}0&0&3^2&3 \cdot 3^4\end{pmatrix}&\begin{pmatrix}0&0&0&0\end{pmatrix}\\
      \langle s \rangle&0&3&\begin{pmatrix}0&0&3^3&3 \cdot 3^6\end{pmatrix}&\begin{pmatrix}0&0&0&0\end{pmatrix}\\
      \{e\}&6&0&\begin{pmatrix}0&0&0&6\end{pmatrix}&\begin{pmatrix}0&0&0&6\end{pmatrix}\\
      \{e\}&5&1&\begin{pmatrix}0&0&0&6\cdot 3\end{pmatrix}&\begin{pmatrix}0&0&0&0\end{pmatrix}\\
      \{e\}&4&2&\begin{pmatrix}0&0&0&6\cdot 3^2\end{pmatrix}&\begin{pmatrix}0&0&0&0\end{pmatrix}\\
      \{e\}&3&3&\begin{pmatrix}0&0&0&6\cdot 3^3\end{pmatrix}&\begin{pmatrix}0&0&0&0\end{pmatrix}\\
      \{e\}&2&4&\begin{pmatrix}0&0&0&6\cdot 3^4\end{pmatrix}&\begin{pmatrix}0&0&0&0\end{pmatrix}\\
      \{e\}&1&5&\begin{pmatrix}0&0&0&6\cdot 3^5\end{pmatrix}&\begin{pmatrix}0&0&0&0\end{pmatrix}\\
      \{e\}&0&6&\begin{pmatrix}0&0&0&6\cdot 3^6\end{pmatrix}&\begin{pmatrix}0&0&0&0\end{pmatrix}\\
    \end{array}\]

  So to compute $W_{D_6}(\mathbb{Z}; \mathbb{Z}/3\mathbb{Z})$ we take
  $W_{D_6}(\mathbb{Z}; \mathbb{Z})$ considered as the subgroup of $\mathbb{Z}^4$
  with basis (\ref{eq:free_basis_list}),
  then quotient by the subgroup generated by the elements $p_\ast(w(y)) - q_\ast(w(y))$ for
  each row of the table. After some straightforward algebra we compute
  \[W_{D_6}(\mathbb{Z}; \mathbb{Z}/3\mathbb{Z}) \cong (\mathbb{Z}/3\mathbb{Z})^2
    \oplus \mathbb{Z}/9\mathbb{Z} \text{.}\]
\end{proof}

To double check this, we can show that
$W_{D_6}(\mathbb{Z}; \mathbb{Z}/3\mathbb{Z})$ has order $81$ independently of
the above calculation.
Recall that the underlying set of $W_{D_6}(R; M)$ is a quotient of $\bigoplus_{V \in \can{S}} M^{\otimes_R
D_6/V}$. Since products and tensor powers preserve reflexive coequalisers, we have
a map of reflexive coequalisers
\[\begin{tikzcd}
  \bigoplus_{V \in \can{S}} (\mathbb{Z}^2)^{\otimes_\mathbb{Z} D_6/V} \ar[r, yshift=0.5ex]
  \ar[r, yshift=-0.5ex] \ar[d, twoheadrightarrow] &
  \bigoplus_{V \in \can{S}} \mathbb{Z}^{\otimes_\mathbb{Z} D_6/V} \ar[r,
  twoheadrightarrow] \ar[d, twoheadrightarrow] &
  \bigoplus_{V \in \can{S}} (\mathbb{Z}/3\mathbb{Z})^{\otimes_\mathbb{Z} D_6/V} \ar[d, twoheadrightarrow]\\
  W_{D_6}(\mathbb{Z}; \mathbb{Z}^2) \ar[r, yshift=0.5ex] \ar[r, yshift=-0.5ex] & W_{D_6}(\mathbb{Z}; \mathbb{Z})
  \ar[r, twoheadrightarrow] & W_{D_6}(\mathbb{Z}; \mathbb{Z}/3\mathbb{Z}) \text{.}
\end{tikzcd}\]
But as we remarked earlier the middle
vertical map is an isomorphism of sets, since the ghost map $w : \bigoplus_{V \in \can{S}}
\mathbb{Z}^{\otimes_\mathbb{Z} D_6/V} \cong \mathbb{Z}^4 \to
\text{gh}_{D_6}(\mathbb{Z}; \mathbb{Z})$ is already injective.
By an easy diagram chase we conclude that the
right vertical map is also an isomorphism.
That is, $W_{D_6}(\mathbb{Z};
\mathbb{Z}/3\mathbb{Z}) \cong \bigoplus_{V \in \can{S}}
(\mathbb{Z}/3\mathbb{Z})^{\otimes_\mathbb{Z} D_6/V}\cong (\mathbb{Z}/3\mathbb{Z})^4$ as sets and has $81$
elements.
\printbibliography

\end{document}